\documentclass[11pt,reqno,letter]{amsart}

\usepackage{mathtools, amsthm, verbatim, amscd, amssymb, setspace, enumitem, hyperref, exscale, color, slashed, arydshln, accents, scalerel,amsmath}
\usepackage{esint}
\usepackage[all,pdf]{xy}
\usepackage{stackengine}
\usepackage{esint}

\newcommand{\CC}{\mathbb{C}}
\newcommand{\DD}{\mathbb{D}}
\newcommand{\RR}{\mathbb{R}}
\newcommand{\NN}{\mathbb{N}}
\newcommand{\OO}{\mathcal{O}}

\newcommand{\TT}{\mathbb{T}}
\newcommand{\im}{\operatorname{Im}}
\newcommand{\psh}{\operatorname{Psh}}
\newcommand{\C}{\mathcal{C}}
\newcommand{\Pl}{\mathcal{P}}
\newcommand{\F}{\mathcal{F}}
\newcommand{\HH}{\mathcal{H}}

\newcommand{\LL}{\mathcal{L}}

\newcommand{\re}{\operatorname{Re}}
\newcommand{\bff}{\mathfrak{b}}

\newcommand{\ZZ}{\mathbb{Z}}

\newcommand{\diam}{\operatorname{diam}}
\newcommand{\tr}{\operatorname{tr}}

\newcommand{\dist}{\operatorname{dist}}
\newcommand{\A}{\mathcal{A}}
\newcommand{\B}{\mathcal{B}}
\newcommand{\D}{\mathcal{D}}
\newcommand{\BMO}{\mathrm{BMO}}

\newcommand{\id}{\operatorname{id}}
\newcommand{\Hilb}{\operatorname{Hil}}
\newcommand{\ssubset}{\subset\joinrel\subset}

\newtheorem{thm}{Theorem}[section]
\newtheorem{pro}[thm]{Proposition}
\newtheorem{cor}[thm]{Corollary}
\newtheorem{lem}[thm]{Lemma}

\newcounter{mtheorem}
\newtheorem{mtheorem}[mtheorem]{Theorem}

\theoremstyle{definition}
\newtheorem{define}[thm]{Definition}
\newtheorem{rem}[thm]{Remark}

\newtheorem*{ac}{Acknowledgements}

\numberwithin{equation}{section}

\title[Asymptotic Behavior of HCMA Equations on ALE K\"ahler manifolds]%
      {Asymptotic Behavior of Homogeneous Complex Monge-Amp\`ere Equations on ALE K\"ahler manifolds}

\author{Qi Yao}
\date{}

\begin{document}
 
\maketitle
\thispagestyle{empty}

\begin{abstract}
This paper is a sequel to the author’s earlier work~\cite{yao2024geodesicequationsasymptoticallylocally}, and investigates the homogeneous complex Monge–Ampère equation (HCMA) on the product space \( X \times D \), where \( X \) is an asymptotically locally Euclidean (ALE) K\"ahler manifold and \( D \subset \mathbb{C} \) is the unit disc. We establish precise asymptotic behavior of the solution to the HCMA equation, showing that the decay rate of the solution matches that of the prescribed boundary data, and that uniform control in weighted Hölder norms can be achieved.

The analysis combines two main ingredients: a redevelopment of pluripotential theory on the noncompact space \( X \times D \), and a PDE-based construction of holomorphic disc foliations on the end of \( X \), inspired by the works of Semmes and Donaldson. As an application in the general K\"ahler manifolds, the techniques developed in this paper also imply a local regularity result for the HCMA equation.
\end{abstract}

\section{introduction}

Initiated by the works of Mabuchi \cite{mabuchi1987some}, Semmes \cite{semmes1992complex}, and Donaldson \cite{donaldson1999symmetric}, the HCMA equation has emerged as a central object in the study of canonical K\"ahler metrics, particularly in relation to the constant scalar curvature K\"ahler (cscK) problem on compact K\"ahler manifolds.
Fixing a reference K\"ahler form $\omega$ on a compact K\"ahler manifold, the space of all K\"ahler metrics in $[\omega]$ can be identified with the space of K\"ahler potential,
\begin{align*}
\HH(X, \omega) = \{\varphi \in \C^\infty (X); \ \omega + i \partial \overline{\partial} \varphi > 0\},
\end{align*}
modulo constants. Given two potential functions in $\HH(X, \omega)$, $\psi_0$, $\psi_1$,
we connect $\psi_0$ and $\psi_1$ with a path $\varphi(t): [0,1] \rightarrow \HH(X,\omega)$ with $\varphi(0) = \psi_0$, $\varphi(1) = \psi_1$. We say $\varphi(t)$ is a geodesic in $\HH(X,\omega)$, if 
\begin{align} \label{geoequintro}
    \ddot{\varphi} - \frac{1}{2} |\nabla_{\omega_{\varphi(t)}} \dot{\varphi}|_{\omega_{\varphi(t)}}^2 = 0.
\end{align}
The above equation \eqref{geoequintro} is called the geodesic equation in the space of K\"ahler potentials.
As observed by Donaldson \cite{donaldson1999symmetric} and Semmes \cite{semmes1992complex}, the geodesic equation is equivalent to the HCMA equation in the product space $X \times \Sigma$, where $\Sigma \cong [0,1]\times S^1$ can be embedded as an annulus in $\CC$. 
Notice that any path in $\HH(X,\omega)$ can viewed as a function defined in $X \times \Sigma$, $\Phi (\cdot, t, e^{is}) = \varphi(t)$. Let $\Omega =\pi^* \omega $, where $\pi$ is the natural projection from $X\times \Sigma$ to $X$. Then the equation (\ref{geoequintro}) can be rewritten on $X \times \Sigma$ as follows,
\begin{align*}
    (&\Omega +i \partial \overline{\partial} \Phi)^{n+1} = 0, \\
    &\Omega +i \partial \overline{\partial} \Phi \geq 0, \\
    &\Phi|_{X \times \partial\Sigma } =  \psi_{0,1}. 
\end{align*}

In the cases of compact K\"ahler manifolds, Chen \cite{chen2000space} showed that for any $\psi_0$, $\psi_1 \in \HH$, the geodesic equation has a unique solution up to $dd^c$-regularity. Blocki \cite{blocki2009gradient} and He \cite{he2015space} built up direct calculations to prove the gradient estimate. The full $\C^{1,1}$ estimate was proved by Chu-Tossati-Weinkove in \cite{chu2017regularity}. In the other direction, Lempert-Vivas \cite{lempert2013geodesics} and Darvas-Lempert \cite{darvas2012weak} constructed counterexamples to assert that $\partial \overline{\partial} \varphi$ is not continuous in general, hence the $\C^{1,1}$ regularity is optimal in general. 

While the main focus of the paper is to deal with the HCMA equation on ALE K\"ahler manifolds, the techniques developed here also yield consequences in more general settings. In particular, the techniques can be adapted to prove a local regularity result for the HCMA equation on arbitrary K\"ahler manifolds. If the boundary data is sufficiently small in suitable Hölder norms on a local coordinate chart, then the global weak solution becomes smooth in the corresponding region of the product space. Although this result is not among the main contributions of the paper, we present it first, as it provides a simple but interesting example of the techniques that we later develop for the ALE cases.

\begin{mtheorem}[see Theorem \ref{thmlocalreg}]\label{mthmlocalreg}
 Let $(X, \omega)$ be a general K\"ahler manifold, and let $D$ be a unit disc in $\CC$.
Let $\Phi$ be a bounded (weak) solution to the HCMA equation \eqref{introHCMA} on $X \times D$, given by the upper envelope of some class of $\Omega$-psh subsolutions. Suppose the boundary data $\Psi$ is sufficiently small in a weighted Hölder norm on $N' \times \partial D$, with $N'$, a local chart $N' \subseteq X$. Then, for a smaller domain $N\ssubset N'$, the solution $\Phi$ is smooth on $N \times D$.
\end{mtheorem}

A quick implication of Theorem \ref{mthmlocalreg} is the global regularity of the HCMA equation under the assumption that $\Psi$ is small globally in $X \times \partial D$, which recovers the results from \cite{donaldson2002holo, CHEN2020108603, hu2024metric} when $D$ is a disk.

We now introduce the basic setup for ALE Kähler manifolds. 
Let $(X, J, g)$ be a complete non-compact K\"ahler manifold of complex dimension $n\geq 2$. We say $(X, J, g)$ is asymptotically locally Euclidean (ALE) K\"ahler if there is a compact subset $K \subseteq X$ such that $ \psi: X-K \rightarrow (\CC^n - B_R)/ \Gamma$ is a diffeomorphism, where $B_R$ is a closed ball in $\RR^n$ with radius $R$, and $\Gamma$ is a finite subset of $U(n)$ acting freely on the unit sphere. (Note that any ALE K\"ahler manifold has only one end, as proved in Hein-LeBrun \cite[Prop 1.5, Prop 3.2]{hein2016mass}.) Let $r$ be a function defined on $X- K$ as the pull-back of the Euclidean radius function in $(\CC^n - B_R)$ via $\psi$. The metric $g$ is said to be asymptotic to the Euclidean metric at infinity with decay rate $-\mu < 1-n$, in the sense that:
\begin{itemize}
    \item In the asymptotic complex coordinates provided by $\psi$, the components of $g$ satisfy
    \begin{align}\label{decayaleintro}
        g_{ij} = \delta_{ij} + O(r^{-\mu}), \qquad |\nabla^k ((\psi^{-1})^*g)|_{g_0} = O(r^{-\mu - k}),
    \end{align}
    where $g_0$ is the standard Euclidean metric on $\mathbb{C}^n$, and $\nabla$ denotes its Levi-Civita connection.
\end{itemize}

Suppose that $X$ is an ALE K\"ahler manifolds with complex dimension $n \geq 2$.  
Let $D$ be a bounded domain in $\CC$ with at least $\C^1$ boundaries. The choice of \( D \) may vary depending on the specific context—for instance, a unit disc, an annulus, or a general domain. We denote by \( \tau \) the standard complex coordinate on \( D \).
In this paper, we study the homogeneous complex Monge–Amp\`{e}re (HCMA) equation on the product space $X \times D$, with boundary data prescribed along $X \times \partial D$. Let $ \omega(\cdot, \cdot) = g(J\cdot, \cdot)$ be the reference K\"ahler form on $X$, and let $\Omega = \pi^* \omega$ be the pullback of $\omega$ to $X\times D$, where $\pi: X \times D \rightarrow X$ is the natural projection. We denote by $\partial$ and $\overline{\partial}$ the complex differential operators on the product space $X \times D$. We consider the following Dirichlet type problem for the HCMA equation:
\begin{align}\label{introHCMA}
    \begin{cases}
        (\Omega + i\partial \overline{\partial} \Phi)^{n+1} = 0, \hspace{1cm} \text{ in } X \times D;\\
        \Omega + i\partial \overline{\partial} \Phi \geq 0, \hspace{1.9cm} \text{ in } X \times D;\\
        \Phi(x, \tau) = \Psi(x, \tau), \hspace{1.35cm} \text{ on } X \times \partial D.
    \end{cases}
\end{align}
The focus of this work is to understand the solvability, regularity, and asymptotic behavior of solutions to this equation under suitable conditions on the boundary data along $X \times \partial D$. Here, we shall specify the regularity and asymptotic conditions of $\Psi$ on $X \times \partial D$. 

Consider the space of K\"ahler potential with prescribed decay rates:
\begin{align} \label{ALEKpotentialintro}
    \mathcal{H}_{-\gamma} (\omega) = \{\varphi \in \C^\infty_{ -\gamma}(X), \ \omega + i\partial \overline{\partial} \varphi\geq 0\},
\end{align}
where we assume $-\gamma < \min\{2-\mu, 0\}$,a and the space $\C^\infty_{ -\gamma} (X)$ is defined to be 
\begin{align*}
    \C^{\infty}_{-\gamma} (X) = \{\psi\in \C^{\infty} (X), \ |\nabla^{i} \psi|_{g_0} = O(r^{ -\gamma -i}) \text{ on } X-K, \ \text{for } i = 0,1,\ldots    \}.
\end{align*}
Let $v$ be the unit vector in the direction of $\partial D$, and let $D_v$ be the differential operator of the directional derivative by the vector $v$. We define the class of functions with prescribed decay rates for mixed derivatives:
\begin{align*}
\begin{split}
    \C^{\infty}_{-\gamma}(X \times \partial D)= \{\Psi \in &\C^\infty(X \times \partial D), 
    \\ 
    &|\nabla^i (\mathrm{D}_v^k \Psi) |_{g_0} = O(r^{-\gamma-i}) \text { on } X -K, \text{ for } i,k = 0,1,\ldots \}.
\end{split}
\end{align*}
Now, we assume the prescribed boundary function in the HCMA equation \eqref{introHCMA}, $\Psi$, satisfies 
\begin{itemize}
    \item $\Psi \in  \C^{\infty}_{-\gamma}(X \times \partial D)$;
    \item $\psi_\tau (\cdot) = \Psi (\cdot, \tau) \in \mathcal{H}_{-\gamma} (\omega)$, for $\tau \in \partial D$.
\end{itemize}

Now, let $(X, g, J)$ be an ALE K\"ahler manifold satisfying the metric decay condition~\eqref{decayaleintro}. Given the asymptotic behavior of ALE K\"ahler metrics, it is natural to impose corresponding decay conditions on the potential functions. In the context of scalar-flat K\"ahler metrics, \cite[Theorem D]{yao2022mass} shows that the natural decay rates for the potential functions are typically taken to be $\mu = n - 1 + \varepsilon$ and $\gamma = 2n - 4 + \varepsilon'$, for some parameters, $\varepsilon, \varepsilon' > 0$.

In solving the HCMA equations in \eqref{introHCMA}, we do not impose precise assumptions on the decay rates \(-\mu\) and \(-\gamma\), apart from requiring that \(-\gamma \leq \min\{2-\mu, 0\}\). 

Under the setting of ALE K\"ahler manifolds, the global $\C^{1,1}$ regularity of HCMA on ALE K\"ahler manifolds was proved in \cite{yao2024geodesicequationsasymptoticallylocally}. That is, given a family of potential functions $\psi_\tau \in \HH_{-\gamma} (X, \omega)$ for $\tau \in \partial D$. Let $\Phi$ be the solution to the HCMA equation \eqref{introHCMA}. Then, there is a uniform constant $C$ depending only on $n$, $ ||\Psi ||_{\C^{1,1} (X \times \partial D)}$ and the geometry of $(X, \omega)$ such that
\begin{align} \label{introglobalC11esti}
||\Phi ||_{\C^{1,1}(X \times D)} = \sup_{X \times D}\big( |\Phi| + |\widetilde{\nabla} \Phi|_{\tilde{g}} + |\widetilde{\nabla}^2 \Phi|_{\tilde{g}}  \big)  \leq C
\end{align}
where $\tilde{g} = g+ |d\tau|^2$ is a K\"ahler metric defined on $X \times D$ and $\widetilde{\nabla}$ is the Levi-Civita connection of $\tilde{g}$. In particular, if we pick $D$ to be the annulus $\Sigma$ in $\CC$, the estimate \eqref{introglobalC11esti} implies the $\C^{1,1}$-regularity of weak geodesics on ALE K\"ahler manifolds. 

In this paper, we aim to discuss the asymptotic behavior of the solution $\Phi$ to the HCMA equations in the special case where $D$ is a disk in $\CC$. The main result is the following:

\begin{mtheorem} \label{mthmasymbehHCMA}
    Let $(X, J, g)$ be an ALE K\"ahler manifold, and let $D$ be the unit disk in $\CC$. Assume that $\Psi$ is a function on $X \times \partial D$ satisfying $\Psi \in \C^\infty_{-\gamma} (X \times \partial D)$ and $\psi_\tau = \Psi(\cdot, \tau) \in\HH_{-\gamma} (X, \omega)$. Then, there is a unique bounded solution $\Phi$ to the HCMA equation \eqref{introHCMA} satisfying the $\C^{1,1}$ estimates \eqref{introglobalC11esti}. For a sufficiently large uniform constant $l$ depending only on $n$, $\Psi$ and the geometry of $(X, J, g)$, the solution $\Phi$ satisfies the weighted estimates:
    \begin{align} \label{introasymbehavior}
        \sup_{X_l \times D}| r^{\gamma + k} \nabla^k \mathrm{D}^m \Phi |_{g_0} \leq C_{k,m}, \qquad \text{on } X_l = \{x \in X, \ r(x)> l\},
    \end{align}
    where $\nabla$ is the Levi-Civita connection of the Euclidean metric in asymptotic coordinates $(X_\infty, g_0)$, $\mathrm{D}$ is the standard derivative in $D\subseteq \CC$, and $C_{k,m}$ is a uniform constant depending on $k$, $m$, $n$, $\Psi$ and the geometry of $(X, J, g)$.  Furthermore, the form $\Omega + i \partial \overline{\partial} \Phi$ is nondegenerate in the direction of $X$ on $X_l \times D$ such that the following holds:
\begin{align*}
   \frac{1}{C} \omega  \leq \omega + i \partial \overline{\partial} \varphi_\tau \leq C \omega, \qquad \text{ on } X_l , \quad \text{ for each }\tau \in D.
\end{align*}
where $C$ is a uniform constant very close to $1$.
\end{mtheorem}

A function $u$ is called $\Omega$-plurisubharmonic ($\Omega$-psh) on $X\times D$ if $u$ is upper-semicontinous and $\Omega + i\partial \overline{\partial} u \geq 0$ in the sense of current.
The proof of Theorem~\ref{mthmasymbehHCMA} is constructive in nature. The main goal is to construct a global \(\Omega\)-psh subsolution \( F \) to the HCMA equation such that \( F \) coincides with the actual solution \( \Phi \) of the HCMA equation \eqref{introHCMA} on the region \( X_l \times D \), and in addition, \( F \) satisfies the desired asymptotic behavior prescribed by \eqref{introasymbehavior}.
The construction can be broadly divided into the following three steps:
\begin{itemize}
    \item Step 1: Build up the pluripotential theory for bounded $\Omega$-psh functions on $X \times D$.
    \item Step 2: Establish the existence and uniqueness of the holomorphic disk foliation on $X_{l_0} \times D$ under a small perturbation of boundary data.
    \item Step 3: Construct a global $\Omega$-psh function $F$ on $X\times D$ such that $F\in \C^\infty_{-\gamma}(X_{l_0} \times D)$ and $F$ agrees with the solution $\Phi$ to HCMA equations on $X_{l} \times D$, where $l\gg l_0$. The construction of $F$ highly depends on the structure of the holomorphic discs foliation in step 2.
\end{itemize}

The Bedford-Taylor's pluripotential theory for bounded psh functions can be generalized to $X \times D$. The main objective of Step 1 is to prove that the solution to the HCMA equation is given by the upper envelope of a certain class of $\Omega$-psh functions. Let $\psh_\Omega(X \times D)$ denote the set of all $\Omega$-psh functions on $X \times D$. The class of $\Omega$-psh functions associated with the boundary function $\Psi$ on $X \times \partial D$ is defined as follows:
\begin{align*}
    \B_{\Omega, \Psi} = \{u \in &\psh_{\Omega}(X \times D);\  u \text{ is bounded, and } 
    \\
    &\limsup_{(x',\tau')\rightarrow (x, \tau)} u(x', \tau') \leq \Psi (x, \tau), \text{ for } (x', \tau') \in X \times  D, \ (x, \tau) \in X \times \partial D   \}.
\end{align*}
The upper envelope of $\B_{\Omega, \Psi}$, $\Phi_{\Omega, \Psi}$, is defined to be
$\displaystyle \Phi_{\Omega, \Psi} (x, \tau) = \sup_{u \in \B_{\Omega, \Psi}} u(x, \tau)$. Before discussing the relation between the global $\C^{1,1}$ solution $\Phi$ to the HCMA equation and upper envelope $\Phi_{\Omega, \Psi}$, we introduce a version of the maximal principle for complex Monge-Amp\`{e}re operator in the class of bounded continuous function on $X \times D$. 

\begin{mtheorem}[see Theorem \ref{mainthmmaxp}]\label{intromainmp}
    Let $X$ be an ALE K\"ahler manifold, and $D$, a bounded domain in $\CC$ with $\C^1 $ boundaries. Let $u$, $v$ be bounded continuous $\Omega$-psh functions satisfying $ (\Omega + i\partial \overline{\partial} u )^{n+1} \leq (\Omega + i\partial \overline{\partial} v)^{n+1} $ and $(\Omega + i\partial \overline{\partial} u) \leq E_0 \Theta^{n+1} $ on $X \times D$, where $\Theta = \Omega + i d\tau \wedge d\bar{\tau}$ and $E_0 >0$ is a constant. Moreover, we assume $u \leq v$ on $X \times \partial D$.
    Then, we have $u \leq v$ on $X \times D$.
\end{mtheorem}

Theorem \ref{intromainmp} indicates the uniqueness of bounded continuous solutions to HCMA equations. Then we have the following theorem:

\begin{mtheorem}[see Theorem \ref{thmc11soluasupenvelope}]\label{mthmHCMAc11}
    Suppose that $\Psi$ is a {bounded uniformly continuous function} defined in $X \times \partial D$, and $\psi_\tau(\cdot) = \Psi(\cdot, \tau)$ is $\omega$-psh function for $\tau \in \partial D$. There exists a {unique} {bounded continuous solution}, $\Phi$, to 
    \begin{align} \label{introHCMA}
    \begin{aligned}
        &\big(\Omega + dd^c \Phi\big)^{n+1} = 0, & \text{in } X\times D;\\
        &\Omega + dd^c \Phi \geq 0, & \text{in } X \times D;\\
        & \Phi = \Psi, & \text{in }  X \times \partial{D},
    \end{aligned}
    \end{align}
    and $\Phi = \Phi_{\Omega, \Psi} $, the {upper envelope of $\B_{\Omega, \Psi}$}.
    Furthermore, If $\Psi$ is bounded, smooth and $\psi_\tau= \Psi(\cdot, \tau) \in \HH(X, \omega)$, then the solution, $\Phi$, is the $\C^{1,1}$ solution satisfies \eqref{introglobalC11esti}.
\end{mtheorem}

From Step 2, we should assume $D$ to be a unit disk in $\CC$. 
The primary tool for constructing the global $\Omega$-psh function $F$ is the holomorphic disk foliation on $X_{l_0} \times D$. 
A short description of holomorphic disc foliation is provided in the following paragraphs; the readers can refer to Section \ref{secholodiscHCMA} (local description) and Section \ref{subsecpatthm} (global version by gluing) for more details.

The holomorphic disc foliation on $X_{l_0} \times D$ is based on the fundamental work by Semmes \cite{semmes1992complex}. Let $X_{l_0}$ be the subset of the end $X_\infty$ defined in \eqref{introasymbehavior}. In~\cite{semmes1992complex}, Semmes constructed a complex manifold \( \mathcal{W} \), associated to the data \( (X_{l_0}, \omega) \), by gluing together the holomorphic cotangent bundles over an open covering of holomorphic charts \( \{U_i\}_{i \in \mathcal{I}} \) of \( X_{l_0} \). In each chart \( U_i \), let \( \rho_i \) be a local K\"ahler potential of \( \omega \). On overlaps \( U_i \cap U_j \), the transition functions are given by fiberwise translations using \( \partial(\rho_i - \rho_j) \). This gluing yields a holomorphic fiber bundle \( \mathcal{W} \) over \( X \), with a natural projection \( p: \mathcal{W} \to X \).
Given \( \psi \in \HH_{-\gamma}(\omega) \), the differential \( \partial(\rho + \psi) \) defines an exact Lagrangian submanifold in \( \mathcal{W} \). For each \( \tau \in \partial D \), define \( \Lambda_\tau = \partial(\rho + \psi_\tau) \subseteq \mathcal{W} \) as an exact Lagrangian submanifold.
A family of holomorphic discs over \( X_{l_0} \) is a smooth map \( G: X_{l_0} \times D \to \mathcal{W} \), holomorphic in the \( D \)-direction, such that \( G(x, \tau) \in \Lambda_\tau \) for \( (x, \tau) \in X_{l_0} \times \partial D \), and satisfying the normalization condition \( \pi \circ G(x, -i) = x \). The projection \( H = \pi \circ G \) then defines a holomorphic disc foliation of \( X_{l_0} \times D \).

The following Theorem discusses the existence and local uniqueness of the family of holomorphic discs under a small perturbation of a trivial foliation. Denote by $\rho$ the K\"ahler potential of $\omega$ on each horizontal chart of $X_{l_0}$ for $\tau \in \partial D$, there exists a trivial family of holomorphic discs $G_0 (w, \tau)$ 
such that in each local holomorphic chart $U$ with K\"ahler potential $\rho$, $ G_0 (x, \tau )= (x, \partial \rho)$. Now, if we perturb the boundary data by a very small potential function $ \Psi (\cdot, \tau) = \psi_\tau (\cdot)$, we have the following theorem: 

\begin{mtheorem}[see Theorem \ref{thmexistbypertb} and Theorem \ref{thmweightestifoli}] \label{mthmexistholodiskfoli}
    Let $\rho$ be the potential function of reference K\"ahler metric $\omega$ in each local holomorphic coordinates of $N' \subseteq X_l$, and let $G_0(w, \tau)$ be the trivial family of holomorphic discs associated with $ \Lambda_\rho$. There exists a uniform constant $\varepsilon > 0$, depending only on $n$, $k$, $\alpha$, and the geometry of $(X, J, g)$. If ${\Psi} \in \C^{\infty} (N' \times \partial D, \RR )$ and 
\begin{align*}
    || \partial \Psi ||_{k+1,\alpha; N' \times \partial D} \leq \varepsilon,
\end{align*}
then, there is a unique smooth family of holomorphic discs $G: N \times D \rightarrow E$ with
\begin{align*}
    ||G(x, \tau)-G_0(x, \tau)||_{k,\beta; N \times D} \leq C (\alpha-\beta)^{-2} ||\partial \widetilde{\Psi}-\partial \widetilde{\Psi}_0||_{k, \alpha; N'\times \partial D}
\end{align*}
where $C$ is a constant depending on $n,~ k,~\alpha,~ \varepsilon$ and the geometry of $(X, J, g)$. Moreover, $G$ satisfies the boundary condition: for each $\tau \in \partial D $ and each $x\in N'$, $G(x, \tau) \in \Lambda_{\tau}$, where $\Lambda_\tau$ is the graph of $\partial (\rho + \psi_\tau)$ on $\mathcal{W}$, and a fixed point condtion: $\tau \in \partial D$, and $\displaystyle  \pi\circ G (x, -i) = x$.

If we further assume that,
\begin{align*}
    ||  \partial \Psi ||_{-\gamma; k+1,\alpha; X \times \partial D} \leq C_0,
\end{align*}
then, we have the following weighted estimates on the displacement of holomorphic discs:
\begin{align} \label{estimeasureperturb}
    ||G(x, \tau) - G_0 (x, \tau) ||_{-\gamma, k,\beta; X_l \times D} \leq C ||\partial \Psi||_{-\gamma, k, \alpha; X_\infty\times \partial D}, 
\end{align}
where $G_0$ is the trivial foliation with respect to $ \Lambda_\rho$, and $C$ is a constant depending on $n, k,~\alpha,~(\alpha-\beta)^{-1},~ C_0$ and the geometry of $(X, J, g)$.
\end{mtheorem}

To place Theorem~\ref{mthmexistholodiskfoli} in context, let us briefly recall the history of related results. Donaldson sketched such a result in \cite[Theorem 1]{donaldson2002holo} by attempting to solve a family of Riemann–Hilbert problems via the classical inverse function theorem. A more detailed exposition of this approach was later given in Chen–Feldman–Hu \cite[Section 2]{CHEN2020108603}. However, to our understanding, there is a gap in \cite[Lemma A.3]{CHEN2020108603} concerning the regularity of the Hilbert transform. The Hilbert transform is known to be a bounded operator in Hölder spaces in the $S^1$ direction, but we will show that it is not a bounded operator in Hölder spaces in the parameter directions. Thus, the standard inverse function theorem cannot be used. Instead, we will use the Nash-Moser-Zehnder version of the inverse function theorem to prove Theorem \ref{mthmexistholodiskfoli}, crucially also using the theory of BMO spaces. 

The close relation between HCMA equations and holomorphic foliation was discussed in many earlier works \cite{bedford1977Foli, lempert1985sym, semmes1992complex, donaldson2002holo, Chen2008geofoli, CHEN2020108603}.
The essential result was proved by Semmes \cite{semmes1992complex}, and later reformulated by Donaldson \cite{donaldson2002holo}: provided the same boundary data, the existence of the smooth nondegenerate solution to the HCMA equation is equivalent to the existence of the holomorphic discs foliation on a compact K\"ahler. In Chen-Tian \cite{Chen2008geofoli} and Chen-Feldman-Hu \cite{CHEN2020108603} (also see Proposition \ref{propexplicitconstructionsolu}), a local version of equivalence was proved. That is to say, provided boundary data on some part of a given K\"ahler manifold, the local existence of the holomorphic discs foliation is equivalent to the existence of a local solution to the HCMA equation. 

However, a key issue is that such local holomorphic disc foliations are \emph{not uniquely determined} by the boundary K\"ahler forms \( \omega_\tau \) for \( \tau \in \partial D \). Consequently, even when a local solution to the HCMA equation exists, it does not necessarily coincide with the global solution. One of the main contributions of this paper is to clarify this point and provide a mechanism to ensure consistency between local and global solutions: by choosing local potential functions that are \emph{restrictions of a globally defined potential}, one can guarantee that the local foliation aligns with the global structure. In the ALE K\"ahler case, the global potential is uniquely determined by a decay condition. This allows us to confirm, in Step~3, that the local solution near infinity constructed via holomorphic disc foliations on \( X_{l_0} \) (Theorem~\ref{mthmexistholodiskfoli}) indeed agrees with the global solution.

The following theorem describes the last step:

\begin{mtheorem}[see Theorem \ref{thmconstpshsolu} and Theorem \ref{thmweiestisolutiontoHCMA}] \label{mthmconstrucsubsol}
Let $\Phi$ be the global $\C^{1,1}$ solution to the HCMA equation \eqref{introHCMA}. There exists a bounded, continuous, $\Omega$-plurisubharmonic subsolution $F$ to the HCMA equation. Moreover, there exists a large uniform constant $l$ such that 
\begin{align*}
    F= \Phi \qquad \text{ on } X_l.
\end{align*}
If the boundary function $\Psi $ in \eqref{introHCMA} belongs to $ \C^{\infty}_{-\gamma} (X \times \partial D)$, then there exists a uniform constant $C$ depending only on $n$, $k$, $\alpha$, $||\Psi||_{-\gamma; k+3, \alpha; X \times D}$, and the geometry of $(X, J, g)$ such that
\begin{align*}
    ||F||_{-\gamma; k,\beta; X_l \times D} \leq C.
\end{align*}
\end{mtheorem}

The construction of $F$ relies essentially on the holomorphic disc foliation. We now briefly outline the construction of the auxiliary function $F$ in Step~3.
Along each holomorphic disc 
\(
   \mathcal{L}_{x_0} = \{\pi \circ G(x_0,\tau) \; ; \; \tau \in D\},
\)
we define a function $L_{x_0}$ in local holomorphic coordinates around 
$\mathcal{L}_{x_0} \subseteq N' \times D$. 
On each leaf $\mathcal{L}_{x_0}$, the function $L_{x_0}$ is taken to be harmonic in the $\tau$-variable, 
with boundary values given by ${\psi}_\tau + \rho$. 
In the spatial directions of $N' \times D$, we define $L_{x_0}$ to be linear, so that the derivative of 
$L_{x_0}$ in the spatial direction coincides with the bundle components of the holomorphic disc $G(x_0,\tau)$. Please see (\ref{pshconst1})-(\ref{pshconst3}) for details.
One then verifies that $L_{x_0}$ is pluriharmonic on $N' \times D$ (see Lemma \ref{lempshLcheck}). 

The auxiliary function $F$ is then roughly defined by taking the supremum of $L_{x_0}-\rho$ over 
$x_0 \in X_{2l}$ together with the constant function $-M$, where $M$ denotes the a priori global 
$C^0$ bound of the solution $\Phi$ to HCMA equation. 
The smallness condition on the prescribed boundary potential $\Psi$ over $X_{2l}$ ensures the 
convexity of the local potential $\rho+\Psi$, and hence guarantees that $F$ is globally well-defined $\Omega$-psh.  
We refer to Section~\ref{secconstructF} for the detailed construction.

The main contributions of this paper are twofold. 
On the one hand, Theorem~\ref{mthmexistholodiskfoli} fills a gap in the works of Donaldson~\cite{donaldson2002holo} and Chen--Feldman--Hu~\cite{CHEN2020108603}, while Theorem~\ref{mthmconstrucsubsol} provides the construction of a global $\Omega$-plurisubharmonic function $F$, which is an exact solution on the end $X_{2l}$ and extends globally as a bounded subsolution, giving the asymptotic control needed for Theorem~\ref{mthmasymbehHCMA}. 
On the other hand, these results establish a Bedford--Taylor type pluripotential theory in the ALE setting and lead to a local strong regularity theorem for the HCMA equation.

The paper is organized as follows. In Section~\ref{secpre}, we review preliminary definitions and results that will be used repeatedly throughout the paper. Many of these are known to experts, but we include them here to keep the presentation self-contained. Sections~\ref{secupenvelope} and~\ref{secMPofMAALE} complete Step~1: we study the regularity of the upper envelope of bounded $\Omega$-plurisubharmonic subsolutions and relate it to the existence of weak solutions to the HCMA equation on $X \times D$. We also prove a version of the maximum principle for the Monge-Ampère operator (Theorem~\ref{intromainmp}), which ensures the uniqueness of bounded weak solutions. Section~\ref{secC11esti} is devoted to establishing $C^{1,1}$ regularity of the global solution and completing the proof of Theorem~\ref{mthmHCMAc11}. 
In Sections~\ref{secholodiscfoliexist} and~\ref{secuniqholodiscfoli}, we carry out Step~2 by giving a PDE-based proof of the existence of holomorphic disc foliations via a small perturbation argument using the Nash–Moser technique, discuss uniqueness, and prove a patching theorem; thus, we complete the proof of Theorem~\ref{mthmexistholodiskfoli}. In Section~\ref{secmaxrkthm}, we complete Step~3 by constructing a global bounded continuous $\Omega$-plurisubharmonic subsolution that connects the local and global solutions, thereby proving Theorem~\ref{mthmconstrucsubsol}. Section~\ref{secweiesti} finishes the proof of the main result, Theorem~\ref{mthmasymbehHCMA}, by establishing weighted estimates. Finally, we prove a local regularity result for the HCMA equation, stated as Theorem~\ref{mthmlocalreg}.

\begin{ac} \textup{
The author thanks Professors Xiuxiong Chen and Claude LeBrun for their interest in this problem and for several stimulating conversations, as well as the colleagues Jingrui Cheng, Jian Wang, and Charles Cifarelli for their encouragement. The author also thanks Professor Xiuxiong Chen for helpful comments on this work.
}
\end{ac}

\section{Preliminaries} \label{secpre}

\subsection{Inhomogeneous weighted H\"older norms}\label{secwholdernorm} Let $X$ be an ALE K\"ahler manifold with the end $X_\infty$. In the subsection, we discuss the H\"older norm and the weighted H\"older norm of a tensor field on certain open subsets of the universal covering $\widetilde{X}_\infty$ of the end, defined with respect to the Euclidean metric $g_0$. Since $\widetilde{X}_\infty$ is diffeomorphic to $\RR^{2n} - B_R$, an open subset of $\widetilde{X}_\infty$ can be viewed as an open set in $\CC^n$ with the standard Euclidean coordinates. 
\begin{define}
    Let $f$ be a function defined on a domain $N$ in $\CC^n$. We say $f$ is of the class $\C^{k,\alpha} (N)$, if 
    \begin{align*}
        ||f||_{k,\alpha; N} = \sum_{|\beta|=0}^k \sup_{N}|D^\beta f|_{g_0} + [f]_{k,\alpha; N} \leq  C.
    \end{align*}
     where the H\"older norm $[f]_{k,\alpha; N}$ is defined as 
     \begin{align*}
         [f]_{k,\alpha; N} = \sum_{|\beta| = k} \sup_{x,y \in N} \frac{|D^\beta f (y) - D^\beta f(x)|}{|y-x|^\alpha}.
     \end{align*}
\end{define}
A tensor field $\mathbf{T}$ defined in $N$ can be represented in the standard coordinate frame of $\RR^{2n}$, 
\begin{align*}
    \mathbf{T} = \sum_{\tiny{\substack{i_1,\ldots, i_q\\j_1,\ldots, j_p} }}T^{j_1\ldots j_p}_{i_1\ldots i_q} \frac{\partial}{\partial x^{j_1}} \otimes \ldots \otimes \frac{\partial}{\partial x^{j_p}} \otimes d x^{i_1} \otimes \ldots \otimes d x^{i_q}.
\end{align*}
Then, we say $\mathbf{T} \in \C^{k,\alpha}(N) $ if each coefficient function, $T^{j_1\ldots j_p}_{i_1\ldots i_q} $, is of class $\C^{k,\alpha}(N)$. 
 \begin{lem}\label{ckalphaesticomposfun}
     Let $N$, $N'$ be domains in $\RR^n$. If $f \in \C^{k,\alpha} (N')$ and $u \in \C^{k,\alpha} (N, N')$, then $g(x) = f(u(x)) \in \C^{k,\alpha} $ with $k \geq 1$, $\alpha \in (0,1)$, then we have
     \begin{align*}
         ||g||_{k,\alpha; N} \leq  ||f||_{k,\alpha; N'} (1 + \sup_{N}|D u|^\alpha|) P(||D u||_{k-1,\alpha; N}).
     \end{align*}
     where $P$ is an order $k$ polynomial with positive coefficients depending only on $n$ and $k$. 
 \end{lem}
 \begin{proof}
     The proof is based on direct calculation. Assuming $y = u(x)$, the $k$-th derivatives of $g$ is given by the following formula:
     \begin{align} \label{calckdericomposfun}
         \begin{split}
    D_{x_{i_1}} \cdots D_{x_{i_k}}g(x) =\sum^{1 \leq t \leq k}_{\scaleto{1\leq k_1 < k_2 <\ldots < k_t = k}{5pt}} & \sum_{j_1, \ldots j_t} f_{y_{j_1}\ldots y_{j_t}}(y) 
    \\ &
    \cdot \frac{\partial^{k_1} y_{j_1}}{\partial x_{i_1} \ldots \partial x_{i_{k_1}}} (x) \cdots \frac{\partial^{k_t-k_{t-1}} y_{j_t}}{\partial x_{i_{k_{t-1}+1}}  \ldots \partial x_{i_{k_{t}}}} (x),
\end{split}
     \end{align}
     Hence, we have 
     \begin{align*}
         \sup_{N} \big| D^k_x g(x) \big| \leq ||f||_{k; N'} P_k \big(||D u||_{k-1; N}\big),
     \end{align*}
     where $P_k $ is an order $k$ polynomial with positive coefficients depending only on $k$. For $\alpha$-H\"older norm of the $k$-th derivative of $g$, 
     \begin{align*}
         \sup_{x,x' \in N} \frac{|D^k_x g(x') - D^k_x g(x)|}{|x'-x|^\alpha} \leq ||f||_{k,\alpha; N'} (1 +\sup_N |D u|^\alpha|) P_k(||D u||_{k-1,\alpha; N}).
     \end{align*}
     Therefore, we complete the proof of the lemma.
 \end{proof}

Throughout this paper, we work with the H\"older norm including weights. Consider the weighted function $w(x) = (1+ r(x)^2)^{\frac{1}{2}}$, where $r$ denotes the standard radial function in $\RR^n$. The weighted H\"older norm is defined as follows:

\begin{define}\label{defweinorm}
    Let $f$ be a function defined on an open subset $N$ in $\RR^n$. We say $f$ is of the class $\C^{k,\alpha}_{-\gamma} (N)$, if 
    \begin{align*}
        ||f||_{-\gamma; k,\alpha; N} = \sum_{|\beta|=0}^k \sup_{N}| w^{\gamma +|\beta|} D^\beta f|_{g_0} + \sup_{x \in N } |w(x)|^{\gamma+k+\alpha} [f]_{k,\alpha; N \cap B_{r(x)/10}(x)} \leq  C.
    \end{align*}
\end{define}
Since we consider solutions on the product space $X \times D$, it is necessary to introduce the H\"older norms for joint derivatives on $X \times D$, with the weight applied in the direction of $X$.

  \begin{define}\label{defweinormprosp}
  Let $N$ be an open subset of $\RR^n$, and let $D\subseteq \RR^s$ be a bounded domain. Let $F$ be a function defined on $N \times D$. We say a function $F$ is of class $\C^{k,\alpha}_{-\gamma} (N \times D)$ with the weight applied in the direction of $N$ if 
  \begin{align*}
      \sup_{x \in N} w(x)^{\gamma}||F (x, \cdot)||_{k,\alpha; D} \leq C
  \end{align*}
  and all derivatives of $F$ in $D$, $D^{\eta}_\tau F(\cdot, \tau)$ with multi-indices $|\eta|\leq k$, is uniformly of the class $\C^{k-|\eta|,\alpha}_{-\gamma}(N)$ for all $\tau \in D$; precisely,
\begin{align*}
    \sup_{\tau \in D} ||D^\eta_{\tau} F(\cdot, \tau)||_{-\gamma; k-|\eta|,\alpha; N} \leq C_\eta.
\end{align*}
If we denote $||\cdot||_{-\gamma;\ k,\alpha; \ N\times D} $ to be the $\C^{k,\alpha}$ norm with inhomogeneous weight applied in $N$, then $F$ is of the class $\C^{k,\alpha}_{-\gamma}(N \times D)$ if 
\begin{align*}
    ||F||_{-\gamma;\ k,\alpha; \ N \times D} =  \sup_{x \in N} w(x)^{\gamma}||F (x, \cdot)||_{k,\alpha; D} + \sup_{\tau \in D} ||D^\eta_{\tau} F(\cdot, \tau)||_{-\gamma; k-|\eta|,\alpha; N} < + \infty.
\end{align*}
  \end{define}

The following lemma provides an equivalent characterization of the inhomogeneous weighted H\"older norms. To verify $F \in \C^{k,\alpha}_{-\gamma} (N \times D)$, one can avoid directly checking the definition by instead showing that $F$ admits a uniform bound in the usual $\C^{k,\alpha}$ norm after being pulled back via the scaling map, $s_x: B_1 \times D \rightarrow B_{r(x)/10} (x) \times D $, at each point $x \in N$. This approach is particularly useful in applications involving scaling techniques.
\begin{lem}\label{lemscalweightedholder}
    Given a point $x \in N \subseteq \RR^n-B_{10}$, consider a scaling map $s_x: B_1 \times D \rightarrow B_{r(x)/10} (x) \times D$. Then, $F \in \C^{k,\alpha}_{-\gamma} (N \times D)$ if and only if for each $x\in N$, there is a uniform constant $C$ such that
    $$w(x)^{\gamma} ||s_x^* F||_{k,\alpha; B_1\cap (s_x^* N) \times D} \leq C. $$
    Moreover, we have the equivalence in norms:
    \begin{align*}
        \sup_{x\in N} w(x)^{\gamma} ||s_x^* F||_{k,\alpha; B_1\cap (s_x^* N) \times D} \leq ||F||_{-\gamma; k,\alpha; N \times D} \leq \kappa \sup_{x\in N} w(x)^{\gamma} ||s_x^* F||_{k,\alpha; B_1\cap (s_x^* N) \times D}.
    \end{align*}
    for some $\kappa>1$ only depending on $k,\alpha$.
\end{lem}

\begin{proof}
    Let $N_x = N \bigcap B_{r(x)/10} (x)$. Notice that $||F||_{-\gamma; k,\alpha; N \times D} = \sup_x ||F||_{-\gamma; k,\alpha; N_x \times D}$. It is easy to observe that
    \begin{align*}
         w(x)^{\gamma} ||s_x^* F||_{k,\alpha; B_1\cap (s_x^* N) \times D} 
         &\leq ||F||_{-\gamma; k,\alpha; N_x \times D}\\ 
         &\leq \sup_{x'\in B_{r(x)/10}(x)}\bigg|\frac{10w(x')}{r(x)} \bigg|^{k+\alpha} w(x)^{\gamma} ||s_x^* F||_{k,\alpha; B_1\cap (s_x^* N) \times D}.
    \end{align*}
    Let $\kappa = 12^{k+\alpha}$; hence, we complete the proof. 
\end{proof}

From Lemma \ref{lemscalweightedholder}, we can easily derive the following inhomogeneous weighted H\"older norm of a product: for $f \in \C^{k,\alpha}_{-\gamma_1} (N \times D)$ and $g \in \C^{k,\alpha}_{-\gamma_2} (N \times D)$
\begin{align} \label{estiproductweiholder}
    ||f\cdot g||_{-\gamma; k,\alpha; N \times D} \leq  ||f||_{-\gamma_1; k, \alpha; N \times D}\cdot ||g||_{-\gamma_2; k, \alpha; N \times D},
\end{align}
where the weighted orders satisfy $\gamma_1 + \gamma_2 = \gamma$. Lemma \ref{ckalphaesticomposfun} can also be generalized to the setting of the inhomogeneous weighted H\"older norms.

 \begin{lem}\label{ckalphaweiesticomposfun}
     Let $N$, $N'$ be domains in $\RR^n$ with $N\subseteq N'$.
     Assume $u\in \C^{k,\alpha} (N \times D, N')$, and that for each $\tau \in D$, $u(\cdot,\tau): N \rightarrow N'$ is a diffeomorphism onto its image. Suppose that 
     $$||u-\id ||_{-\varsigma;k,\alpha; N \times D} \leq C_0,$$
     where $\id(\cdot, \tau)$ denotes the identity map for each $\tau \in D$, and $\varsigma \geq 0$.
     If $F \in \C_{-\gamma}^{k,\alpha} (N'\times D)$,
     assuming $k \geq 1$ and $\alpha \in (0,1)$,
     then we have $G(x, \tau) = F(u(x, \tau), \tau) \in \C_{-\gamma}^{k,\alpha}(N \times D) $ and
     \begin{align} \label{estiweikalphacomposfunprosp}
         ||G||_{-\gamma; k,\alpha; N \times D} \leq C ||F||_{-\gamma; k,\alpha;N'\times D}.
     \end{align}
     where $P$ is a polynomial of order $k$ with positive coefficients depending only on $k$, and $C$ is a uniform constant depending on $n$, $k$, $\alpha$, $\gamma$ and $C_0$.
 \end{lem}
 \begin{proof}
The proof of this lemma is based on the scaling technique, Lemma \ref{lemscalweightedholder}, and Lemma \ref{ckalphaesticomposfun}. Consider $x \in N\subseteq N'$ and the scaling map $s_x: B_1 \rightarrow B_{\delta r(x)} (x)$. 
Let $u^* = \big( \delta r(x)\big)^{-1} s_x^* u $, $F^* = s_x^* F$ and $G^* = s_x^* G$. The $\C^{k,\alpha}$ map $u^*: B_\frac{1}{2} \cap s_x^* N \rightarrow B_1 \cap s_x^* N' $ satisfies,
\begin{align} \label{esticomposicondiscal}
w(x)^{\varsigma} ||u^* - \id||_{k,\alpha; (B_\frac{1}{2} \cap s_x^* N) \times D} \leq C_0.
\end{align}
If we take derivatives in the coordinates of $B_1$, we have $w(x)^{\varsigma}||D^* (u^*-\id)||_{k-1,\alpha} \leq C$. By the definition of $u^*$, $F^*$ and $G^*$, it's easy to derive that $G^*(x', \tau) = F^* (u^* (x',\tau), \tau)$ for $x' \in B_\frac{1}{2} \cap s_x^* N$. Lemma \ref{ckalphaesticomposfun} indicates that
\begin{align*}
    ||G^*||_{k,\alpha; (B_{\frac{1}{2}} \cap s_x^* N ) \times D} \leq ||F^*||_{k,\alpha; (B_1 \cap s_x^* N') \times D} (1+ \sup_{(B_{\frac{1}{2}} \cap s_x^* N) \times D}{|D^* u^*|^{\alpha}}) P(||D^*u^*||_{k-1,\alpha;B_{\frac{1}{2}} \cap s_x^* N }),
\end{align*}
where $P$ is a polynomial of order $k$ with positive coefficients depending only on $k$. We derive from \eqref{esticomposicondiscal} that
\begin{align*}
    ||D^*u^* ||_{k-1,\alpha; B_{\frac{1}{2}} \cap s_x^* N}   \leq C_0 \big( 1+ w(x)^{-\varsigma} \big).
\end{align*}
Inserting the above estimate for $D^* u^*$ into the inequality for $G^*$, and  multiplying by the weight on both sides, we obtain
\begin{align*}
    w(x)^\gamma ||G^*||_{k,\alpha; (B_{\frac{1}{2}}\cap s_x^* N) \times D} \leq C ||F||_{-\gamma; k,\alpha;N'\times D},
\end{align*}
where $C$ is a uniform constant depending on $n$, $k$, $\alpha$, $\gamma$ and $C_0$.
Applying Lemma \ref{lemscalweightedholder} once again, we complete the proof of \eqref{estiweikalphacomposfunprosp}.
\end{proof}

 Throughout this paper, the regularity of functions on product spaces will be a recurring topic of discussion. The following lemma provides a useful tool for this purpose.
 The result is due to Bernstein \cite{Bernstern1958tranconstructive}, and please also refer to \cite{krantz2008ontology}.

 \begin{lem} \label{lemfamilylaplace}
     Let $U $ and $V$ be open subsets of $\RR^n$ and $\RR^m$ respectively. Let $k$ be a nonnegtive inetger and $\alpha \in (0,1)$. Consider a function $f(x, y)$ defined on $U \times V$. Suppose that for each $x \in U$, $f_x: y \rightarrow f(x, y)$ is a $\C^{k,\alpha}$ function on $y$ and for each $y \in V$, $f_x: y \rightarrow f(x, y)$ is a $\C^{k,\alpha}$ function on $x$. Additionally, we assume that there exists a uniform constant $C$ such that
     \begin{align*}
         ||f(\cdot, y)||_{k,\alpha; U \times \{y\}} \leq C, \qquad ||f(x, \cdot)||_{k,\alpha; \{x\} \times V} \leq C.
     \end{align*}
     Then, $f(x, y)$ is a $\C^{k,\alpha}$ function on $U \times V$ and 
     \begin{align*}
        ||f(x, y)||_{k,\alpha; U \times V} \leq C.
     \end{align*}
 \end{lem}

We discuss H\"older estimates related to the Dirichlet problem of the Laplacian equation,
\begin{align} \label{equationDirichlet}
\begin{split}
    \Delta u = 0, \qquad &\text{ in } D,\\
    u = f, \qquad & \text{ in } \partial D.
\end{split}
\end{align}
It is known that $u$ is an analytic function in $D$ and continuous on the boundary if we assume $\psi$ is continuous on $\partial D$. Given a better regularity of the boundary function, $\psi$, the corresponding H\"older norm of $u$ can be controlled by the H\"older norm of $\psi$. The family of Dirichlet problems with parameter $x \in N$ is given the following system:
 \begin{align}\label{equationDiripara}
 \begin{split}
       &\  \Delta_\tau \tilde{u}(x, \tau) = 0, \hspace{1cm}(x, \tau)\in N \times D;\\
        &\tilde{u}(x,\tau) = \tilde{f}(x,\tau), \hspace{0.65cm} (x, \tau) \in N \times \partial D,
\end{split}
    \end{align}
    where $N$ is a domain in $\RR^n$. Now, if we assume the boundary function $\tilde{\psi}$ is bounded in the inhomogeneous weighted H\"older norm, then we will show that the solution to the above equation is also bounded in the inhomogeneous weighted H\"older norm of the same weight and regularity on the total space. In the following lemma, we summarize the estimates for the solutions to both \eqref{equationDirichlet} and \eqref{equationDiripara}.

\begin{lem} \label{lemplaestiwrtbdry} Let $u$ be the solution to the Dirichlet problem (\ref{equationDirichlet}). Assume that $k \geq 0$ and $0< \alpha < 1 $.
    If $\psi \in \C^{k,\alpha}(\partial D)$, then $ u \in \C^{\infty} (D) \bigcap \C^{k,\alpha} (\overline{D})$ with
    \begin{align} \label{estilakalpha}
        ||u||_{k,\alpha; \overline{D}} \leq C ||f||_{k,\alpha; \partial D},
    \end{align}
    where $C$ is the uniform constant depending on $k$, $\alpha$. Furthermore, consider a family of Dirichlet problems \eqref{equationDiripara}. If $\tilde{\psi} \in \C^{k,\alpha}_{-\gamma}(N \times \partial D)$ with $k \geq 0$ and $0< \alpha< 1$, then $\tilde{u}  \in \C^{k,\alpha}_{-\gamma} (N \times D)$ satisfying 
    \begin{align}\label{estilaweikalpha}
        ||\tilde{u}||_{-\gamma; k,\alpha; N \times D} \leq C ||\tilde{\psi}||_{-\gamma; k,\alpha; N \times \partial D},
    \end{align}
    where $C$ is the uniform constant depending on $k$, $\alpha$.
\end{lem}
\begin{proof}
    We first prove the H\"older regularity of the Dirichlet problem. The case $k\geq 2$ and $0<\alpha<1$ follows directly from the classic boundary Schauder estimates and the maximal principle. For the case $k=1$ and $0< \alpha < 1$, see \cite[Theorem 8.43]{gilbarg2001elliptic}. When $k=0$ and $0<\alpha<1$, the H\"older estimates can be proved directly using the Poisson integral,
    \begin{align*}
        u (r, \eta) = \int_{\partial D} P(r, \theta) f(\theta) d\theta. \hspace{0.5cm} P(r,\eta-\theta) =  \frac{1}{2\pi} \frac{1-r^2}{1-2r \cos \theta + r^2} 
    \end{align*}

We first consider the $\alpha$-H\"older norm near the boundary, namely, 
\begin{align*}
    \frac{|u(x)- u(x')|}{|x-x'|^\alpha}, \hspace{0.8cm} \text{ for } x \in \partial D, \ x \in D \text{ and } |x-x'|< \frac{1}{2}.
\end{align*}
The H\"older norm in the angular direction is from the following calculation:
\begin{align*}
    u(r, \eta ) - u(r, \eta_0) =\int_{\partial D} P(r, \eta - \theta) \big( f(\theta -\delta \eta) - f(\theta)\big)( d\theta,
\end{align*}
where $\delta \eta = \eta - \eta_0$. Then, we have
\begin{align*}
    |u(r, \eta ) - u(r, \eta_0)| \leq ||f||_{0,\alpha; \partial D} |\delta\eta|^\alpha.
\end{align*}
Let $x'= (r, \eta_0)$ and $y = (r, \eta)$, where $1/2< r <1$. Then, we have
\begin{align} \label{estiholdernearbdry1}
    \frac{|u(x')- u(y)|}{|x'-y|^\alpha} \leq \pi^{\alpha} ||f||_{0,\alpha; \partial D}
\end{align}
The key point is to estimate the H\"older norm in the radial direction. Let $x= (1, \eta)$ and $y= (r, \eta)$. Then, we observe that
\begin{align*}
    u(r,\eta)-u(1,\eta) = \bigg\{ \int_{0<|\theta|< 1-r} +   \int_{1-r<|\theta|<\pi} \bigg\}
 P(r, \theta ) \big(f(\eta-\theta) - f(\eta)\big) d\theta 
    \end{align*} 
The following two inequalities for the Poisson kernel will be applied to the above integrals separately:
$$|P(r, \theta)| \leq \frac{2}{1-r}$$
$$|P(r, \theta)| \leq \frac{1-r}{\theta^2}, \hspace{0.8cm} \text{ for } \frac{1}{2} < r< 1. $$
Then, we have
\begin{align}
    |u(x)-u(y)| &\leq 2\int_{0}^{1-r}\frac{2||f||_{0,\alpha;\partial D }}{1-r}  \theta^\alpha d\theta 
    + 
    2\int_{1-r}^\pi \frac{2(1-r) ||f||_{0,\alpha; \partial D}}{\theta^{2-\alpha}} d\theta \nonumber \\
    & \leq \frac{4}{1-\alpha} (1-r)^\alpha ||f||_{0,\alpha; \partial D}. \label{estiHoldernearbdry2}
\end{align}
Combining \eqref{estiholdernearbdry1} and \eqref{estiHoldernearbdry2}, we obtain
\begin{align*}
      \frac{|u(x)- u(x')|}{|x-x'|^\alpha} \leq C||f||_{0,\alpha; \partial D}, \hspace{0.8cm} \text{ for } x \in \partial D, \ x \in D \text{ and } |x-x'|< \frac{1}{2}.
\end{align*}

Now, for $x,x' \in D$ with $|x-x'| < \frac{1}{2}$, let $t = x-x' $. The function $u(x) - u(x-t)$ is a harmonic function in $x$ defined in $D \cap D_t$, where $D_t$ is the unit disc centered at $t$. Then, by the maximal principle, we have
\begin{align*}
    \sup_{D\cap D_t} |u(x)- u(x-t)|\leq \sup_{\partial (D \cap D_t) } |u(x)- u(x-t)| \leq C |t|^{\alpha} ||f||_{0,\alpha; \partial D}.
\end{align*}
In conclusion, we proved \eqref{estilakalpha}. 

To show \eqref{estilaweikalpha}, we apply the scaling technique to the parameter space $N$. Let $s_x: B_1 \rightarrow B_{\delta r(x)} (x)$ be the scaling map. By Lemma \ref{lemscalweightedholder}, we have 
\begin{align*}
    \kappa^{-1}||f||_{-\gamma; k,\alpha; N \times \partial D} \leq 
    \sup_{x\in N } r(x)^{\gamma}||s_x \tilde{f}||_{k,\alpha; B_1 \cap (s_x^* N \times \partial D)}\leq ||f||_{-\gamma; k,\alpha; N \times \partial D}.
\end{align*}
The family of Dirichlet problems after scaling is given by
\begin{align*}
    &  \Delta_\tau \tilde{v} = 0, \hspace{1cm}\text{ in }  B_1 \cap (s_x^* N \times  D);\\
        &\tilde{v} = s^*\tilde{f}, \hspace{1.15cm} \text{ in } B_1 \cap (s_x^* N \times \partial  D).
\end{align*}
Note the $\tilde{u} (x', \tau) = \tilde{v} ((\delta r(x))^{-1}x', \tau)$ gives the solution to the original Dirichlet problem \eqref{equationDiripara}.
By taking derivatives in the parameter direction after scaling and applying the H\"older estimates for the Dirichlet problem \eqref{estilakalpha}, we have
\begin{align*}
    ||\tilde{v}||_{k,\alpha;B_1 \cap (s_x^* N \times  D) } \leq C ||s^*_x\tilde{f}||_{k,\alpha; B_1 \cap (s_x^* N \times \partial D)}.
\end{align*}
Noting that $\tilde{v} = s^* \tilde{u}$, we obtain
\begin{align*}
    ||\tilde{u}||_{-\gamma; k,\alpha; N\times D} \leq \kappa C ||\tilde{f}||_{-\gamma; k,\alpha; N \times \partial D},
\end{align*}
which completes the proof.
\end{proof}

\subsection{The asymptotic charts of ALE K\"ahler manifolds} \label{secasympchartALEK}
Many works have been dedicated to finding the optimal asymptotic charts of ALE K\"ahler manifolds.
In this subsection, we summarize the remarkable results regarding the optimal asymptotic charts of ALE K\"ahler manifolds. The main reference of this subsection is \cite{hein2016mass}.

Let $(X, J, g)$ be an ALE K\"ahler manifold satisfying the fall-off conditions \eqref{decayaleintro} and let $X_\infty$ be the end of $X$.  The diffeomorphism $I: X_\infty \rightarrow (\CC^n- B_R) /\Gamma $ gives a real asymptotic coodrinate system $\{x^1, \cdots, x^{2n}\}$ on the universal covering $\widetilde{X}_\infty$ of $X$. The asymptotic coordinates admits a standard Euclidean complex structure $J_0$ by setting $z^j = x^j + i x^{j+n}$,, $j = 1, \ldots, n$. In general, the complex structure $J$ decays to $J_0$ with the same rate as the metrics decay,
\begin{align*}
    |J-J_0| = O(r^{-\tau }), \qquad |\nabla^k_0 J| = O(r^{-\tau -k}).
\end{align*}
However, the decay rate of $J$ can be improved by choosing ``better'' asymptotic coordinates as follows: 

\begin{pro}\label{propcxasymcoordinates}
    Let $(X, g, J)$ be an ALE K\"ahler manifolds satisfying the fall-off conditions \eqref{decayaleintro}. When the complex dimension $n \geq 3$, there are asymptotic complex coordinates $(z^1, \ldots, z^n)$ on the universal covering $\widetilde{X}_\infty$ of the end $X_\infty$, in which the complex structure $J$ becomes the standard one $J_0$, and in which the metric $g$ decays to $g_0$
    \begin{align*}
        |g-g_0| = O(|z|^{-\tau}), \qquad |\nabla^k_0 g| = O(|z|^{-\tau -k}).
    \end{align*}
    When the complex dimension $n =2$, there are asymptotic real coordinates $(x^1, x^2, x^3, x^4)$ on the universal covering $\widetilde{X}_\infty$ in which the metric and the complex structure decay as follows:
    \begin{align*}
        |J-J_0| = O(|x|^{-3}), \qquad |\nabla^k_0 J| = O(|x|^{-3-k})
    \end{align*}
    and 
    \begin{align*}
        |g-g_0| = O(|x|^{-\tau}), \qquad |\nabla^k_0 g| = O(|x|^{-\tau-k}).
    \end{align*}
\end{pro}
\begin{proof}
    See \cite[Lemma 1.3, Proposition 3.5]{hein2016mass}.
\end{proof}

According to Proposition \ref{propcxasymcoordinates}, when the complex dimension $n\geq 3$, there is a biholomorphism between the universal covering of the end $\widetilde{X}_\infty$ and $\CC^n-B_R$. However, in the case of complex dimension $n=2$, we cannot find a biholomorphism in general. The decay rate that $J$ approaches $J_0$ is optimal due to the examples of Honda \cite{honda2014scalar}. 

For the analysis that follows, particularly in the constructive arguments, we will frequently work in holomorphic coordinates rather than general diffeomorphisms. To overcome the absence of asymptotic complex coordinates in complex surfaces, we will show that there exists a ``good" covering $\mathcal{U} = \{U_i| \ i \in \mathcal{I} \}$ of the end $X_\infty$. For each point in the end $ x\in X_\infty $, there is a biholomorphism $I_x: U_x \rightarrow B_{\tilde{r}_x} \subseteq \CC^n$. In the following lemma, we prove that the radius $R_x$ can be a sufficiently large constant with $R_x = \theta_0 r(x)$, $0<\theta_0<1$, where $\theta_0$ is a uniform constant independent of $x$.

According to Proposition \ref{propcxasymcoordinates}, there are asymptotic complex coordinates $(z_1, z_2, \ldots, z_n)$, on the universal covering $\widetilde{X}_\infty$ of the end $X_\infty$ such that the complex structure $J$ decays to the standard complex structure $J_0$ as the following:
\begin{align}\label{inproofdecaycxstruc}
    |J-J_0| \leq A r^{-\gamma},  \qquad |\nabla^k_0 J | \leq A r^{-\gamma-k}
\end{align}
and 
    \begin{align}\label{inproofdecaymetric}
        |g-g_0| \leq A r^{-\tau}, \qquad |\nabla^k_0 g| \leq A r^{-\tau-k},
    \end{align}
where $\gamma\geq n+1$ and $\gamma \geq \tau$. 

\begin{lem} \label{lemholocornearinfty}
    Let $(X,J,g)$ be an ALE K\"ahler manifold and let $X_\infty$ be the end of $X$ with complex asymptotic coordinates $(z_1,\ldots, z_n)$. If $l$ is a large constant and $X_l \subseteq X_\infty$, then for any $x \in X_l$, there is an open neighborhood $U$ of $x$ and a biholomorphism $I_x: (U, J) \rightarrow B_{R_x} \subseteq \CC^n$, where $R_x = \kappa r(x)$ with a constant $0<\kappa<1$ independent of $x$. 
    Furthermore, $I_x$ can be written in the standard coordinate system of $B_R \subseteq \CC^n$, $I_x = (\tilde{z}^1, \ldots, \tilde{z}^n)$, and let $(z^1, \ldots, z^n)$ be the complex asymptotic coordinates by shifting them such that the origin corresponds to $x$, then we have
    \begin{align}\label{weightesticoordinatesholocball}
 \qquad |\nabla_0^k (\tilde{z}^i- z^i)| = O(r(x)^{1-\gamma - k}), \qquad \text{for } k\geq 0.
    \end{align}
    where $-\gamma$ is the decay rate of $J-J_0$ given in (\ref{inproofdecaycxstruc}).
\end{lem}

\begin{proof}
Given $X_l \subseteq X_\infty$, we need to prove that, for each point $x \in X_l$, there exist $n$ holomorphic coordinate functions $(\tilde{z}_1, \ldots, \tilde{z}_n)$ in a large ball centered at $x$ such that $d\tilde{z}_1, d\tilde{z}_2, \ldots, d\tilde{z}_n$ are linearly independent. Now, fixing $x \in X_l$, let $B \subseteq B'$ in $X_\infty$ be two balls centered at $x$ with radius $R$, $R' \gg 1$ respectively in terms of Euclidean distance, where the large constants $c, R, R'$ will be determined later. In the following, we will prove the existence of holomorphic coordinates in $B'$. 

Let $(\partial, \overline{\partial})$ and $(\partial_0, \overline{\partial}_0)$ be the complex differential operators of $J$ and $J_0$ respectively. Starting from $dz_1, \ldots, dz_n$, we can derive a $\C^\infty$ basis $w^1, \ldots,w^n$ of $(1,0)$ form with respect to $(J,g)$ by projecting:
\begin{align*}
    w^j = \frac{1}{2} (d z^j + i J dz^j) =   dz^j + \frac{i}{2} (J-J_0) d z^j
\end{align*}
The decay of complex structure \eqref{inproofdecaycxstruc} implies that 
$ w^i = d z^i + O(r^{-\gamma})$. In the dual case,
the corresponding $(1,0)$ vector fields $\frac{\partial}{\partial w^1}, \ldots,  \frac{\partial}{\partial w^n}$ satisfy $ \frac{\partial}{\partial w^j}= \frac{\partial}{\partial z^j} + O(r^{-\gamma})$ for all $i =1, \ldots n$.
Now , we define a family of $(0,1)$-forms with respect to $J$ by 
\begin{align} \label{proofdericoordinates}
     f_i = \overline{\partial} z_i = \sum_j \Big(\frac{\partial}{\partial \bar{w}^j} z^i \Big) \bar{w}^j = O(r^{-\gamma}) , \qquad i=1,\ldots, n.
\end{align} 
To find the holomorphic coordinate functions $(\tilde{z}_1, \tilde{z}_2, \ldots, \tilde{z}_n)$ in $B'$, we solve the $\overline{\partial}$ problem $\overline{\partial} v_i = f_i$. If there is a solution $v_i$ to the $\overline{\partial}$-problem such that $v_i$ is as small as $f_i$, then $\{\tilde{z}_i =z_i- v_i; i = 1,\ldots,n \}$ is the class of holomorphic functions with respect to $J$ in $B'$. It suffices to show the existence of $v_i$ and the weighted estimates \eqref{weightesticoordinatesholocball}. 

To prove the weighted estimates for $v_i$, we apply the scaling technique. 
Fix $x \in X$ with $r(x) \geq 3 R_0$ and a large ball $B' =B_{\tilde{r}_x}$ centered at $x$ and radius $\tilde{r}_x = r(x)/2$ in terms of the Euclidean distance. Introducing the scaling map, $s_{\tilde{r}_x}: B_1 \rightarrow B'$, we will rewrite the $\overline{\partial}$ equation $\overline{\partial} v_i = f_i$ in $B_1$. Let $\widetilde{J}$ be the pull-back of $J$ under the scaling map and define $\tilde{f}_i $ as the pull-back of one-form given by $\tilde{f}_i  = \tilde{r}_x^{*} f_i$. For the sake of simplicity, we abuse the notation by letting $(z^1,\ldots, z^n)$ denote the standard Euclidean coordinates on $B_1$, and $w^1, \ldots, w^n$ represent the $(1,0)$ forms with respect to $\widetilde{J}$, defined as $w^j = (dz^j + i \widetilde{J} dz^j)/2$. Expressing $\widetilde{J}$ in the Euclidean coordinate frame, we have 
\begin{align*}
|\widetilde{J} - J_0| = O(\tilde{r}_x^{-\gamma}), \hspace{1.5cm} |\nabla^k_0 \tilde{J}|  = O(\tilde{r}_x ^{-\gamma}) \quad \text{ for } k\geq 1. 
\end{align*}
Then, the $\overline{\partial}$ equation becomes
\begin{align*}
    \overline{\partial} \tilde{v}_i = \tilde{f}_i,
\end{align*}
where $\overline{\partial}$ denotes the complex differential operator associated with $\widetilde{J}$ in $B_1$ and $|\nabla_0^k \tilde{f}_i| = O(r^{1-\gamma}_x)$, $k \geq 0$. 

The existence of $\tilde{v}_i$ directly follows from the classic H\"ormander's $L_2$ estimates \cite{hormander1973introduction}. By choosing the function $\mu = |z|^2$ in $B_1$, it can be verified that $\mu$ is a strictly plurisubharmonic function with respect to $\widetilde{J}$. Consider the weight function $ \nu = (1- \mu)^{-1}$. H\"{o}mander's $L^2 $ estimates imply that there exists $\tilde{v}_i \in L^2_{loc} (B_1)$ and 
\begin{align*}
     \int_{B_1} |\tilde{v}_i|^2 e^{-\nu} dx \leq C \int_{B_1} |\tilde{f}_i|^2 e^{-\nu} dx.
\end{align*}
By restricting to a smaller ball, the above inequality implies $\int_{B_{3/4}} |\tilde{v}_i|^2 dx  \leq C \int_{B_1} |\tilde{f}|^2 d x.$
According to regularity estimates on Sobolev norms, we have
\begin{align*}
    \sum_{|\alpha| \leq k+1} \int_{B_{1/2}} \big|D^\alpha \tilde{v}_i \big|^2 dx \leq C \bigg(\sum_{|\alpha|\leq k} \int_{B_1} \big|D^\alpha \tilde{f}_{i}  \big|^2 dx + \int_{B_{3/4}} |\tilde{v}_i |^2 dx \bigg),
\end{align*}
where $D^\alpha$ is the derivatives in terms of the Euclidean coordinates with multi-indices $\alpha$.
Hence, $\tilde{v}_i$ belongs to $W^{k}$ for all integers $k \geq 1$ in $B_{1/2}$. Sobolev embedding theorem implies that $\tilde{v}_i \in \C^\infty(B_{1/2})$.
The Sobolev inequality also indicates the following $L^\infty$ estimates:
\begin{align} \label{Linftyesti}
    \sup_{{B_{1/2}}}|\tilde{v}_i| \leq C \sum_{|\alpha|\leq n+1} \Big(\int_{B_1} \big|D^\alpha \tilde{f}_{i}  \big|^2 dx \Big)^{1/2}.
\end{align}
Therefore, we have $\sup_{B_{1/2}} |\tilde{v}_i| = O(r(x)^{1-\gamma})$.

We now discuss the weighted estimates for $v_i$. By applying derivatives to the equation $\overline{\partial} \tilde{v}_i = \tilde{f}_i$, we obtain 
 \begin{align*}
     \Delta_0 \tilde{v}_i + \mathbf{b} \cdot \nabla_0\tilde{v}_i = \tilde{u}_i, \qquad \text{ in } B_1,
 \end{align*}
where $|\mathbf{b}| \leq |\nabla_0 \tilde{J}|$ sufficiently small, and $\tilde{u}_i = \tr_{g_0} \partial \tilde{f}_i $ with $||\tilde{u}_i||_{k-2, \alpha; B_{1/2}} = O(r(x)^{1-\gamma})$. By the classic Schauder estimates, we have
\begin{align} \label{SchauderestiB1}
    ||\tilde{v}_i||_{k,\alpha; B_{1/4} } \leq C \big( ||\tilde{u}_i||_{k-2,\alpha; B_{1/2}} + ||\tilde{v}_i||_{0; B_{1/2} } \big)
\end{align}
The scaling map gives the relation:
\begin{align*}
    \sup_{B_1}|\nabla_0^m \tilde{v}_i| = \tilde{r}_x^m \sup_{B'} |\nabla_0^m {v_i}|.
\end{align*}
The Schauder estimates in \eqref{SchauderestiB1} imply that $||\tilde{v}_i||_{k,\alpha; B_{1/4} } = O(r(x)^{1-\gamma})$. Let $B'' = s_{\tilde{r}_x} (B_{1/2})$, $B=s_{\tilde{r}_x} (B_{1/4})$, and $u_i (z) = \tilde{u}_i (z/\tilde{r}_x)$. Then, we have
\begin{align*}
    ||v_i||_{1-\gamma;\ k,\alpha;\ B} \leq C \big( ||{u}_i||_{1-\gamma;\ k-2,\alpha;\ B''} + ||{v}_i||_{1-\gamma;\ 0;\ B'' }  \big)
\end{align*}
The above inequality, together with \eqref{Linftyesti} and the definition of $\tilde{f}_i$, $\tilde{u}_i$, we have $||v_i||_{1-\gamma; \ k,\alpha; \ B} \leq C r(x)^{1-\gamma}$. Then, the estimates of coordinate transform \eqref{weightesticoordinatesholocball} follow immediately. The image of $B$ under the biholomophism, given by $I_x (z) = \tilde{z}$, contains a holomorphic ball with radius at least $\frac{1}{9} r(x)$, where we can take $\kappa= \frac{1}{9}$.
\end{proof}

\begin{cor} \label{corcxasymptoticcovering}
    Let $(X, J, g)$ be an ALE K\"ahler manifold with the end $X_\infty$ and $X_l \subset X_\infty$. Then there exists an locally finite and countable open covering $\{U_i|\ i \in \mathcal{I} \}$ of $X_l$ and each open set $U_i$ admits a biholomorphism $I_i : (U_i, J) \rightarrow B_{R_i} \subseteq \CC^n$ where $R_i = \kappa r(x)$, with $x$, $I(x) =0$, and the coordinate transformation under $I_i$ satisfies (\ref{weightesticoordinatesholocball}).
\end{cor}

\begin{proof}
    The proof is immediately from Lemma \ref{lemholocornearinfty}.
\end{proof}

\subsection{H\"older Estimates for $\overline{\partial}$ Equation} 
In this subsection, we discuss the H\"older estimates for the $\overline{\partial}$ equation, $\overline{\partial} u = f$, for a $(0,1)$ form $f$ with $\overline{\partial} f = 0$ on a domain $D$ with $\C^1$ boundary. The following Bochner-Martinelli-Koppelman integral representation generalizes the Cauchy integral formula to any $(p,q)$ form defined on a domain $D\subseteq \CC^n$ with $\C^1$ boundary. Let $f$ be a $(0,q)$ form on $D$
\begin{align}\label{BMKformula}
    f(z) = \int_{\partial D} f(w) \wedge K_q(w,z) dw + \int_D \overline{\partial} f(w)\wedge K_q (w, z) dw - \overline{\partial}_z \int_{D} f(w) \wedge K_{q-1} (w, z)dw,
\end{align}
where $K_q $, $0\leq q \leq n-1$ are Bochner-Martinelli kernels defined to be 
\begin{align*}
    K_q (w, z) = \sum_{|J| =q} \big(- *_{w} \partial_w \Gamma(w, z) dw^{J}\big) \wedge d\bar{z}^J, 
\end{align*}
where $*_{w}$ is the Hodge star operator in  the Euclidean coordinates $w$ in $\CC^n$, and $\Gamma(w, z)$ is the \textit{Newtonian potential} of real dimension $2n$. The kernel $K_{q} (w, z)$ is a differential form defined on $\CC^n \times \CC^n$, smooth away from the diagonal $\{w=z\}$, and of type $(n, n-q-1)$ in $w$ and $(0,q)$ in $z$. The BMK equality \eqref{BMKformula} can be proved easily from the definition of $K_q$. In this paper, we only care for the case where $f$ is a $(0,1)$ form in $\Omega$. The kernel $K_0 (w,z)$ can be written explicitly as follows:
\begin{align*}
    K_0 (w, z) = &\frac{1}{\sigma_{2n-1}} \sum_{k} \frac{(\bar{w}_k-\bar{z}_k) }{|w-z|^{2n}}\frac{i}{2^{n-1}} dw^k \bigwedge_{j\ne k} \big( i dw^j \wedge d\bar{w}^j \big),
    \\
    K_1 (w, z) = &\frac{1}{\sigma_{2n-2}} \sum_{k<l} \frac{(\bar{w}_k-\bar{z}_k) }{|w-z|^{2n}}\frac{1}{2^{n-1}} dw^k \wedge dw^l \bigwedge_{j\ne k,l} \big( i dw^j \wedge d\bar{w}^j \big) \wedge d\bar{z}^l,
\end{align*}
where $\sigma_{2n-1}$ is the volume of the unit sphere of dimension $2n-1$.
Applying the BMK equality, together with the explicit formula for $K_q$, we have the following results on the H\"older estimates for $\overline{\partial}$ equation:

\begin{pro}\label{prodbaresti}
    Let $f$ be a $(0,1)$ form defined on the unit ball $B_1 \subseteq \CC^n$ satisfying $\overline{\partial} f =0$. If $f \in \C^{k,\alpha} (B_1)$ for $k\geq 1$ and $0< \alpha<1$, then the equation $\overline{\partial} u = f$ on $B_{1/2}$ has a solution $u \in \C^{k+1,\alpha}(B_{1/2})$ satisfying 
    \begin{align}\label{dbaresti}
        ||u||_{k+1,\alpha; B_{1/2}} \leq C ||f||_{k,\alpha; B_1},
    \end{align}
    where $C$ is a uniform constant depending only on $n,~k,~\alpha$.
\end{pro}
\begin{proof}
    Let $\chi$ be a cutoff function so that $\chi \equiv 1 $ on $B_{3/4}$ and has a compact support in $B_1$. Additionally, we can assume that the gradient of $\chi$ satisfies $|D\chi| \leq 4$. If we apply BMK integral representation to $\chi f$ on $B_1$, we have
    \begin{align*}
        \chi f (z) = \int_{B_1} \overline{\partial} \big(\chi f(w)\big) \wedge K_1(w, z) dw - \overline{\partial}_z \int_{B_1} \chi f (w) \wedge K_0 (w,z) dw, \qquad z \in B_1.
    \end{align*}
    Let $u_1 = -\int_{B_1} \chi f (w) \wedge K_0 (w,z) dw $ and $g=\int_{B_1} \overline{\partial} \big(\chi f(w)\big) \wedge K_1(w, z) dw$. Then, we have $g +\overline{\partial} u_1 = \chi f$ and $\overline{\partial} g =0$ on $B_{3/4}$. It suffices to construct $u_2$ such that $\overline{\partial} u_2 = g$ on $B_{3/4}$. The key observation is that $g$, viewed as a section of the cotangent bundle of over $B_{3/4}$,  can be factored through the space $B_{3/4} \times B_{3/4}$ as $g = \tilde{g} \circ h$, where $h (z) = (z, \bar{z})$ and $\tilde{g}(z, \zeta) = \int_{B_1} \overline{\partial} \big(\chi f(w)\big) \wedge \widetilde{K}_1(w, z, \zeta) dw$. The revised kernel $\widetilde{K}_1$ defined on $\CC^n \times \CC^n \times \CC^n$ is
    \begin{align*}
         \widetilde{K}_1 (w, z) = &\frac{1}{\sigma_{2n-1}} \sum_{k<l} \frac{(\bar{w}^k-\zeta^k) }{\big((w-z) \cdot (\bar{w}-\zeta)\big)^n}\frac{1}{2^{n-2}} dw^k \wedge dw^l \bigwedge_{j\ne k,l} \big( i dw^j \wedge d\bar{w}^j \big) \wedge d\zeta^l.
    \end{align*}
    Notice that $\overline{\partial} g = h^* (d_{\zeta} \tilde{g}) =0 $. By Poincar\'{e} lemma, for each fixed $z$ in $B_{3/4}$,  there exists a $\tilde{u}_2 (z, \zeta)$ such that $d_{\zeta} \tilde{u}_2 = \tilde{g}$ in $B_{3/4}$. Let $u_2 = h^* \tilde{u}_2$. Then, we have $ \overline{\partial} {u}_2  = h^* (d_{\zeta} \tilde{u}_2) = g$. In conclusion, we find a solution to $\overline{\partial}$ equation, $\overline{\partial} (u_1 + u_2) = g$. 

    The functions $u_1$ and $u_2$ can be written explicitly. Let us write the $(0,1)$ form $f$ as $f_{\bar{k}} d\bar{z}^k$. Inserting the formula of $K_0$ into the integral representation of $u_1$, we have
    \begin{align}\label{formuu1dbar}
        u_1 = \frac{2}{\sigma_{2n-1}} \sum_{k} \int_{B_1} \chi f_{\bar{k}}(w) \frac{(\bar{w}^k-\bar{z}^k) }{|w-z|^{2n}} dw
    \end{align}
    The explicit representation of $u_2$ is from the constructive proof of Poincar\'{e} lemma through the homotopy formula and we have
    \begin{align}\label{formuu2dbar}
        u_2 = \frac{2}{\sigma_{2n-1}} \sum_{k\ne l} \int_0^1 \int_{B_1} \big(\chi_{\bar{l}} f_{\bar{k}}- f_{\bar{l}} \chi_{\bar{k}} \big) (w)  \frac{(\bar{w}^k-t\bar{z}^k) \bar{z}^l }{\big((w-z)\cdot(\bar{w} - t\bar{z})\big)^n} dw dt.
    \end{align}

    We first derive the $L^\infty$ estimates for $u_1$, $u_2$.
 For $z \in B_{3/4}$, $u_1(z)$ satisfies
 \begin{align*}
     |u_1 (z)| &\leq \frac{2}{ \sigma_{2n-1}} \sum_{j}\sup_{B_1} |f_{\bar{j}}|  \int_{B_{1}} \frac{1}{|w|^{2n-1}} dw
     \\
     &\leq 2 \sum_{j} \sup_{B_1} |f_{\bar{j}}|. 
 \end{align*}
 For $u_2 (z)$, $z\in B_{3/4}$, we have
 \begin{align*}
     |u_2(z)| &\leq \frac{16 n}{ \sigma_{2n-1}} \sum_{j} \sup_{B_1} |f _{\bar{j}}| \int_0^1 |z| \int_{B_1} \frac{1}{|z-w|^n} \frac{1}{|tz-w|^{n-1}} d w dt.
 \end{align*}
Notice that
\begin{align*}
    \int_{B_1} \frac{1}{|z-w|^n} \frac{1}{|tz-w|^{n-1}} d w &\leq \int_{B_1} \frac{1}{|z-w|^{2n-1}} + \frac{1}{|tz-w|^{2n-1}} d w \\
    & \leq 2 \int_{B_{1}} \frac{1}{|w|^{2n-1}} dw 
\end{align*}
Hence, we obtain that
\begin{align*} 
|u_2 (z)|\leq 16 n \sum_j \sup_{B_1} |f_{\bar{j}}|.    
\end{align*}
In conclusion, we have the $L^{\infty}$ estimates of $u=u_1+u_2$: 
\begin{align} \label{Linftyesti}
\sup_{B_{3/4}}|u| \leq C_n  \sum_{i,j} \sup_{B_1} |f_{\bar{j}}|.
\end{align}

By taking derivative to $\overline{\partial} u = f$ on $B_1$, we have $\tr i \partial \overline{\partial} u = \tr i\partial f$. The classic interior Schauder estimates, together with $L^\infty$ estimates \eqref{Linftyesti}, imply that
\begin{align*}
    ||u||_{k+1, \alpha; B_{1/2}} \leq C \big( ||f||_{k, \alpha; B_{3/4}} + ||u||_{L^\infty(B_{3/4})}  \big) \leq C ||f||_{k, \alpha; B_{1}}.
\end{align*}
\end{proof}

In the next lemma, we deal with the H\"older estimates for a family of $\overline{\partial}$-equations. Consider a family of $\overline{\partial}$-closed $(0,1)$ forms in $B_1\subseteq \CC^n$ depending differentiably on parameters in a space $N\subseteq \CC^m$. Write the form as $f = f_i(z, \tau) d\bar{\tau}^i$. We say that $f \in \C^{k,\alpha} ( N \times B_1 )$ if each coefficient $f_{i}(z, p) \in \C^{k,\alpha}(B_1 \times N)$.

\section{The upper envelope of continuous plurisubharmonic functions} \label{secupenvelope}
Let $(X, J, g)$ be an ALE K\"ahler manifold satisfying the metric decay condition \eqref{decayaleintro}. 
In this section, we aim to build up the basic pluripotential theory on the product space $X \times D$, where $D$ is a domain in $\CC$ with $C^1$ boundary. 
Now, we define the following classes of $\Omega$-psh functions which are subsolutions to the HCMA equation \ref{introHCMA}:
\begin{align*}
    \B_{\Omega, \Psi} = \{u \in \psh_{\Omega}(X \times D);\  u \text{ is bounded, and } \limsup_{(x',\tau')\rightarrow (x, \tau)} u(x', \tau') \leq \Psi (x, \tau), \text{ on } X \times \partial D   \}
\end{align*}
and
\begin{align*}
     \F_{\Omega, \Psi}=\{u \in \psh_\Omega (X\times D) \bigcap \C (X \times D); \  u \text{ is bounded,} \text{ and } u(x, \tau) \leq \Psi(x, \tau), \text{ on } X \times \partial D \}
\end{align*}
In Subsections \ref{subsecapproxpsh}-\ref{subsecregupperenve}, we develop an approximation theorem for $\Omega$-psh functions on $X \times D$. Then, we establish a regularity theorem for the upper envelope of $\B_{\Omega, \Psi}$ under a uniform continuity assumption on the boundary data, and we will see that the upper envelopes of $\B_{\Omega, \Psi}$ and $\F_{\Omega, \Psi} $ agree on $X \times D$.
In Subsection \ref{subsecgeneralsolu}, it will be shown that the global weak solution to \eqref{introHCMA} can be obtained as the upper envelopes of the above classes.

\subsection{Approximation of $\Omega$-Plurisubharmonic Functions}\label{subsecapproxpsh}

Psh functions are present in many problems in K\"ahler geometry, mainly related to the Monge-Amp\`ere equations. It is a well-known fact that any psh function in a domain of $\mathbb{C}^n$ can be approximated from above by a sequence of smooth psh functions. It is of significant interest to prove such an approximation in some classes of K\"ahler manifolds. Guedj and Zeriahi \cite{Guedj2005} applied methods developed by Demailly (see, \cite{demailly1992reg,demailly2004Num,demailly2000pseudoeff}) to prove that such an approximation can be established in the case of compact K\"ahler manifolds admitting a positive holomorphic line bundle. Blocki and Kolodziej \cite{Blocki2007reg} prove the approximation in the arbitrary compact K\"ahler case. 

In this subsection, we extend the result of Blocki and Kolodziej to the class of noncompact complex manifolds that admit a ``good'' covering of holomorphic balls near infinity. By a ``good'' covering here, we mean a locally finite and countable open covering, $\{V_j; \ j \in \mathcal{J} \}$, of $X_l \subseteq X_\infty$ for some sufficiently large constant $l$. For each $V_j$, $j \in \mathcal{J}$, the open set $V_j$ is biholomorphic to a unit ball $B_1 \subseteq \CC^n$, and the transition map $T_{j_1,j_2}$ bewteen $U_{j_1}$ and $U_{j_2}$ satisfies uniform bounds on its differential: $A^{-1} \leq |DT_{j_1,j_2}| \leq A$, for a uniform constant $A$ independent of $j \in \mathcal{J}$. Such a covering can be constructed using Lemma \ref{lemholocornearinfty} and Corollary \ref{corcxasymptoticcovering}, by pulling back unit balls in $ B_{R_i}$ via the biholomorphic map $I_i$ for each $i \in \mathcal{I} $. The uniform bounds on $|DT_{j_1, j_2}|$ then follow from the estimates \eqref{weightesticoordinatesholocball}.

Let us recall some techniques developed in the flat case. The smooth approximation can be achieved by convolution. Let $\eta (z) = \hat{\eta} (|z|) \in \C^{\infty}(\CC^n)$ such that $\hat{\eta} \geq 0$, $\hat{\eta}(r) = 0$ for $r \geq 1$, and $\int_{\CC^n} \eta d\lambda =1$. Let $\eta_\delta ( z) = \delta^{-2n} \eta (z/ \delta)$ for $\delta >0$. Then, the convolution is given by
\begin{align*}
    u_\delta (z) = (u * \eta_{\delta}) (x) = \int u(z- w) \eta_\delta (w) d\lambda(w).
\end{align*}
If $u$ is psh, then $u_\delta$ is decreasing to $u$ as $\delta \rightarrow 0$.
The key idea of \cite{Blocki2007reg} is to compare the approximation under biholomorphic coordinate transformations. Then the global regularization of psh functions can be obtained by patching the local psh functions based on the regularized maximum developed in \cite{demailly1992reg}.

\begin{lem}[Blocki-Kolodziej \cite{Blocki2007reg}, Lemma 4] \label{BKlemmacompare}
    Let $U, V \subseteq \CC^n$ be open sets and $F: U\rightarrow V$, a biholomorphic mapping satisfying
        $A^{-1} \leq|dF| \leq A$.
    Let $u$ be a bounded psh function in $U$. The smooth psh function transformed by $F$ is defined as $u_\delta^F = (u\circ F^{-1})_{\delta} \circ F$. Then $u_\delta - u_{\delta}^F$ tends locally uniformly to $0$ as $\delta \rightarrow 0$. In particular, there exists a uniform constant $C$ only depending on $n$ and $\displaystyle\sup_{x\in U} |u|$ such that
    \begin{align} \label{comparebddpshesti}
        |u_\delta- u_\delta^F| \leq  - \frac{1+ \log A }{\log \delta} C.
    \end{align}
\end{lem}
\begin{proof}
    The proof is parallel to \cite{Blocki2007reg} based on the ideas developed from \cite{Kiselman1994}. Here, we only sketch how we get the estimate (\ref{comparebddpshesti}). We define two functions based on $u$,
    \begin{align}\label{defsuppsh}
        \widehat{u}_\delta (z) = \sup_{w\in\overline{B}(z, \delta)} u(w),
    \end{align}
    and 
    \begin{align} \label{defsphmeanpsh}
        \widetilde{u}_\delta (x) = \frac{1}{\sigma ( \partial B(z, \delta))} \int_{\partial B(z, \delta)} u d\sigma.
    \end{align}
     Both $\widehat{u}_\delta (z)$ and $\widetilde{u}_\delta (z)$ are increasing, and, according to Hadamard's 3-circles theorem,  $\widehat{u}_\delta (z)$ and $\widetilde{u}_\delta (z)$ are logarithmically convex in $\delta$. The logarithmic convexity of $u$ implies that,
    \begin{align*}
        0 \leq \widehat{u}_{A \delta} - \widehat{u}_{\delta} \leq \frac{\log A}{\log (r/ \delta)} (\widehat{u}_r - \widehat{u}_{\delta}) \leq  C \frac{\log A}{\log \delta^{-1}} 
    \end{align*}
    Let $\displaystyle \widehat{u}^F_\delta = (\widehat{u \circ F^{-1}})_\delta\circ F$. Hence, we have $\displaystyle \widehat{u}^F_\delta (z) = \max_{F^{-1}(\overline{B}(F(z), \delta) )} u$. The inclusion of sets, $\displaystyle \overline{B}(F(z), \delta) \subseteq F(\overline{B}(z, A \delta))$ and $F(\overline{B}(z, \delta)) \subset \overline{B}(F(z), A \delta)$, implies that
    \begin{align}\label{comparemaxdiffcoordinates}
        \big|\widehat{u}^F_\delta -\widehat{u}_\delta \big| \leq C \frac{\log A}{\log \delta^{-1}}
    \end{align}
    
    Let $h$ be the harmonic majorant of $u$ in $\overline{B} (z, r)$ (without loss of generality, we assume $h \leq 0$).
    Using Harnack's inequality for harmonic functions, we have that 
    \begin{align*}
        \widehat{u}_{s} \leq \sup_{\overline{B}(z,s) } h \leq \frac{1-s/r}{ (1+ s/r)^{2n-1}} h(z) = \frac{1-s/r}{ (1+ s/r)^{2n-1}} \widetilde{u}_{r}
    \end{align*}
    Now, we can compare $\widetilde{u}_\delta$ with $\widehat{u}_\delta$,
    \begin{align} \label{comparemaxsphmean}
        0 \leq \widehat{u}_\delta - \widetilde{u}_\delta \leq \frac{3^{2n-1}}{2^{2n-2}} (\widehat{u}_\delta - \widehat{u}_{\delta/2}) \leq \frac{C}{\log \delta^{-1}}. 
    \end{align}
    Let $\widetilde{\eta} (t) = \sigma(\partial B(0,1)) t^{2n-1}\hat{\eta} (t) $. Then, we have $\displaystyle \int_{0}^1 \widetilde{\eta} (t) dt =1$ and $\displaystyle u_\delta(z) = \int_{0}^1 \widetilde{u}_{t\delta} (z) \widetilde{\eta} (t) dt$. Using logarithmic convexity, 
    \begin{align} \label{comparesphmeaninv}
        0 \leq \widetilde{u}_\delta - u_\delta = \int_0^1 (\widetilde{u}_\delta- \widetilde{u}_{t \delta}) \widetilde{\eta} (t) dt \leq \int_{0}^1 \frac{\log (1/t)}{\log (r/t\delta)} (\widetilde{u}_r - \widetilde{u}_{t\delta}) \widetilde{\eta}(t) dt \leq \frac{C}{\log \delta^{-1}}.
    \end{align}
    Combining (\ref{comparemaxdiffcoordinates}), (\ref{comparemaxsphmean}) and (\ref{comparesphmeaninv}), we obtain (\ref{comparebddpshesti}). 
\end{proof}

\begin{pro} \label{regpsh}
    Let $(X, \omega)$ be an ALE K\"ahler manifold, and $D \subseteq \CC$, a bounded domain with $\C^1$ boundary. Let $\Omega = p^* \omega$. Then for every $\varphi \in \Pl_{\Omega} (X \times D)$, there exists a sequence $\varphi_{k} \in \Pl_{\Omega} (X \times D_{1/k}) $ decreasing to $\varphi$, where $D_{1/k} = \{ z\in D; ~ d_{g_{0}}(z, \partial D)  > 1/k\}$ and $d_{g_0}$ means the Euclidean distance.
\end{pro}

\begin{proof}
Let $K$ be a compact set given by $K= \{x \in X;\ r(x)\leq l \}$ for some large constant $l$, and the complement of $K$ in $X$, $X_l$, is contained in the asymptotic chart of $X$. By the discussion at the beginning of the subsection, there exists a locally finite and countable open covering of $X_l$, $\{V_j;\ j \in \mathcal{J} \}$, such that each $V_j$ is biholomorphic to a unit ball in $\CC^n$. For the compact set $K$, we may choose a finite open covering $\{V_j, 1\leq j \leq N\}$ of $K$. Consequently, we obtain a locally finite and countable open covering of $X$, $\{V_j; \ 
j \in \mathcal{J}' \}$ with $\mathcal{J}' = \mathcal{J} \sqcup \{1,\ldots, N \}$.
    
    Then, it is straightforward to construct three locally finite and countable open coverings of $X$, $\{V_j;~ j \in \mathcal{J}'\} $, $\{U_j;~ j \in \mathcal{J}'\}$ and $\{W_j; ~j \in \mathcal{J}'\}$ with $V_j \subseteq U_j \subseteq W_j$, for each $j\in \mathcal{J}'$. Furthermore, we suppose that for each $j \in \mathcal{J}'$, there exists a biholomorphism $F_j : W_j \rightarrow B_3 \subseteq \CC^n$ satisfies, 
    \begin{align}
         B_1 \cong F_j(V_j) \subseteq B_{3/2} \subseteq\overline{B}_2 \subseteq F_j(U_j)\subseteq F_j(\overline{U}_j) \subseteq F_j (W_j) \cong B_3.
    \end{align}

    By the construction of open covering, the transition map between $W_{j_1}$ and $W_{j_2}$, $j_1, j_2 \in \mathcal{J}$, satisfies $A^{-1} \leq |DT_{j_1, j_2}| \leq A$.
    Since only finitely many open sets are added to $\{U_j;\ j \in \mathcal{J} \}$ and the full covering, $\{U_j;\ j\in \mathcal{J}' \}$, remains locally finite, the transformation maps between any pair $U_{j_1}, \ U_{j_2} \in \{U_j;~ j \in \mathcal{J}'\}$, $T_{j_1, j_2}$, also satisfy a uniform bound on their differential. By increasing the constant $A$ if necessary, we may assume $$A^{-1} \leq |DT_{j_1, j_2}| \leq A,\hspace{0.8cm} j_1, j_2 \in \mathcal{J}'.$$
    
    The decay condition, $|\omega-\omega_0|_{g_0} = O(r^{-\tau})$, implies $|(F_j^{-1})^*(\omega|_{W_j})|_{g_0} \leq C$, with a uniform constant $C$ independent of $j \in \mathcal{J}'$. Then, there is a local potential $\rho_j$ on $U_j$ such that 
    \begin{align*} 
        |\rho_j|,\quad |d\rho_j|_{g_0} \leq C, \hspace{0.6cm} \text{ for } j \in \mathcal{J}',
    \end{align*}
    where $C$ is a uniform constant.

    For the product manifold, $X \times D$, there are countably many, locally finite coverings of $X \times D$,  $\{\widetilde{U}_j = U_j \times D,\ j \in \mathcal{J}'\}$ and $\{\widetilde{W}_j = W_j \times D,\ j \in \mathcal{J}'\}$. 
    The transformation map between $\widetilde{U}_{j_1}$ and $ \widetilde{U}_{j_2}$, $\widetilde{T}_{j_1, j_2} = T_{j_1, j_2} \times \id$, also satisfies $ A^{-1} \leq |D \widetilde{T}_{j_1, j_2}| \leq A$. Let $k$ be a sufficiently large constant and $D_{1/k} = \{ z\in D; ~ d_{g_{0}}(z, \partial D)  > 1/k\}$. Then, $\widetilde{V}_j = V_j \times D_{1/k}$ is a proper subset of $\widetilde{U}_j$ and $\{\widetilde{V}_j; ~ j \in \mathcal{J}'\}$ forms an open covering of $X \times D_{1/k}$. 
    Let $\tilde{\rho}_j= p^* \rho_j$. Then $\tilde{\rho}_j$ is the local potential of $\Omega = p^*\omega$ on $\widetilde{U}_j = U_j\times D$ for each $j \in \mathcal{J}'$, and satisfies the estimates: $
        |\tilde{\rho}_j|$, $|d\tilde{\rho}_j|_{g_0} \leq C.$
    
    Since $\varphi \in \Pl_{\Omega} (X \times D)$, $ u_j = \varphi + \tilde{\rho}_j$ is plurisubharmonic in $\widetilde{W}_j$. By $u_{j, \delta}$, we denote the regularization of $u_j$ by taking convolution with $\eta_\delta$. Assuming that $\delta$ is small enough ($\delta \leq 1/k$), the function $u_{j, \delta}$ is smooth on the domain $\{\widetilde{V}_j\}$. In the following, we write $u^{ j_2 }_{j_1,\delta}$ to be the regularization of $u_{j_1}$ in $\widetilde{V}_{j_2}$; precisely, $u^{ j_2 }_{j_1,\delta} = (u_{j_1} \circ T_{j_1, j_2}^{-1})_{\delta} \circ T_{j_1, j_2}$. Then, in $U_{j_1} \cap U_{j_2}$,
    \begin{align*}
        u_{j_1, \delta}- u_{j_2, \delta} = (u_{j_1, \delta} - u^{j_2}_{j_1, \delta}) + (u_{j_1} - u_{j_2})_{\delta}^{j_2}.
    \end{align*}
    By Lemma \ref{BKlemmacompare}, we have 
    \begin{align}\label{uniformestiintersection}
    \begin{split}
        &|u_{j_1, \delta} - u^{j_2}_{j_1, \delta}| \leq  \frac{1 + \log A}{\log \delta^{-1}} C,\\
        &|(u_{j_1} - u_{j_2})^{j_2}_\delta - (\tilde{\rho}_{j_1}-\tilde{\rho}_{j_2})| \leq C\delta.
    \end{split}
    \end{align}
    Hence, $(u_{j_1, \delta}-\tilde{\rho}_{j_1})-(u_{j_2, \delta}-\tilde{\rho}_{j_2})$ uniformly approaches to $0$ as $\delta$ goes to $0$ (independent of $j_1, j_2 \in \mathcal{J}'$). 
    The following proof goes the same as Blocki-Kolodziej [Theorem 1, Theorem 2], and we briefly explain here. Let  $\chi =0$ in $B_{3/2}$ and $\chi = -1$ in $B_3 - \overline{B}_2$. And we can assume $dd^c \chi \geq -C \omega$.
    Then, we pick a smooth function $\chi_j$ on $U_j$, for each $j \in \mathcal{J}'$, by $\chi_j = F^*_j(\chi|_{U_j})$. Then we define a smooth function in $\widetilde{U}_j$ by $\widetilde{\chi}_j = p^* \chi_j$. There is a uniform constant $C$ such that $dd^c \widetilde{\chi}_j \geq -C \Omega$. For a sufficiently small $\varepsilon >0$, we define
    \begin{align*}
    \varphi_\delta = \mathop{\mathrm{reg}\max}\limits_{j \in \mathcal{J}'}\left\{ u_{j, \delta} - \tilde{\rho}_{j} + \frac{\varepsilon\widetilde{\chi}_j}{C} \right\}
    \end{align*}
    By (\ref{uniformestiintersection}), if $\delta$ is sufficiently small, the values on the sets $\{\widetilde{\chi}_j = -1\}$ do not contribute to the regularized maximum. Hence, we obtain $\varphi_\delta \in \Pl_{(1+ \varepsilon)\Omega} (X \times D_{\delta})\bigcap \C^\infty$ and $\varphi_\delta$ decreases to $\varphi$ as $\delta \rightarrow 0$. 
    Choose $\varepsilon_q \rightarrow 0$ and $\varphi_q \in \Pl_{(1+ \varepsilon_q)\Omega}(X \times D_{1/q})\bigcap \C^\infty$. Without loss of generality, we assume $\varphi_q$ is negative. Let $\psi_{q} = \varphi_q/(1+ \varepsilon_q)$. Then $\psi_{q}\in \Pl_{\Omega} (X \times D_{1/q})\bigcap \C^\infty$ decreases to $\varphi$ as $q \rightarrow \infty$.
\end{proof}

\subsection{The continuity of the upper envelopes}  \label{subsecregupperenve}
Let $\Phi_{\Omega, \Psi}$ be the upper envelope of $\B_{\Omega,\Psi}$,
\begin{align*}
    \Phi_{\Omega, \Psi} (x,\tau) = \sup_{u \in \B_{\Omega, \Psi}} u (x, \tau), \qquad (x, \tau) \in X \times D.
\end{align*}
Let $D$ be a domain in $\CC$. A point $\tau_0 \in \partial D$ is called strongly regular if there exists a local \textit{barrier function} $w$ at $\tau_0$. Precisely, a local barrier function at $\tau_0$, $w_{\tau_0}$, is a continuous superharmonic function defined in a neighborhood of $\tau_0$ in $\overline{D}$, $ B_\delta(\tau_0 )\cap \overline{D}$, satisfying $w(\tau_0) =0$ and $w(\tau)>0$ if $\tau \ne \tau_0$. In Bremermann \cite{Bremermann1959Diri} and Walsh \cite{Walsh1968Continuity},  the upper envelope is proved to be continuous for a bounded pseudoconvex domain in $\CC^n$ if the boundary function is continuous. The following proposition extends the continuity result to the product manifold $X \times D$,

\begin{pro}\label{contiofupperenve}
Let $X$ be an ALE K\"ahler manifold, and $D\subseteq \CC$, a bounded domain with strongly regular boundary. Suppose that $\Psi$ is a bounded uniformly continuous function on $X \times \partial D$ and $\psi_\tau 
 (\cdot)= \Psi(\cdot, \tau)$ is $\omega$-psh function for each $\tau \in \partial D$. Then the upper envelope $\Phi_{\Omega, \Psi}$ is bounded and continuous in $X \times \overline{D}$ with
\begin{align*}
    \Phi_{\Omega, \Psi} (\cdot, \tau) = \psi_{\tau}(\cdot), \qquad \tau \in \partial D
\end{align*}
Furthermore, under the assumption, the upper envelopes of $\B_{\Omega, \Psi}$ and $\F_{\Omega, \Psi}$ coincide,
\begin{align} \label{coincideenves}
    \Phi_{\Omega, \Psi} (x, \tau) = \sup_{u \in \F_{\Omega, \Psi} } u(x, \tau), \qquad (x, \tau) \in X \times D
\end{align}
\end{pro}
\begin{proof} It is obvious to see that $\Phi_{\Omega, \Psi}$ is bounded in $X \times D$ and (\ref{coincideenves}) directly follows from the continuity of $\Phi_{\Omega, \Psi}$. In this proof, we only need to show $\Phi_{\Omega, \Psi}$ is continuous.

    The first step is to prove that $\Phi_{\Omega, \Psi}$ is uniformly continuous to $\Psi$ on $X \times \partial D$. In other words, we need to prove that for $\varepsilon >0$, there exists a $\delta>0$ such that for any $(x,\tau) \in X \times \partial {D}$ and $B'_{\delta}(x,\tau) = \{(x',\tau') \in X \times \overline{D}; d_{\Theta} ((x,\tau), (x', \tau')) \leq \delta\}$, we have
    \begin{align} \label{uniconofbdry}
            \big|\Phi_{\Omega, \Psi}(x', \tau') - \Psi(x, \tau)\big| \leq \varepsilon, \qquad (x,\tau) \in X\times \partial D, \quad (x', \tau') \in B'_\delta (x, \tau).
    \end{align}
   By virtue of the uniform continuity of $\Psi$, for each $\varepsilon >0$, there exists $\delta >0$ such that 
   \begin{align*}
   |\Psi(x_1, \tau_1)- \Psi(x_2,\tau_2)| \leq \varepsilon, \qquad (x_1,\tau_1), \ (x_2, \tau_2) \in X \times \partial D,
   \end{align*}
   if $d_\Theta ((x_1,\tau_1), (x_2, \tau_2)) < \delta$. Let $M$ be the $L^\infty$ bound of $ \Psi$. Fixing a point $ \tau_0 \in  \partial D$, let $w (\tau)$ be a barrier function at $\tau_0$ satisfying
    \begin{align*}
        w(\tau_0) =0, \qquad w(\tau) >0 \text{ if }  \tau \in \overline{D}-\{\tau\},
    \end{align*}
    and that $w$ is continuous super-harmonic.
    There exists a constant $k$ such that $k w(\tau) \geq 2M$ for $|\tau-\tau'| >\delta$. Consider the following functions in $X \times D$,
    \begin{align} \label{barrierfunctions}
       u_- (x, \tau) = \Psi(x,\tau_0) -\varepsilon- k w(\tau), \qquad u_+(x, \tau) = \Psi (x, \tau_0) + \varepsilon + k w (\tau)
    \end{align}
    
    To prove $u_- \leq \Phi_{\Omega, \Psi}$ in $X \times D$, notice that $-k w$ is a continuous subharmonic function in $D$. It is observed from the uniform continuity of $\Psi$ and the choice of $k$ in (\ref{barrierfunctions}) that $u_-(x, \tau) \leq \Psi(x, \tau)$ for $\tau \in \partial D$. Hence $u_- \in \B_{\Omega, \Psi}$. By definition of $\Phi_{\Omega, \Psi}$, we have $u_- \leq \Phi_{\Omega, \Psi}$ in $X \times D$. Hence, we have
    \begin{align*}
        \Psi(x, \tau_0) -\varepsilon \leq \Phi_{\Omega, \psi} (x, \tau_0),
    \end{align*}
    and by shrinking $\delta$ further (only depending on $\varepsilon$), for any $(x',\tau') \in B'_\delta(x,\tau_0)$, we have,
    \begin{align} \label{pfofuniconbdry1}
        \Phi_{\Omega, \Psi}(x',\tau')\geq \Psi(x', \tau_0) -2\varepsilon \geq \Psi(x, \tau_0) -3\varepsilon.
    \end{align}
    
    For each function $u \in \B_{\Omega, \Psi}$, notice that $\displaystyle u_+(x,\tau) \geq \Psi(x,\tau) \geq \limsup_{(x',\tau') \rightarrow (x, \tau)} u(x, \tau)$ for $(x,\tau) \in X \times \partial D$.  
    It is observed that $u$ is a subharmonic function restricted to each slice $\{x\} \times D$, and $u_{+}$ is a continuous superharmonic function on $\{x\} \times D$. The fundamental property of subharmonic functions implies that $u_{+}(x, \tau) \geq u(x,\tau)$ for $\tau \in D$. Then, for any $(x', \tau') \in B'_\delta(x, \tau_0)$, we have
    \begin{align}\label{pfofuniconbdry2}
    \begin{split}
        \Phi_{\Omega, \Psi} (x',\tau') = \sup_{u \in \B_{\Omega, \Psi}} u(x', \tau') &\leq u_+ (x', \tau') \\
        &\Psi(x', \tau_0) + 2\varepsilon \leq \Psi(x, \tau_0) + 3\varepsilon. 
    \end{split}
    \end{align}
    In conclusion, (\ref{pfofuniconbdry1}) and (\ref{pfofuniconbdry2}) imply the uniform continuity on boundary (\ref{uniconofbdry}).

    The second step is to show that $\Phi_{\Omega, \Psi}$ is continuous on $X \times \overline{D}$.  Let 
    \begin{align*}
        \big(\Phi_{\Omega, \Psi}\big)^*(x,\tau) = \limsup_{(x', \tau')\rightarrow (x, \tau)} \Phi_{\Omega, \Psi} (x', \tau'), \qquad (x,\tau),\ (x',\tau') \in X \times D.
    \end{align*}
    Then, $(\Phi_{\Omega, \Psi})^* $ is an $\Omega$-plurisubharmonic function with the boundary condition (\ref{uniconofbdry}); hence, $(\Phi_{\Omega, \Psi})^* \in \B_{\Omega, \Psi}$. We conclude that $ \Phi_{\Omega, \Psi}=(\Phi_{\Omega, \Psi})^*$, and $\Phi_{\Omega, \Psi}$ is upper-semicontinuous. 
    
    It suffices to prove that $\Phi_{\Omega,\Psi}$ is lower-semicontinuous. The inequalities, (\ref{pfofuniconbdry1}) and (\ref{pfofuniconbdry2}), imply that $\Psi_{\Omega, \Psi}$ is uniformly continuous to $\Psi$ on the boundary (see (\ref{uniconofbdry})).
    According to proposition \ref{regpsh} there exists a sequence of smooth $\Omega$-psh functions, $\{\varphi_k\}$, decreasing to $\Phi_{\Omega, \Psi}$. The uniform continuity of $\Phi_{\Omega, \Psi}$ on the boundary implies that by choosing $k$ large enough, we have that 
    \begin{align*}
        \Phi_{\Omega, \Psi} (x, \tau) \leq \varphi_{k} (x, \tau) \leq  \Phi_{\Omega, \Psi} (x, \tau) + 2\varepsilon,
    \end{align*}
    for $(x, \tau) \in X \times (D_{1/k} - D_{\delta/2})$. Consider the function 
    \begin{align*}
        \tilde{\varphi}_k (x, \tau) = 
        \begin{cases} \max \{\Phi_{\Omega, \Psi} (x, \tau), \varphi_k (x,\tau) -2\varepsilon\}, \quad & (x, \tau) \in X \times D_{1/k} \\
        \Phi_{\Omega, \Psi}(x,\tau), & (x,\tau) \in X \times (D - D_{1/k}). 
        \end{cases}
    \end{align*}
    Then, $\tilde{\varphi}_k $ is an $\Omega$-psh function defined in $X \times D$ satisfying the boundary condition $\tilde{\varphi}_k|_{X \times \partial D}  = \Psi $. Hence, we have that $\tilde{\varphi}_k  \leq \Phi_{\Omega, \Psi}$. For any $(x, \tau) \in X \times D_{\delta/2} \subseteq X \times D_{1/k}$, we have
    \begin{align*}
        \Phi_{\Omega, \Psi} (x, \tau)- 2\varepsilon \leq \varphi_k(x,\tau) -2\varepsilon \leq \Phi_{\Omega, \Psi} (x, \tau) 
    \end{align*}
    Since $\varphi_k$ is continuous at $(x, \tau)$, there is a small neighborhood $B_{\theta} (x, \tau) \subseteq X \times D_{1/k}$ such that $|\varphi_k (x', \tau') - \varphi_k (x, \tau)| \leq \varepsilon$, for $(x', \tau' )\in B_\theta (x, \tau)$. Hence, we have
    \begin{align} \label{contiinside}
        \Phi_{\Omega, \Psi} (x, \tau) -3 \varepsilon \leq \varphi_k (x, \tau)-3\varepsilon \leq \varphi_{k} (x', \tau') -2\varepsilon \leq \Phi_{\Omega, \Psi} (x', \tau'), \quad (x', \tau') \in B_\theta (x, \tau).
    \end{align}
    For any $(x, \tau) \in X \times (D-D_{\delta/2})$, there exists a neighborhood $B'_{\delta/2} (x, \tau) $ centered at $(x, \tau)$ and $(x, \tau_0) \in X \times \partial D$ such that $B'_{\delta/2} (x, \tau) \subseteq B'_{\delta} (x, \tau_0)$, . Based on (\ref{uniconofbdry}), we have that
    \begin{align} \label{continearbdry}
        \Phi_{\Omega, \Psi} (x, \tau) - 2\varepsilon \leq  \Psi(x,\tau_0) -\varepsilon \leq \Phi_{\Omega, \Psi} (x', \tau'), \quad (x', \tau') \in B'_{\delta/2} (x, \tau).
    \end{align}
    (\ref{contiinside}) and (\ref{continearbdry}) indicate the lower semi-continuity of $\Phi_{\Omega, \Psi}$.
    \end{proof}

\subsection{The generalized solution of Dirichlet problem on $X \times D$} \label{subsecgeneralsolu}

In Bedford-Taylor \cite{bedford1976dirchlet}, the existence of the generalized solution to the Monge-Amp\`er{e} equations was proved for the bounded domain in $\CC^n$. In this subsection, we generalize the result to the product space of an ALE K\"ahler manifold and a bounded domain in $\CC^n$. By Proposition \ref{contiofupperenve}, if we assume the boundary is strongly regular, the existence of the generalized solution can be proved without additional assumptions.

\begin{thm} \label{thmgsolHCMAupenv}
    Let $D$ be a bounded domain in $\CC$. Let $\Psi $ be a bounded uniformly continuous function on $X \times \partial D$ and $\psi_{\tau}(\cdot)$ is $\omega$-psh function for each $\tau \in\partial D$. If 
    \begin{enumerate}
        \item[(i)] $\Pl_{\Omega, \Psi} (X \times D)$ is nonempty,
        \item[(ii)] the upper envelope of $\Pl_{\Omega, \Psi}(X \times D)$, $\Phi$, is continuous with $\Phi_{\Omega, \Psi} = \Psi$ on $\partial \Omega$, 
    \end{enumerate}
     then $\Phi$ is the unique bounded continuous solution to the Dirichlet problem,
        \begin{align*}
            &(\Omega + dd^c \Phi)^{n+1} = 0, \quad \text{in } X \times D,\\
            &\Phi \in \Pl_{\Omega} (\Omega) \cap C(X \times \overline{D}), \\
            & \Phi = \Psi, \qquad \text{on } X \times \partial D.
        \end{align*}
        Furthermore, if $D$ is a bounded domain with strongly regular boundaries, then the upper envelope $\Phi$ is the unique bounded continuous solution to the above Dirichlet problem.
\end{thm}

\begin{proof} The uniqueness of the bounded continuous solution follows from the maximum principle, Theorem \ref{intromainmp}, whose proof will be completed in the next section. To prove the existence, we only have to prove $(\Omega +dd^c \Phi_{\Omega, \Psi})^{n+1} =0$ in $X \times D$, where $\Phi_{\Omega,\Psi}$ is the upper envelope of $\B_{\Omega, \Psi}$. Fixing $p \in X \times D$, there exists a local holomorphic chart, $\widetilde{U}$, in $X \times D$ around $p$. Choose a small positive $\varepsilon>0$ such that the Euclidean ball, $B_\varepsilon (p)$, has its closure contained in $\widetilde{U}$. It's the classic result by Bedfore-Taylor \cite{bedford1976dirchlet} that the continuous solution exists for Monge-Amp\`er{e} equation in Euclidean balls with continuous boundary conditions. Precisely, a continuous solution to the following equation exists,
    \begin{align*}
        &(\Omega + dd^c v )^{n+1} = 0, \quad \text{on } B_\varepsilon (p), \\
        &v = \Phi_{\Omega, \Psi},\qquad \qquad \quad \text{on } \partial B_\varepsilon(p),\\
        & v \in \Pl_\Omega(B_\varepsilon(p)).
    \end{align*}
    Since $\Phi_{\Omega, \Psi}$ is a continuous $\Omega$-psh function, we have $(\Omega +dd^c\Phi_{\Omega, \Psi})^{n+1} \geq 0 = (\Omega + dd^c v)^{n+1} $. By the maximum principle of Monge-Amp\`{e}re operator on $B_{\varepsilon} (p)$ (see B{\l}ocki \cite{blocki2009geodesics}), we have that $\Phi \leq v$. Set
    \begin{align*}
        \Phi^*(z) = \begin{cases}
            v(z),\quad & \text{on } B_\varepsilon(p) \\
            \Phi_{\Omega, \Psi}(z), & \text{on } (X \times D) \backslash B_\varepsilon(p).
        \end{cases}
    \end{align*}
    Then, $\Phi^*$ is bounded continuous psh function on $X \times D$ such that $\Phi^* (z) = \Psi_{\Omega, \Psi}(z)$ on $X \times \partial D $. Then, $\Phi^* \leq \Phi_{\Omega, \Psi}$ and $\Phi_{\Omega, \Psi} = v$ on $B_\varepsilon(p)$. Hence, we complete the proof. 

    In the case of $D$ with strongly regular boundaries, the conditions (i) and (ii) are satisfied due to Proposition \ref{contiofupperenve}.
\end{proof}

\begin{rem} \label{remgsolHCMAupenvcptcase}
Now assume that \( X \) is a compact K\"ahler manifold without boundary, and let \( D \subseteq \mathbb{C} \) be a bounded domain with strongly regular boundary. Suppose further that \( \Psi \in \mathcal{C}^0(X \times \partial D) \). Since \( X \) is compact, \( \Psi \) is in particular uniformly continuous and bounded.

In this setting, the HCMA equation on \( X \times D \) with boundary data given by \( \Psi \) admits a bounded continuous solution, which can be obtained as the upper envelope of bounded \( \Omega \)-plurisubharmonic subsolutions. The argument proceeds as follows: first, using the same method as in the proof of Proposition~\ref{contiofupperenve}, one shows that the upper envelope is continuous; then, by following the proof of Theorem~\ref{thmgsolHCMAupenv}, one concludes that the envelope solves the HCMA equation in the weak sense.
\end{rem}

\section{Maximal Principle} \label{secMPofMAALE}

The uniqueness parts of Theorem \ref{mthmasymbehHCMA} and Theorem \ref{mthmHCMAc11} follow immediately from the maximal principle, Theorem \ref{intromainmp}, applied in the setting of ALE K\"ahler manifolds. This section is dedicated to completing the proof of Theorem \ref{intromainmp}.
Let $X$ be an ALE K\"ahler manifold with a fixing ALE K\"ahler metric $\omega$ and $\Omega= p^* \omega$, the pullback form of $\omega$ on $X \times D$. Consider the class of bounded continuous $\Omega$-psh functions; precisely
$$\F_{\Omega} = \{u \in \Pl_\Omega (X \times D) \bigcap \C(X \times D); ~ u \text{ is bouned on } X \times D \}.$$
Each function $u \in \F_{\Omega}$ defines a positive current, $\Omega + dd^c u$, on $X \times D$. The complex Monge-Amp\`{e}re operator associated with a function $u \in \F_\Omega$, namely $(\Omega +dd^c u)^{n+1}$, is also well-defined in the sense of currents. Consider a smooth test $(p,p)$-form $v$ with a compact support $K \subseteq X \times D$. If $T$ is a closed positive current, $ dd^c u \wedge T$ is defined by
$$ \int_K dd^c u \wedge T \wedge v =  \int_K uT \wedge dd^c v. $$
This inductive procedure defines the complex Monge-Amp\`{e}re operator $(\Omega +dd^c u)^{n+1}$. The theory of positive currents and the generalized complex Monge-Amp\`{e}re operators has been 
extensively developed in tons of literature (for instance, see \cite{bedford1976dirchlet, bedford1976capacity, demailly1997complex, Klimek2023pluri}). In the following, we recall two useful results that will be applied in context of ALE K\"ahler manifolds. 

The generalized complex Monge-Amp\`{e}re operator extends the classic definition on $\F_\Omega \cap \C^\infty $ to bounded $\Omega$-psh functions via approximation by smooth functions. Let $\{u_1^k\}_{k \in\NN}, \ldots, \{u_{p}^k\}_{k \in \NN}$ be decreasing sequences in $\F_\Omega \cap \C^\infty (X \times D)$ and let $u_1, \ldots, u_{p} \in \F_\Omega$ such that $\displaystyle \lim_{k} u^k_j = u_j $ for $j=1, \ldots, p$. Then, we have
$$\big(\Omega+dd^c u_1^k\big) \wedge \ldots \wedge \big( \Omega + dd^c u_p^k \big)  \rightarrow  \big( \Omega + dd^c u_1 \big)  \wedge \ldots \wedge  \big( \Omega +  dd^c u_p \big),$$
in the sense of currents. Another useful result is the known Chern-Levine-Nirenberg estimates, which provide uniform control on wedge products of positive currents associated with plurisubharmonic functions. Here, we state the estimate in local coordinates. Let $K$ be a compact set in $\CC^{n+1}$ with an open neighborhood $K\subseteq U$, and let $u_1, \ldots u_{n+1}$ be psh functions on $U$. Then, we have
$$ \int_K dd^c u_1 \wedge \ldots \wedge dd^c u_{n+1} \leq C ||u_1||_{L^\infty (U)} \ldots ||u_{n+1}||_{L^\infty (U)},$$
where $C$ is a uniform constant depending only on $n$ and the geometry of $U \setminus K$, and can be explicitly bounded by
$$C = C (n, U\backslash K) \int_U \omega_{euc}^{n+1}.$$

The maximum principle for the generalized complex Monge-Amp\`{e}re operator was proved by Bedford-Taylor \cite{bedford1976dirchlet} in the setting of bounded domains in $\CC^n$. A generalization of the maximal principle to compact K\"ahler manifolds with boundary was developed by Blocki in \cite{blocki2009geodesics}. In Yao~\cite{yao2024geodesicequationsasymptoticallylocally}, a version of the maximal principle has been proved on $X \times D$, under the assumption of uniform positivity of the forms and at least $\C^2$-ragularity of the functions involved. In this section, we establish a version of the maximum principle that removes both the uniform positivity and $\C^2$-regularity assumptions.

The proof of the maximal principle is nontrivial compared to the compact version due to the following simple example:
consider functions $u \equiv 0$ and $v= 1.1- |\tau|^2$ defined on $X \times \DD$, where $\DD$ is the unit disc centered at the origin in $\CC$ and $\tau$ is the complex coordinate of $\DD$. It can be checked that $(\partial \overline{\partial} u )^{n+1} = (\partial \overline{\partial} v)^{n+1} =0$ and $v > u $ on $X \times \partial \DD$, but, $v \geq u$ does not hold on $X \times \DD$. However, the maximum principle still holds under the assumption that the reference K\"ahler form is positive in the directions tangent to $X$. In our setting, this is naturally satisfied by considering the pullback of the reference ALE K\"ahler form, $\Omega$, on $X \times D$.

Let $\Theta$ be the K\"ahler form induced by the standard product metric on $X \times D$, given by $\Theta = \Omega + i d\tau \wedge d\bar{\tau}$.
The main theorem of this section is the following:

\begin{thm}\label{mainthmmaxp}
    Let $X$ be an ALE K\"ahler manifold, and let $D$ be a bounded domain in $\CC$. Let  $\Omega$ be the pullback of the reference K\"ahler form on $X$ to $X \times D$. Suppose that $u$, $v$ are bounded continuous $\Omega$-psh functions on $X \times D$ such that for some constant $E_0>0$,
    \begin{align} \label{conmpbddvolume}
     (\Omega+ dd^c v)^{n+1} \leq E_0 \Theta^{n+1}
    \end{align}
    in the sense of currents, where $ \Theta = \Omega + i d\tau \wedge d\bar{\tau}$ is the product K\"ahler form on $X \times D$. Assume further that
    $$ (\Omega + dd^c v)^{n+1} \leq  (\Omega + dd^c u)^{n+1} \hspace{0.6cm}  \text{ on } X \times D,$$ 
    and $v \geq u$ on $ X \times \partial D$.
    Then, we have $v\geq u$ on $X \times D$.
\end{thm}

Before proving Theorem \ref{mainthmmaxp}, we recall the following two lemmas. One is the comparison theorem originally proved by Bedford–Taylor \cite{bedford1976dirchlet}. Another one is Wu-Yau's maximal principle (see \cite{cheng1980existence, wu2008kahler}) on open manifolds. Here, we prove a slightly refined version of Wu-Yau's maximal principle applicable to functions with weak regularity on open manifolds in our setting.

    \begin{lem} \label{intcompare}
    Let $U$ be a bounded domain with smooth boundary in $\CC^{n+1}$ and let $u,v$ be bounded psh functions defined on $U$. Suppose that for each $w\in \partial U$
    \begin{align*}
        \liminf_{\substack{z\rightarrow w\\ z \in U}} (u(z) - v(z)) \geq 0.
    \end{align*}
    Then, we have
    \begin{align*}
        \int_{\{u< v\}} (dd^c v)^{n+1} \leq \int_{\{u<v\}} (dd^c u)^{n+1}.
    \end{align*}
\end{lem}

\begin{lem}\label{WuYauweak} Let $X$ be an ALE K\"ahler manifold, and $D$, a bounded domain with smooth boundary in $\CC$. Let ${g}$ be a K\"ahler metric on $X$ and the product metric $\tilde{g}$ on $X \times D$ satisfying 
\begin{align*}
    \lambda \delta_{ij} \leq \tilde{g}_{ij} \leq \Lambda \delta_{ij}, \qquad \textup{ in } X_\infty \times D,
\end{align*}
where $X_\infty$ is the asymptotic chart of $X$, and $\lambda < \Lambda$ are two positive constants. If $u$ is a continuous function bounded from above satisfying $\sup_{X \times D} u > \sup_{X \times \partial D} u$, then there exists a sequence of points, $\{x_k\}$, in $X \times D$ such that
\begin{align}\label{WuYautosup}
    \lim_{k \rightarrow \infty} u(x_k) = \sup_{x \in X \times D} u(x), 
\end{align}
and in a small geodesic ball $B_{s} (x_k)$, we have 
\begin{align} \label{WuYauweekest}
    u(x) \leq u(x_k) + \frac{C}{k} \big(d(x, x_k)+ d(x, x_k)^2\big), \qquad x \in B_{{s}} (x_k)
\end{align}
where $C$ and $s$ are constants independent of $k$.
\end{lem}
\begin{proof}
    Let $r$ be the radial function inherited from the asymptotic chart of $X$, and smoothly extended to a positive smooth function, $ r \geq 1$, in the whole manifold, $X \times D$. Notice that the function $r$ satisfies the following estimates,
    \begin{align*}
        |\nabla \log r|_{\tilde{g}} \leq C, \quad |\nabla^2 \log r|_{\tilde{g}} \leq C, \qquad \textup{in } X \times D,
    \end{align*}
    for some uniform constant $C$. Consider the function $u_{\delta} = u- \delta \log r$. Since, for each $\delta >0$, $u_\delta$ tends to negative infinity as $r $ goes to infinity, $u_\delta$ achieves its maximum at some point $x_\delta$. If $\delta$ is sufficiently small, $x_\delta$ is an interior point in $X \times D$ based on the assumption that $\sup_{X \times D} u  > \sup_{X \times \partial D} u$. Choose a sequence $(x_k)$, $x_k = x_{\delta_k}$ such that $\delta_k\rightarrow 0$ as $k \rightarrow \infty$.  Let $s = \min\{1, \inf_{k} \dist (x_{k} , X \times  \partial D) \} >0$. 
    Then, we have $B_{s} (x_k) \subseteq X \times D$, and 
    \begin{align*}
        u(x) &\leq u(x_k) - \delta_k \big(\log (r(x_k))- \log (r(x))\big)\\
        & \leq u(x_k) + C \delta_k |r(x) - r(x_k)| + C \delta_k |r(x) - r(x_k)|^2,
    \end{align*}
    and 
    \begin{align*}
    u_{\delta_k} (x_k) \geq u(x) - \delta_k \log r(x), \qquad \textup{ for all } x \in X \times D.
    \end{align*}
    By taking limits for both sides, we have $$\sup_{x\in X \times D} u(x) \geq \lim_{k \rightarrow \infty} u(x_k) \geq \lim_{k \rightarrow \infty} u_{\delta_k}(x_k) \geq \sup_{x \in X \times D} u(x).$$
    By requiring $\delta_k \leq 1/k$, we complete the proof of Lemma \ref{WuYauweak}.
\end{proof}

\begin{rem} \label{remWYmaxp}
    From the above proof, it is clear that if one assumes $u \in \C^2 (X \times D)$, the original version of Wu-Yau's maximal principle is recovered: there exists a sequence $\{x_k\}$ in $X \times D$ such that 
    \begin{align*}
        \lim_{k\rightarrow \infty} u(x_k) = \sup_{x \in X \times D} u(x),  \qquad \lim_{k\rightarrow \infty} |du(x_k)|_{\tilde{g}} = 0, \qquad \lim_{k\rightarrow \infty} \Delta_{\tilde{g}} u(x_k) \leq 0.
    \end{align*}
    This classical version of the maximal principle will be applied in the next section in the proof of the $\C^{1,1}$ regularity.
\end{rem}

The key observation of proving Theorem \ref{mainthmmaxp} is the following:

\begin{lem} \label{MPcomputationlem}
Let $A$, $B$ and $C$ be semi-positive definite $(n+1)\times ( n+1)$ Hermitian matrices satisfying  
    \begin{align*}
        \det A > \det C, \quad \text{and} \quad  \det B \geq \det C.
    \end{align*}
    Then, for each $ \theta\in (0,1)$,
    \begin{align*}
        \det (\theta A + (1-\theta) B ) \geq \delta + \det C,
    \end{align*}
    where $\delta$ is the positive constant with $\delta >  \big(\theta (\det A - \det C)\big)^{n+1} / \big((n+1)\det A\big)^{n}$.
\end{lem}

\begin{proof} Consider the function $f(P) = (\det P)^{\frac{1}{n+1}}$, defined on the space of semi-positive definite $(n+1)\times ( n+1)$ Hermitian matrices. It is known that $f$ is a concave function. To see this, let $P$ be a positive definite Hermitian matrix, and let $Q$ be an arbitrary Hermitian matrix. The concavity follows from the following direct calculation:
$$D^2 f_P (Q, Q) = \frac{1}{n^2} f(Q) \big( \tr^2 (P^{-1} Q) - n \tr (P^{-1} Q)^2 \big) \leq 0.$$
For each $\theta \in (0,1)$, the concavity of $f$ implies that
\begin{align*}
    \big\{\det\big(\theta A + (1-\theta )B\big)\big\}^{\frac{1}{n+1}} 
    &\geq \theta (\det A )^{\frac{1}{n+1}} + (1-\theta) (\det B)^{\frac{1}{n+1}}\\
    & \geq \theta  \big\{(\det A )^{\frac{1}{n+1}}-(\det C )^{\frac{1}{n+1}} \big\} + (\det C)^{{\frac{1}{n+1}}}.
\end{align*}
Then, we have
\begin{align*}
    \det\big(\theta A + (1-\theta )B\big) \geq \delta + \det C,
\end{align*}
where $\delta> \big(\theta (\det A - \det C)\big)^{n+1} / \big((n+1)\det A\big)^{n}$.
\end{proof}

For the sake of simplicity, we introduce a new notation $\Theta_E$, given by $\Theta_E = \Omega + Ei d\tau\wedge d\bar\tau $, for an arbitrary positive constant $E$. Then, $u_E = u- E |\tau|^2$ and $v_E = v-E |\tau|^2$ are bounded continous $\Theta_E$-psh functions. If we specify $E = 2 E_0$, where $E_0$ is the constant appearing in \eqref{conmpbddvolume}, the condition \eqref{conmpbddvolume} can be interpreted as 
\begin{align*}
    (\Theta_E + dd^c v_E)^{n+1} \leq \frac{1}{2} \Theta_E^{n+1}.
\end{align*}
It suffices to prove $v_E \geq u_E $ on $X \times D$.

\begin{proof}[Proof of Theorem \ref{mainthmmaxp} in the case $ u,~ v \in \C^2$]
Assume that the theorem is false. That is, there exists $x_0 \in X \times D$ such that $u_E(x_0) > v_E (x_0) $. Since $u_E$, $v_E$ are bounded continuous functions on $X \times D$, by adding a positive constant, we may assume that $u_E$, $v_E$ are nonnegative. Then, there exists a $\theta \in (0,1)$, arbitrarily close to $1$, such that $\theta u_E  \leq v_E$ on $X \times \partial D$ and $\theta u_E (x_0) > v_E (x_0)$. If we apply the maximal principle, Lemma \ref{WuYauweak}, to $ \theta u_E - v_E$, there exists a sequence $\{x_k\}$ in $X \times D$ satisfying \eqref{WuYautosup} and \eqref{WuYauweekest}. Let $h_k = \theta u_E (x_k) - v_E (x_k) >0$.  Fixing a sufficiently large $k$, we locally modify $\theta u_E$ and $v_E$ in the ball $B_s(x_k)$ as follows:
\begin{align}\label{pfmodfpotens}
\begin{split}
    \tilde{u} &= \theta u_E - \delta |z|^2,\\
    \tilde{v} &= v_E + \tilde{h},
\end{split}
\end{align}
where $z$ represents the local holomorphic coordinates in $B_s (x_k)$ centered at $x_k$. By carefully choosing the coordinates (see Lemma \ref{lemholocornearinfty}), we may assume that $\lambda dd^c |z|^2 \leq \Theta \leq \Lambda dd^c |z|^2$ for some constant $0<\lambda<\Lambda $ independent of $k$, and hence, $\lambda |z| \leq d(x, x_k) \leq \Lambda |z|$ in $B_s (x_k)$. The constants $\delta $ and $\tilde{h}$ are explicitly given by: 
\begin{align*}
    \delta = \frac{8\Lambda^2 }{s}\frac{C}{k}, \quad \textup{and} \quad \tilde{h}  = h_k- s \frac{C}{k}. 
\end{align*}
By Lemma \ref{WuYauweak}, it is straightforward to verify that $ \tilde{u} (x_k) > \tilde{v} (x_k)$, and $\tilde{u} (x) \leq \tilde{v} (x) $ for $ \displaystyle x \in B_{{s}}(x_k) \backslash B_{\frac{s}{2}} (x_k)$. Note that
\begin{align*}
    \Theta_E + dd^c \tilde{u} = \theta (\Theta_E + dd^c u_E) + (1-\theta) \Theta_E - \delta dd^c |z|^2. 
\end{align*}
If we choose $k$ sufficiently large such that $\delta \leq \frac{ (1-\theta)\lambda}{3(n+1)}$, then by direct computation, we have 
\begin{align*}
    (1-\theta)\Theta_E -\delta dd^c |z|^2 >0, 
\end{align*}
and 
\begin{align} \label{pfcompareposipart}
    \big((1-\theta)\Theta_E - \delta dd^c |z|^2\big)^{n+1} > \frac{4}{3} (1-\theta)^{n+1} E_0\cdot\Theta^{n+1}.
\end{align}
In the local chart $B_s (x_k)$, $\Theta_E$ admits a local potential, $\Theta_E = dd^c \rho_E$.  According to Lemma \ref{intcompare}, we can compare the following integrals in $B_s (x_k)$:
\begin{align} \label{pfcompareintegral}
    \int_{\{\tilde{u} > \tilde{v}\}} \big( dd^c (\rho_E + \tilde{v})\big)^{n+1} \geq \int_{\{\tilde{u} > \tilde{v}\}} \big( dd^c (\rho_E + \tilde{u})\big)^{n+1}.
\end{align}

It suffices to derive a contradiction from the integral inequality above. If we assume $u,~ v \in \C^2$, then $\Theta_E + dd^c u_E$ and $\Theta_E + dd^c v_E$ can be regarded as families of semi-positive Hermitian matrices which vary continuously on $B_s (x_k)$. We denote by $A_u$ and $A_v$ the Hermitian matrices associated with $\Theta_E + dd^c u_E $ and $\Theta_E + dd^c v_E$ and by $A_\delta$ the Hermitian matrix associated with $\Theta_E - \delta dd^c |z|^2 $.
By the assumption of Theorem \ref{mainthmmaxp}, we have $\det (A_u) \geq \det( A_v)$, and by \ref{pfcompareposipart}, $ \det( A_{\delta/(1-\theta)}) > \det( A_v)$.
By \eqref{pfcompareintegral} and Lemma \ref{MPcomputationlem}, we have 
\begin{align*}
    \int_{\{\tilde{u} > \tilde{v}\}} \big( dd^c (\rho_E + \tilde{v})\big)^{n+1} 
    &\geq \int_{\{\tilde{u} > \tilde{v}\}} \big( dd^c (\rho_E + \tilde{u})\big)^{n+1}\\
    &\geq \int_{\{\tilde{u} > \tilde{v}\}} \Big\{ \theta (\Theta_E  + dd^c u_E) + (1-\theta) \Big(\Theta_E - \frac{\delta}{1-\theta} dd^c |z|^2\Big) \Big\}^{n+1}\\
    & = \int_{\{\tilde{u} > \tilde{v}\}}\det \big\{ \theta A_u + (1-\theta) A_{\delta/ (1-\theta)} \big\} d\mu(z)  \\
    &\geq \int_{\{\tilde{u} > \tilde{v}\}} \delta' d\mu(z) + \int_{\{\tilde{u} > \tilde{v}\}} \big( dd^c (\rho_E + \tilde{v})\big)^{n+1},
\end{align*}
where $d\mu(z)$ is the measure associated with the standard Euclidean metric in $B_s(x_k)$ and $\delta'$ is the positive constant depending only on $(1-\theta)$, $E_0$. Therefore, we complete the proof under the additional assumption that $u, ~ v \in \C^2$.
\end{proof}

The next lemma, due to Blocki \cite[Theorem 3.10]{Blocki1996} and further generalized by Guedj-Lu-Zeriahi \cite{Guedj2019weaksub}, is particularly useful for extending the above proof to the weak setting. Before stating the lemma, we recall some notation. Let $\chi$ be the standard smoothing kernel defined in $\CC^{n+1}$ and let $\chi_{\varepsilon} (z) = \varepsilon^{-2(n+1)} \chi(\varepsilon z)$. We define the regularization $u_\varepsilon = u \star \chi_\varepsilon$, and the operator $\Delta_H$ as
\begin{align*}
    \Delta_H  = \frac{1}{n+1} \sum_{j,k} h_{jk} \frac{\partial^2 }{\partial z^j \partial \bar{z}^k},
\end{align*}
for some positive Hemitian matrices $H$ with $\det H =1$. We denote $\mathcal{H} = \{H \in \text{Herm}^+: \det H =1 \}$ the space of all positive Hermitian matrices with determinant one.

\begin{lem} \label{lemregu}
    Let $ u$ be a bounded psh function in a bounded domain $N \subseteq \CC^{n+1} $ and let $f $ be a nonnegative $L^1$ function in $N$. Then, the following are equivalent:
    \begin{itemize}
        \item[(i)]  $(dd^c u)^{n+1} \geq f d\mu$ in the sense of measure. 
        \item[(ii)]  $\Delta_H u \geq f^{1/(n+1)}$ in the sense of distribution, for all $H \in \mathcal{H}$.
        \item[(iii)]  $ \big(\det (u_\varepsilon)_{i\bar{j}}\big)^{1/(n+1)} \geq \big(f^{1/(n+1)}\big)_\varepsilon$ in the classic sense.
    \end{itemize}
\end{lem}

\begin{proof}
    The full proof is parallel to Guedj–Lu–Zeriahi~\cite[Section 3]{Guedj2019weaksub}; we briefly recall the main idea for the reader's convenience. The implication (ii) \(\Rightarrow\) (iii) follows by taking convolution with the standard smoothing kernel, $(\Delta_H u ) \star \chi_\varepsilon = \Delta_H(u \star \chi_\varepsilon) \geq \big(f^{1/(n+1)}\big)_\varepsilon $, along with the following well-known observation due to Gaveau~\cite{gaveau1977methodes}: 
    \begin{align}\label{keyequsmoothMAop}
    \det u_{i\bar{j}} (x) = \inf_{H \in \HH} \Delta_H u (x), \qquad \text{for } u \in \C^2(N) \text{ and } x \in N.
    \end{align}
    For (iii) \(\Rightarrow\) (i), we observe that $(dd^c u_\varepsilon)^{n+1} $ converges to $(dd^c u)^{n+1}$ in the sense of currents and $\big(f^{1/(n+1)}\big)_\varepsilon$ converges to $f^{1/(n+1)}$ in $L^{n+1}$.

The proof of (i) \( \Rightarrow \) (ii) is more delicate; we sketch the main idea below. In the case where $f$ is continuous on $N$, the result follows from the equivalence between the weak subsolution to $(dd^c u)^{n+1} = f d\mu$ and the viscosity subsolution (see  Eyssidieux-Guedj-Zeriahi \cite{Eyssidieux2011viscosity}). In this setting, the weak inequality in the pluripotential sense implies the viscosity inequality, which then yields the pointwise inequality in the classic sense. By \eqref{keyequsmoothMAop}, we obtain the result. 
In the case where $f\in L^1$,  we reduce to the bounded case by considering the truncation $f_A = \min\{f, A\}$. For $f \in L^\infty(N) \subseteq L^2 (N)$, there exists a sequence of bounded continuous functions $\{f_k\}$ on $N$ converging to $f$ in $L^2$. Then, using a construction from~\cite[Section 3]{Guedj2019weaksub}, there exists a family of uniformly bounded continuous psh function $u_{j,\delta, k}$ satisfying 
\begin{align*}
    \big(dd^c u_{j,\delta, k}\big)^{n+1} \geq \frac{e^{-j ||f_k-f||_{L^2(N)}}} {1+\delta} f_k d\mu, \qquad \text{for } j \in \NN,~ \delta \in (0,1).
\end{align*}
As $k \rightarrow \infty$, $u_{j,\delta, k}$ converges in $L^1$ to a continuous bounded psh function $u_{j,\delta}$. Moreover, the family of functions $\{u_{j,\delta}\}$ satisfies a stability estimate:
\begin{align*}
    u- \frac{\log(1+\delta)}{j} \leq u_{j,\delta} \leq u+ \frac{-\log \delta}{j}.
\end{align*}
In conclusion, by taking limits in $k$ and then in $j$, we recover the inequality $\Delta_H u \geq \frac{1}{1+\delta} f^{1/{n+1}} $ for each $\delta \in (0,1)$ in the sense of distribution, which completes the proof. The reader may refer to \cite{Guedj2019weaksub} for details.
\end{proof}

\begin{proof}[Proof of Theorem \ref{mainthmmaxp} in the general case] By Radon-Nikodym Theorem, the condition $ (\Omega+ dd^c v)^{n+1} \leq E_0 \Theta^{n+1}  $ implies that, in the local coordinates on $B_s (x_k)$,
\begin{align*}
    (\Omega + dd^c v)^{n+1} = f d\mu, \qquad \text{ in the sense of measure (and hence of currents)},
\end{align*}
where $f \in L^\infty (B_s (x_k))$ with $f \leq E_0 \Lambda$ a.e. in $B_s(x_k)$. Here, $d\mu$ is the standard Lebesgue measure with respect to the local coordinates on $B_s(x_k)$. In $B_s(x_k)$, we write $A_{u, \varepsilon} $ as the Hermitian matrices associated with the regularized $(1,1)$ form with potential $(\rho_E + u_E)_\varepsilon$; that is, $dd^c\big((\rho_E + u_E)_\varepsilon \big)$. 
By Lemma \ref{lemregu}, $\big(dd^c (\rho_E + u_E)\big)^{n+1} \geq f d\mu $ is equivalent to 
\begin{align*}
    \big(\det A_{u_\varepsilon} \big)^{1/(n+1)} \geq \big( f^{1/(n+1)}\big)_\varepsilon. \qquad \text{in the classic sense.}
\end{align*}
Meanwhile, \eqref{pfcompareposipart} implies that 
\begin{align*}
   \big( \det A_{\delta/(1-\theta), \varepsilon} \big)^{1/(n+1)} > \big( f^{1/(n+1)}\big)_\varepsilon, \qquad \text{in the classic sense,}
\end{align*}
where $A_{\delta/(1-\theta), \varepsilon} $ is the Hermitian matrix associated with $dd^c (\rho_E- \frac{\delta}{1-\theta} |z|^2)_\varepsilon$. Then, by the same estimates as in the proof of Lemma \ref{MPcomputationlem}, we have
\begin{align*}
    \det \big\{ \theta A_{u, \varepsilon} + (1-\theta) A_{\delta/ (1-\theta), \varepsilon} \big\} \geq \delta' + \big( (f^{1/(n+1)})_\varepsilon\big)^{n+1},
\end{align*}
where $\delta'$ is a positive constant independent of $\varepsilon$. Observe that $  \big( (f^{1/(n+1)})_\varepsilon\big)^{n+1} $ converges to $f$ in $L^1$ as $\varepsilon \rightarrow 0$. Consequently, the measures $\big( (f^{1/(n+1)})_\varepsilon\big)^{n+1} d\mu$ converges to $f d\mu$ as $\varepsilon\rightarrow 0$.

Once again, assume that $u_E > v_E$ at some point in $X \times D$. We define $\tilde{u}$ and $\tilde{v}$ in $B_s (x_k)$ in the same manner as in the smooth case $u,~ v \in \C^2$; see \eqref{pfmodfpotens}.
Consider the regularizations of the local potential functions $\rho_E + \tilde{u}$ and $\rho_E + \tilde{v}$. As $\varepsilon \rightarrow 0$, the Monge-Amp\`ere operator of regularizations $\big(dd^c(\rho_E + \tilde{u})_\varepsilon\big)^{n+1}$ and $\big(dd^c(\rho_E + \tilde{v})_\varepsilon\big)^{n+1}$ converges weakly to $\big(dd^c(\rho_E + \tilde{u})\big)^{n+1}$ and $\big(dd^c(\rho_E + \tilde{v})\big)^{n+1}$, respectively. Let $N = \{\tilde{u} > \tilde{v}\} \subseteq B_s (x_k)$.
\begin{align*}
    \int_N \big( dd^c (\rho_E + \tilde{u})_\varepsilon\big)^{n+1} 
    &=  \int_{N}
    \det \big\{ \theta A_{u, \varepsilon} + (1-\theta) A_{\delta/ (1-\theta), \varepsilon} \big\} d\mu \\
    &\geq \int_{N} \delta' d\mu 
    + \int_{N} 
    \big\{ (f^{1/(n+1)})_\varepsilon\big\}^{n+1} d\mu
\end{align*}
Fix an arbitrary test function $0\leq\eta\leq 1$ with a compact support in $N$. Then, by comparison principle, Lemma \ref{BKlemmacompare}, and Lemma \ref{MPcomputationlem},  we have
\begin{align*}
     \int_{N} \big( dd^c (\rho_E + \tilde{v})\big)^{n+1} 
    &\geq \int_{N} \big( dd^c (\rho_E + \tilde{u})\big)^{n+1}\\
    &\geq \lim_\varepsilon \int_{N} \eta \big( dd^c (\rho_E + \tilde{u})_\varepsilon\big)^{n+1}\\
    &\geq \int_{N} \delta' \eta~ d\mu 
    + \lim_\varepsilon \int_{N}  \eta
    \big\{ (f^{1/(n+1)})_\varepsilon\big\}^{n+1} d\mu \\
    & = \int_{N} \delta' \eta~ d\mu + \int_N \eta f ~d\mu.
\end{align*}
Hence, 
\begin{align*}
    \int_{N} \big( dd^c (\rho_E + \tilde{v})\big)^{n+1} \geq \int_{N} \delta'~ d\mu + \int_{N} \big( dd^c (\rho_E + \tilde{v})\big)^{n+1},
\end{align*}
which leads to a contradiction.
\end{proof}

\section{The uniform priori estimates up to $\C^{1,1}$} \label{secC11esti}

This section is mainly dedicated to completing the proof of Theorem \ref{mthmHCMAc11}. In Subsection \ref{subsecC11esti}, we briefly recall the global $\C^{1,1}$ estimates of the solution to the HCMA equation \ref{introHCMA}. For the $\C^{1,1}$ a priori estimates in the case that $X$ is a compact K\"ahler manifold, we refer the reader to \cite{chen2000space,blocki2009gradient,chu2017regularity}. In the author's previous paper \cite{yao2024geodesicequationsasymptoticallylocally}, the global $\C^{1,1}$ estimates were established on $X \times D$, where $X$ is an ALE K\"ahler manifolds and $D$ is an annulus in $\CC$. When $D\subseteq \CC$ is a bounded domain with smooth boundary, the corresponding $\C^{1,1}$ estimates follow by a straightforward generalization of the proof \cite{yao2024geodesicequationsasymptoticallylocally}. This generalization is discussed briefly in Subsection~\ref{subsecC11esti}. In Subsection \ref{subsecwsolutoHCMA}, we complete the proof of Theorem \ref{mthmHCMAc11}.

\subsection{$\varepsilon$-Monge-Amp\`{e}re Equations and Priori Estimates} \label{subsecC11esti} Fixing a K\"ahler form $\omega$ on $X$, consider the following K\"ahler potential space on $X$ with respect to $\omega$:
\begin{align*}
    \HH^{k,\alpha} (\omega) = \{ \varphi \in \C^{k,\alpha}(X); \ \omega + dd^c\varphi \},
\end{align*}
where the regularity order $k \geq 5$ is assumed only for technical reasons in this section. Note that the potential functions $\varphi \in \C^{k,\alpha}(X)$ satisfies a uniform global bound: $||\varphi||_{k,\alpha; X} \leq C$.
For each $\tau \in \partial D$, we prescribe a K\"ahler form $\omega_\tau$ on the fiber $X \times \{\tau\}$, where $\omega_\tau$ lies in the same $dd^c$-cohomology class as $\omega$. Precisely, we can write $\omega_\tau = \omega + dd^c \psi_\tau$ with $\psi_\tau \in \HH^{k,\alpha} (\omega)$ on $X \times \{\tau\}$. Additionally, we can assume $\Psi(\cdot, \tau) = \psi_{\tau} (\cdot)$ is of the class $\C^{k,\alpha}(X \times \partial D) $.  

Note that the function $\Psi \in \C^{k,\alpha} (X \times \partial D) $ can be extended to $\Psi \in \C^{k,\alpha} (X \times \overline{D})$ such that, for each $\tau \in D$, the slice $ \psi_\tau (\cdot) = \Psi (\cdot, \tau)$ is $\omega$-psh. If we solve the following Dirichlet problem on $D \subseteq \CC$:
\begin{align} \label{equbarriersupharXD}
\begin{split}
    \Delta u(\tau) &= 1, \hspace{0.8cm} \tau \in D,\\
    u(\tau) &=0, \hspace{0.8cm} \tau \in \partial D,
\end{split}
\end{align}
there is a smooth solution $u(\tau)$ to the above equation such that $||u||_{k,\alpha; D}\leq C$, where $C$ is a uniform constant depending only on $k$, $\alpha$ and $D$. Let $U(x, \tau) = u(\tau)$ be the pull-back of $u(x)$ to the total space $X \times D$. There is a large constant $A$ such that $(\Omega + dd^c \Psi + A\cdot dd^c U) $ is a positive $(1,1)$ form on $X \times D$. Let $\Theta = \Omega + A \cdot dd^c U$. Consider the following $\varepsilon$-Monge-Amp\`{e}re equations, by adding a positive term with a parameter $\varepsilon$:
\begin{align*}
    (E_\varepsilon) \qquad \begin{cases}
        \big(\Theta + dd^c \Phi_\varepsilon\big)^{n+1} = \nu(\varepsilon) \big(\Theta + dd^v \Psi\big)^{n+1}, \hspace{0.4cm} &\text{ on } X \times D,\\
        \Theta + dd^c \Phi_\varepsilon >0, &\text{ on } X \times D,\\
        \Phi_\varepsilon = \Psi, & \text{ on } X \times \partial D,
    \end{cases}
\end{align*}
where $\nu(\varepsilon)$ is a smooth nonnegative function defined on $X \times D \times [0,1]$ satisfies the following conditions:
\begin{align*}
\begin{split}
    &\nu (0) \equiv 0, \qquad \nu(1) \equiv 1; \\
    &C^{-1} \varepsilon \leq \nu (\varepsilon) \leq \min (C\varepsilon,1 ), \hspace{0.6cm} \text{ for } \varepsilon\in [0,1]\\
    & |\nabla^{k} \nu (x, \tau,\varepsilon)|_{\Theta} \leq C \varepsilon,\hspace{0.4cm} \text{ on } X \times D \times [0,1], \quad k \geq 1.
\end{split}
\end{align*}
The $\varepsilon$-Monge-Amp\`{e}re equation can be solved by the standard continuity method. It is straightforward to observe that when $\varepsilon=1$, there is a trivial solution $\Phi_1 = \Psi$ to $(E_1)$. The openness follows from the linearized problem of Monge-Amp\`{e}re equation. The uniform $\C^{1,1}$ estimates, together with the bootstrapping proof of regularity theory of Monge-Amp\`{e}re equation to $\C^{k,\alpha}$ (the $\C^{k,\alpha}$ estimates depend on $\varepsilon$), imply the closedness. Furthurmore, if we let $\varepsilon$ go to zero, by the uniformity of the $\C^{1,1}$ estimates of the complex $\varepsilon$-Monge-Amp\`{e}re solution, there is a subsequential limit $\Phi $ solving the HCMA equation \eqref{introHCMA} and satisfying the global $\C^{1,1}$ estimates.

\subsubsection{The openness} Assuming that there exists a solution of $(E_{s_0})$ in $\C^{k,\alpha}$ for some $s_0 \in (\varepsilon,1)$, we show that $(E_s)$ can be solved for all $s$ in a small open neighborhood of $s_0$. In this subsubsection, we write $\Phi$ for the solution to $(E_{s_0})$ for simplicity. The space of solutions to the $\varepsilon$-Monge-Amp\`{e}re equation here is given by $\{\Phi \in \C^{k,\alpha} (X \times D); ~ \Phi |_{X \times \partial D}= \Psi \}$.
Let $u$ be a function in the tangent space of the solution space, $\{u \in\C^{k,\alpha}(X\times \overline{D}); ~ u|_{X \times \partial D} = 0 \}$.
The Monge-Amp\`{e}re operator at $\Phi$ is defined to be,
\begin{align*}
    \mathcal{M} (\Phi) = \frac{(\Theta + dd^c \Phi)^{n+1}}{(\Theta + dd^c \Psi)^{n+1}}.
\end{align*}
Given the assumption that $\Theta + dd^c \Phi>0 $, the linearization of Monge-Amp\`{e}re operator at $\varphi$ is uniformly elliptic, which is given by,
\begin{align*}
    \LL_{\Phi} (u ) =\big( \Delta_{\Phi}  u \big) \cdot \frac{(\Theta + dd^c \Phi)^{n+1}}{(\Theta + dd^c \Psi)^{n+1}}= \nu(s_0) \cdot \Delta_{\Phi} \chi,
\end{align*}
where $\Delta_\Phi$ represents the Laplacian with respect to $\Theta_\Psi + dd^c \Phi$. Let $(\C^{k,\alpha})_0$ be the functions in $\C^{k,\alpha}$ vanishing on the boundary $X \times \partial D$. Then, we have the following property of $\LL_{\Phi}$.

\begin{lem} \label{lemisolinMA}
Let $\Phi$ be the solution of $(E_{s_0})$, then the linearized operator $\LL_\Phi : ({\C}^{k,\alpha})_0 \rightarrow {\C}^{k-2,\alpha}$ is an isomorphism for all integers $k \geq 2$ and $\alpha\in (0,1)$.
\end{lem}

\begin{proof}
The proof proceeds in parallel with that of \cite[Lemma 2.3]{yao2024geodesicequationsasymptoticallylocally}{yao2024geodesicequationsasymptoticallylocally}, and we only outline the main ideas here. Take an exhaustive sequence of pre-compact sets with smooth boundaries, $\widetilde{N}_j \subseteq X \times D$, by smoothing the corner of $ B_{r_j} \times D$, where $B_{r_j} = \{x \in X; ~ r(x) < r_j\}$ and $r_j \rightarrow \infty$ as $j \rightarrow \infty$. The key step is to solve a sequence of Dirichlet problems of elliptic linear equations:
\begin{align*}
    \begin{cases}
        \LL_{\Phi} u_j = f, \hspace{0.6cm} &\text{ on } \widetilde{N}_j,\\
        u_j = 0, & \text{ on } \partial \widetilde{N}_j.
    \end{cases}
\end{align*}
It's a classic result that for each $j\in \NN$, there is a solution $u_j$ with the $\C^{k,\alpha}$ estimates: $$ ||u_j||_{k,\alpha; \widetilde{N}_j} \leq C (1 + ||f||_{k-2,\alpha; X \times D}), $$
where $C$ is a uniform constant independent of $j$. The uniqueness follows from the maximal principle for $\Delta_\Phi$ on $X \times D$, see Lemma \ref{WuYauweak} and Remark \ref{remWYmaxp}.
\end{proof}

\subsubsection{Priori estimates up to $\C^{1,1}$}

In this subsection, we establish the global $\C^{1,1}$ regularity of the solution to the HCMA equation on $X \times D$. The overall strategy closely follows the general framework developed in \cite{chen2000space}, which proceeds through successive estimates: first $\C^0$ and $\C^1$ bounds on the boundary, then interior $\C^1$ estimates using a maximum principle applied to a suitable test function (cf. \cite{blocki2009gradient}), followed by $\C^2$ boundary estimates and Laplacian bounds, and finally interior $\C^2$ estimates using a test function constructed as in \cite{chu2017regularity}. 
Most of the analytic groundwork for this scheme has already been carried out in the author’s earlier paper \cite{yao2024geodesicequationsasymptoticallylocally}. Hence, we only summarize the key points in this subsection required in the setting where $D \subseteq \mathbb{C}$ is a general bounded domain with smooth boundary.

The $\C^0$ bound of the solution of $(E_\varepsilon)$ is a direct corollary of the second type of maximal principle of the Monge-Amp\`{e}re operator, Theorem \ref{mainthmmaxp}. Let $ \Phi_\varepsilon$ be a bounded solution to $(E_\varepsilon)$.  Then,
\begin{align*}
    \big(\Theta + dd^c \Phi_\varepsilon\big)^{n+1} = \nu(\varepsilon) \big(\Theta + dd^v \Psi\big)^{n+1}\leq  \big(\Theta + dd^v \Psi\big)^{n+1}
\end{align*}
and $\Phi_\varepsilon = \Psi$ on $X \times \partial D$. By Theorem \ref{mainthmmaxp}, we have $\Phi\geq \Psi$ on $X \times D$. Recall that $\Psi \in \C^{\infty}_{-\gamma} (X)$ implies that $d_\tau d_\tau^c \Psi\leq A  d\tau \wedge d \bar{\tau} $ on $X \times D$ for some uniform constant $A$. Note that the function $U$ satisfies the equation \eqref{equbarriersupharXD}. The classic maximal principle implies that $U < 0$ on $X \times D$. Now, if we restrict to each slice $\{x\} \times D$ with the natural embedding $e: \{x\} \times D \rightarrow X \times D $, we have
$$ e^*\Omega + d_{\tau}d^c_{\tau} \Psi - A d_\tau d^c_\tau U \leq 0 \leq e^* \Omega + d_\tau d^c_\tau \Phi_\varepsilon,$$
and $\Psi - A U = \Phi_\varepsilon$ on $ X \times D$. Also by classic maximal principle, we have $\Psi \leq \Phi_\varepsilon \leq \Psi - A U $ on $ X \times D$. Then, we have the following $\C^0$ priori bounds:
\begin{lem} 
Let $\Phi_\varepsilon$ be the solution of $(E_\varepsilon)$. Then $\Phi_\varepsilon$ satisfies the following $\C^0$ priori estimates,
\begin{align*}
    \Psi \leq \Phi_\varepsilon \leq \Psi - AU. 
\end{align*}
where $A$ is a constant independent of $\varepsilon$.
\end{lem}

The $\C^1$ boundary estimates following directly from the above Lemma. Since $\Psi \leq \Phi_\varepsilon \leq \Psi - AU $ and the functions $\Psi $, $\Phi_\varepsilon$, $\Psi - A U$ agree along $X \times \partial D$, we have
\begin{align*}
     | \nabla \Phi |_{\Theta} \leq \max \{ |\nabla \Psi|_{\Theta}, |\nabla (\Psi -AU)|_{\Theta} \}, \quad \text{ on } X \times \partial D.
\end{align*}
Hence, $\sup_{X \times \partial D} |\nabla \Phi|_{\Theta_\Psi} \leq C$, where $C$ is a uniform constant.

The interior $\C^1$ estimate closely follows the gradient estimates established in Blocki \cite{blocki2009gradient}. To adapt the argument to the setting of ALE K\"ahler manifolds, the maximal principle (cf. Lemma \ref{WuYauweak} and the following remark) should be involved. Note that $\Theta$ defines a global K\"ahler form on $X \times D$. In the remainder of this subsection, we denote by $g$ the K\"ahler metric associated with $\Theta + dd^c \Psi$, and by $g'$ to the metric associated with $\Theta+ dd^c \Phi_\varepsilon$.

\begin{lem} \label{c1priest} Let $ \varphi = \Phi_\varepsilon -\Psi $, where $\Phi_\varepsilon$ is a solution to $(E_{\varepsilon})$, and let $\nabla$ be the Levi-Civita connection corresponding to $g$. Assume that $\varphi$ is a bounded in $\C^1(X \times D, g)$. Then, 
\begin{align*}
    \sup_{X \times D}|\nabla \varphi|_g \leq C,
\end{align*}
where $C$ is a positive constant depending only on upper bounds for $|\varphi|$, on lower bound for bisectional curvature of $g$, and on $n$.
\end{lem}

\begin{proof} The idea of the proof is to apply the maximal principles to the following test function:
\begin{align*}
    \alpha = \log \beta - \gamma \circ \varphi,
\end{align*}
where $\beta = |\nabla \varphi|^2_g$ and $\gamma$ is a quadratic polynomial determined by the data $\sup_{X \times D} \varphi$, $\inf_{X \times D} \varphi$ and the negative lower bound of bisectional curvature of $g$.  
By Wu-Yau's maximal principle, there is a sequence $\{x_k\}$ on $X\times D$ such that 
\begin{align*}
    \lim_{k \rightarrow \infty} \alpha (x_k) = \sup_{X \times D} \alpha, \quad \lim_{k \rightarrow \infty} |\nabla \alpha (x_k)|_{g} =0, \quad \limsup_{k \rightarrow \infty} \Delta \alpha (x_k) \leq 0. 
\end{align*}
We now fix  $O =x_k$, where $k$ is sufficiently large. Around the point $O$, we choose the normal coordinates with
$g_{i\overline{j}} (O) = \delta_{ij}, \ g_{i\overline{j}, k}(O) = 0  \text{ and }  g'_{i\overline{j}}(O) \text{ is diagonal}.$

Let $\rho'$ be the local potential function of $g'$ around $O$. By a refined computation due to Blocki \cite{blocki2009gradient}, combined with the above Wu-Yau's maximal principle, there exists a small constant $0<\mathbf{e}\ll 1$ such that at the given point $O$,
\begin{align*}
    \mathbf{e} \geq \sum_{p=1}^{n+1} \frac{\alpha_{p\bar{p}}}{\rho'_{p\bar{p}}} \geq \big(\gamma' &+ \inf_{j\ne l} R_{j\bar{j}l\bar{l}}\big) \sum_{p} \frac{1}{\rho'_{p\bar{p}}} - \gamma'' \sum_p \frac{|\varphi_p|}{\rho'_{p\bar{p}}}\\
    & - (n+2) \gamma' - \frac{2}{\beta} - (1+ |\gamma'| + \beta^{-1})\mathbf{e},
\end{align*}
where $ \inf_{j\ne l} R_{j\bar{j}l\bar{l}}\big)$ is the bisectional curvature of $g$.
Let $\gamma = (-\inf_{j\ne l} R_{j\bar{j} l\bar{l}}+3) (t-A) - (B-A)^{-1} (t-A)^2 $, where $A = \sup_{X \times D} \varphi $ and $B = \inf_{X \times D} \varphi$. If we assume that $\beta (O) \geq 1$, then we have
\begin{align*}
    \sum_{p} \frac{1}{u_{p\overline{p}}} + \frac{2}{B-A} \sum_{p} \frac{|\varphi_p|^2}{ u_{p\overline{p}}} \leq 3 + (n+2) (-\inf_{j\ne l} R_{j\bar{j} l\bar{l}}+3).
\end{align*}
It is straightforward to conclude that
$\beta (O) \leq \max \{ [(n+3)(-\inf_{j\ne l} R_{j\bar{j} l\bar{l}}+3)]^{n} n(B-A) ,1\}$. Noting that $\beta \leq \exp \{ \mathbf{e} + \log \beta(O) - \gamma\circ\varphi (O) + \gamma\circ \varphi\}$, hence, $\beta$ is controlled by some uniform depending only on $n$, $||\varphi||_{L^\infty}$ and the negative lower bound of bisectional curvature of $g$.
\end{proof}

It comes to the point to deal with the $\C^2$ priori estimates on $X \times D$. The idea is

\begin{lem}\label{c2pribdyest}
Let the data $(X \times D,  g, g', \varphi)$ be the same as in Lemma \ref{c1priest}, then 
\begin{align*}
    \sup_{X\times \partial D}|\nabla^2 \varphi|_g \leq C,
\end{align*}
where the constant $C$ only depends on $\sup_{X\times D} |\nabla \varphi|$ and $ (X \times D, g) $.
\end{lem}
\begin{proof} The boundary $\C^2 $ estimates are obtained by analyzing the tangential-tangential, tangential-normal, and normal-normal directions separately. The tangential-tangential estimates is trivial due to $\varphi \equiv 0 $ on $X \times \partial D$. The normal-normal estimate then follows from tangential-tangential and mixed estimates. The key step is tangential-normal estimates, which can be proved by constructing a local barrier function around a given boundary point $p \in X \times \partial D$.

Fixing a point $p \in X \times \partial D$, consider the auxiliary function in $B'_\delta (p) = (X \times D ) \cap B_\delta (p) $,
\begin{align}\label{auxfun}
    v = \varphi - N U,
\end{align}
where $N$ is a large uniform constant and $U$ is the function satisfies the equation \ref{equbarriersupharXD}. Recall $\LL_{\Phi_\varepsilon}$ be the linearization of Monge-Amp\`{e}re operator at $\Phi_\varepsilon$. If we assume $m\delta_{ij} \leq g \leq M \delta_{ij}$ in a local neighborhood containing $B'_\delta(p)$, and taking $N = (2/m)^{n+1} \det g$, we then have $\displaystyle\LL v \leq -\frac{m}{2} \sum_i (g')^{i\bar{i}}$. Also note $v \geq 0$ on $\partial B_\delta(p)$. The barrier functions can be constructed as follows,
\begin{align*}
    w = A v + B |z|^{2} \pm \frac{\partial}{ \partial x_k} \varphi,
\end{align*}
where $A$, $B>0$ are large uniform constants, and $\frac{\partial}{\partial x_k}$ is a local coordinate vertor field near $p$ in the direction of $X$. By picking a very large constants $A$, $B$, we have $\LL_{\Phi_\varepsilon} w \leq 0$ in $B_\delta'(p)$ and $w\geq 0$ on $\partial B'_\delta(p)$. Hence, $w\geq 0$ in $B'_\delta (p)$. Together with the fact that $w(p) =0$, the inward normal derivative of the test function at $p$, denoted $dw (\mathbf{n})$, satisfies $dw (\mathbf{n}) \geq 0$, which yields the tangential-normal estimate on the boundary.
\end{proof}

The lemma \ref{c2pribdyest} together with the Yau's standard calculation on Laplacian estimate implies the following interior Laplacian estimate, referring to \cite{Yau78}.

\begin{lem} \label{Yaulaest}
Let the data $(X \times D,  g, g', \varphi)$ be the same as in Lemma \ref{c1priest}, and let $\Delta$, $\Delta'$ be the Laplacian operators of $g$ and $g'$ respectively. Then, for any constant $C$,
\begin{align*}
    \Delta' \big(e^{-C \varphi} (n+1 + \Delta \varphi) \big) \geq e^{-C \varphi} & (n+1)^2 \inf_{i\ne l} (R_{i\overline{i}  l \overline{l}}) - C e^{-C \varphi} (n+1) (n+1+ \Delta \varphi)\\ & + (C + \inf_{i\ne l} (R_{i\overline{i}l \overline{l}})) e^{-C\varphi} (n+1 + \Delta \varphi)^{1+\frac{1}{n}} (\varepsilon)^{-1}.
\end{align*}
Furthermore, we have the interior Laplacian estimate,
\begin{align*}
    \sup_{X \times D}|\Delta \varphi| \leq C (1+ \sup_{X \times \partial D} |\Delta \varphi|),
\end{align*}
where the constant $C$ only depends on $\sup_{X \times D} \varphi $, the negative lower bound of $\inf_{i\ne l}(R_{i\overline{i}l \overline{l}})$.
\end{lem}

Note that Lemma \ref{Yaulaest} implies a comparison of metrics $\varepsilon C^{-1 } g \leq g' \leq C g$ with a uniform constant $C$.
The full $\C^{1,1}$ estimates then follow from a careful construction of the test function, and again, together with Lemma \ref{WuYauweak}. 

\begin{lem}\label{leminterc11}
Let the data $(X \times D,  g, g', \varphi)$ be the same as in Lemma \ref{c1priest}. Then there exists a constant $C$,
\begin{align*}
    |\nabla^2 \varphi| \leq C,
\end{align*}
where $C$ depends only on $(X \times D, g)$ and on $\sup_{X\times D}|\varphi|$, $\sup_{X \times D} |\nabla \varphi|_g$, $\sup_{X \times D} |\Delta \varphi|$, $\sup_{X \times \partial D} |\nabla^2 \varphi|_g$.
\end{lem}

\begin{proof}
Let $\lambda_1(\nabla^2 \varphi)$ be the largest eigenvalue of the real Hessian $\nabla^2 \varphi$. By observing that there exists a uniform constant $C$, $\lambda_1 (\nabla^2 \varphi) \leq |\nabla^2 \varphi|_g \leq C \lambda_1 (\nabla^2 \varphi) +C $, it suffices to prove that $\lambda_1 (\nabla^2 \varphi)$ has a uniform upper bound. Consider the following test function
\begin{align*}
    Q = \log \lambda_1 (\nabla^2 \varphi) + h (|\nabla \varphi|_g^2) -A \varphi,
\end{align*}
where $h$ is defined to be $\displaystyle h(s) = -\frac{1}{2} \log \big (1 + \sup_{X \times \Sigma} |\nabla \varphi|_g^2 -s  \big)$ and $A>0$ is a large uniform constant. Let $\mathbf{e} >0$ an arbitrarily small constant. Then there exists a point $p \in X \times D$ such that $\sup_{X \times D} Q - Q (p ) < \mathbf{e}$. Assume further that the eigenvalues of the real Hessian $\nabla^2 \varphi$ satisfy $\lambda_1(p)> \lambda_2(p) \geq \lambda_3 (p)\geq \ldots \geq \lambda_{2n+2}(p)$. Under this assumption, the test function $Q$ is smooth near $p$ and satisfies 
\begin{align*}
   | d {Q}(p)|_g \leq C\mathbf{e}, \quad \Delta {Q} (p) \leq C\mathbf{e},
\end{align*}
for a uniform constant $C$. In the general case, one can perturb the Hessian matrix slightly to reduce to the above situation. For further details, we refer the reader to \cite{yao2024geodesicequationsasymptoticallylocally}.

Then, by direct calculation in \cite[Lemma 2.1]{chu2017regularity}, if we assume $\lambda_1 \geq 1$ at $p$, we have
\begin{align} \label{laQ1}
\begin{split}
    C\mathbf{e} \geq \Delta {Q} \geq & \ 2 \sum_{\alpha> 1} \frac{(g')^{i\overline{i}}|\partial_i (\varphi_{V_\alpha V_1})|^2}{\lambda_1 (\lambda_1- \lambda_\alpha)} + \frac{(g')^{i\overline{i} }(g')^{j\overline{j}} \big|V_1 \big( (g')_{i\overline{j}}\big)\big|}{\lambda_1} - \frac{(g')^{i\overline{i}} |\partial_i (\varphi_{V_1 V_1})|^2}{\lambda_1^2} \\
    & + h' \sum_{k} (g')^{i\overline{i}} \big( |\varphi_{ik}|^2 + |\varphi_{i\overline{k}}|^2 \big) + h'' (g')^{i\overline{i}} \big|\partial_i |\nabla \varphi|^2_g \big|\\
    & + (A-B) \sum_{i} g'^{i\overline{i}} -A n,
\end{split}
\end{align}
where the constant $B$ only depends on $(X\times D, g)$ and $\sup_{X \times D} |\nabla \varphi|_g$, and $V_\alpha$ is the corresponding eigenvector of $\lambda_\alpha$. To cancel the annoying terms on the right-hand side of \eqref{laQ1}, we address the third term in the expansion. The detailed computation has been provided in [CTW, Lemma 2.2]. If we assume further that $ \displaystyle \lambda_1 \geq 8 A^2   \big(\sup_{X \times D } |\nabla \varphi|^2 +1\big) C $, we have 
\begin{align*}
    C \mathbf{e} \geq &\ h' \sum_k (g')^{i\overline{i}} \big( |\varphi_{ik}|^2 + |\varphi_{i\overline{k}}|^2 \big) +\big( h''- 2 (h')^2 \big)  (g')^{i\overline{i}} \big| \partial_i |\nabla \varphi|^2 \big| \\
    &  +(A-B -C \mathbf{e}-1) \sum_{i} (g')^{i \overline{i}} -A n.
\end{align*}
Notice that $h'' = 2(h')^2$. Picking $\varepsilon \leq 1/C $, $A = B+3$ , then we have,
\begin{align*}
    h' \sum_k (g')^{i\overline{i}} \big(|\varphi_{ik}|^2 + | \varphi_{i\overline{k}}|^2\big) + \sum_{i} (g')^{i\overline{i}} \leq An+1
\end{align*}
Recall that $g' \leq C g $. Hence, at $p$, $(g')^{i\overline{i}} \geq C^{-1}$. Then, 
\begin{align*}
    \lambda_1 (p) \leq \max \Big\{ 8 A^2   \big(\sup_{X \times D } |\nabla \varphi|^2 +1\big) C, \big\{(An +1)C -n \big\} (1+ \sup_{X \times \Sigma} |\nabla \varphi|_g^2) \Big\}
\end{align*}
Together with the fact that $\sup_{X \times \Sigma} Q \leq Q(p) +1 $, we prove that $\sup_{X \times \Sigma} \lambda_1 $ is bounded by some uniform constant. 
\end{proof}

\subsection{The Weak Solution to HCMA up to $\C^{1,1}$}\label{subsecwsolutoHCMA}
According to Lemmas from \ref{lemisolinMA} to \ref{leminterc11}, we conclude the existence of the weak solution to HCMA on $X \times D$: 

\begin{thm} \label{thmc11soluasupenvelope}
    Let $X$ be an ALE K\"ahler manifold with ALE K\"ahler form $\omega$, D, a domain with smooth boundary in $\CC$ and $\Omega = p^* \omega$, an nonnegative real closed $(1,1)$ form on $X \times D$. Suppose that $\Psi$ is a smooth function on $X \times \partial D$, and $\Psi \in \C^{k,\alpha} (X \times \partial D)$, There exists a unique solution, $\Phi$, of 
    \begin{align*} 
    \begin{aligned}
        &\Phi \text{ plurisubharmonic } & \text{in } X \times D\\
        &\Phi \text{ continuous } & \text{in }  X \times \overline{D}\\
        &\Phi = \Psi & \text{in } X \times \partial D\\ 
        &\big(\Omega + dd^c \Phi\big)^{n+1} = 0 & \text{in } X\times D,
    \end{aligned}
    \end{align*}
    and $\Phi$ is a $\C^{1,1}$ function on $X \times D$ satisfies
    \begin{align*}
        ||\Phi||_{1,1; X \times D} \leq C,
    \end{align*}
    where $C$ is a uniform constant depending only on the data $D$, $\Theta$ and $||\Psi||_{k,\alpha; X \times D}$.
    
    Furthermore, $\Phi$ agrees with the upper envelope of $\Pl_{\Omega, \Psi}$.
\end{thm}
\begin{proof} By solving $\varepsilon$-Monge-Amp\`{e}re equation $(E_\varepsilon)$, we obtain the solution $\Phi_\varepsilon \in \C^{k,\alpha} (X \times D)$, with the uniform estimates up to $\C^{1,1}$. 
The existence of the weak solution to HCMA equation follows by taking subsequential limits of $\Phi_\varepsilon$ in $\C^{1, \beta}(X \times D)$ for any $\beta<1$, and $||\Phi||_{1,1; X \times D} \leq C$ is immediately from the uniform $\C^{1,1}$ bound for $\Phi_\varepsilon$. 

The uniqueness is an immediate implication of the maximal principle, Theorem \ref{mainthmmaxp}. Moreover, the upper envelope of $\Pl_{\Omega, \Psi}$ is a bounded countinuous solution to HCMA by Theorem \ref{thmgsolHCMAupenv}, which automatically agrees with $\Phi$ due to the uniqueness.
\end{proof}

\section{Holomorphic discs foliation} \label{secholodiscfoliexist}

This section is dedicated to establishing the existence of a foliation by a family of holomorphic discs on $(X-K) \times D$, where $K \subseteq X$ is a closed ball with a large radius. While the full proof of the global foliation is completed in the next section, our goal here is to construct the foliation locally in a holomorphic ball. This serves as a crucial step toward the global result. 

The existence of holomorphic disc foliations has been studied previously in the compact K\"ahler setting by Donaldson \cite{donaldson1999symmetric,donaldson2002holo}. More recently, Chen–Feldman–Hu \cite{CHEN2020108603} revisited the problem and reproved the existence of holomorphic disc foliations by reducing it to a local PDE problem also on compact K\"ahler setting. In this paper, we work in the non-compact, ALE K\"ahler setting, and one key new observation is that solving a family of Riemann–Hilbert problems for the holomorphic discs leads to a loss of regularity in the parameter direction, i.e. in the $X$-driection. This phenomenon does not appear to be addressed in the existing literature, and we will analyze and overcome this difficulty in detail.

Let $X$ be an ALE K\"ahler manifold with the end $X_\infty$. According to Proposition \ref{propcxasymcoordinates} and Corollary \ref{corcxasymptoticcovering}, there are two different ways to describe the asymptotic complex coordinates of $X$:
\begin{enumerate}
    \item[(i)] In the case of complex dimension $n \geq 3$, there is a biholomorphism between the universal covering of the end $\widetilde{X}_\infty$ and $\CC^n - B_R$, $I: \widetilde{X}_\infty \rightarrow \CC^n - B_R$.
    \item[(ii)] In the case of complex dimension $n=2$, the end $X_\infty$ can be covered by a family of locally finite and countably many large balls in the sense that $X_\infty \subseteq \bigcup_{i \in \mathcal{I}} U_i$ and for each $i\in \mathcal{I}$, there is a biholomorphism $I_i: U_i \rightarrow B_R \subseteq \CC^n$. 
\end{enumerate}
In this section, we complete the proof of the existence of holomorphic disc foliations in case (i) and a part of case (ii). More precisely, we prove the existence of foliations within each holomorphic coordinate chart covering the end. The full proof of Theorem \ref{mthmexistholodiskfoli}, including the necessary patching argument for case (ii), will be completed in the next section (see Theorem \ref{thmpatch}). Additionally, we prove a uniform estimate on the displacement (or shifting) of holomorphic discs. The weighted version of the estimates will also be completed next section, see Theorem \ref{thmweightestifoli}.

\subsection{Holomorphic discs and homogeneous complex Monge-Amp\`{e}re equation} \label{secholodiscHCMA} In Donaldson\cite{donaldson2002holo} and Semmes \cite{semmes1992complex}, the existence of nondegenerate smooth solutions of HCMA equation is proved to be equivalent to the global existence of families of holomorphic discs under the setting of compact K\"ahler manifolds. 
In this subsection, we first provide a precise description of the holomorphic disc foliation within a holomorphic coordinate chart. We then give a constructive proof of the equivalence statement. The theorem applies uniformly to both cases (i) and (ii), as the analysis takes place entirely within holomorphic charts modeled on domains in $\CC^n$, such as $B_R$ or its complement $\CC^n \backslash B_R$.

 For the sake of simplicity, throughout this section, we denote $N$ as either the ball of radius $R$, $B_R$ or the complement of the ball in $\CC^n$, $\CC^n- B_R$. In addition, we will use $N'$, $N''$ to denote the open domains in $\CC^n$ by stretching and shrinking $N$ slightly that satisfy $N' \supseteq \overline{N} \supseteq N \subseteq \overline{N''}\supseteq N''$. For instance, we can take $N' = B_{R+1}$, $N = B_R$, $N'' = B_{R+1}$ and $N' = \CC^n - B_{R-1} $, $N = \CC^{n} - B_{R}$, $N = \CC^n - B_{R+1}$. The complex coordinates of $N$ are the standard one in $\CC^n$, denoted by $\{z^1, \ldots, z^n\}$.

 Let $E$ be the holomorphic cotangent bundle of $N'$ with canonical projection $\pi_E: E \rightarrow N'$. The cotangent bundle is trivial and the complex bundle coordinates are denoted by $\{\xi_1, \ldots, \xi_{n}\}$. There is a canonical complex symplectic form defined in the total space of $E$, in terms of canonical holomorphic coordinates $z_i$, $\xi_i$, $(i=1,\ldots, n)$ of $E$,
 \begin{align*}
     \Xi = dz_i \wedge d\xi_i.
 \end{align*}
 Let $\omega$ be a K\"ahler form in $N$ by restricting the reference K\"ahler form on $X$. According to [CH], the ddbar lemma can be applied to $N$ for both of the cases. Then, we have
 \begin{align*}
     \omega = i\partial \overline{\partial} \rho,
 \end{align*}
 where $\rho$ is a smooth function defined in $N$. In particular, in the case of $N= \CC^n - B_R$, the potential function is given by $\rho = r^2 + \psi$, where $\psi \in \C^{\infty}_{2-\tau}$. The K\"ahler form $\omega$ can be associated with a submanifold, $\Lambda_\rho$ of $E$ defined by the graph of $\partial \rho $. By restricting $\Xi$ on $\Lambda_\rho$, a direct calculation shows that,
 \begin{align*}
     \Xi|_{\Lambda_\rho} =  \partial \overline{\partial} \rho = -i \omega.
 \end{align*}
 Thus, $\re \Xi|_{\Lambda_\rho} =0$ and $\im \Xi|_{\Lambda_\rho} = -\omega$. Hence, $\Lambda_\psi$ is an exact Lagrangian submanifold with respect to the canonical symplectic structure of $E$.

 Consider the boundary data given by the HCMA equation, $\omega_\tau = \omega + i\partial \overline{\partial} \psi_\tau = i\partial \overline{\partial} \tilde{\psi}_\tau$ (with $ \tilde{\psi}_\tau = \rho + \psi_\tau $), for $\tau \in \partial D$.
 Let $\Lambda_{\tilde{\psi}_{\tau}}$ be the exact Lagrangian submanifold of $E$ given by the graph of $\partial \tilde{\psi}_\tau$.
 The family of holomorphic discs, whose boundary data belongs to $\Lambda_{\tilde{\psi}_\tau}$ for each $\tau \in \partial D$, can be described as follows. For each $x \in N$, there is a smooth family of holomorphic discs $G : N \times D \rightarrow E$ such that
 \begin{itemize}
     \item $G$ is smooth in $N \times D$;
     \item let $g_x (\tau) = G(x, \tau)$. $g_x (\tau)$ is holomorphic.
     \item for each $\tau \in \partial D $ and each $x\in N$, 
     \begin{align*}
         g_x(\tau) \in \Lambda_{\tilde{\psi}_\tau};
     \end{align*}
     \item for each $x\in N$, we have $\displaystyle \pi\circ g_x (-i) = x$;
     \item let $H (x,\tau) = \pi \circ G(x, \tau) $ and $h_\tau (x) = H(x, \tau)$. For each $ \tau \in  D$, the map $h_\tau: N \rightarrow N'$ is a diffeomorphism with the image. 
     In addition, for each $\tau \in D$, $h_\tau (x)$ satisfies
     \begin{align*}
         N''  \subseteq h_\tau ( N ) \subseteq N'.
     \end{align*}
 \end{itemize}

The foliation by holomorphic discs can be written in terms of the holomorphic coordinates of $E$. Denote $z_i(x, \tau) = z_i (g_x (\tau))$ and $\xi_i (x, \tau) = \xi_i (g_x (\tau))$. Also noting that $h_\tau$ defines a diffeomorphism from $N$ to $h_{\tau} (N) \subseteq N'$ for each $\tau \in D$, throughout the section, we introduce the complex coordinates $w = (w^1, \ldots, w^n)$ of $N$. 
The following PDE gives another interpretation of the foliation by holomorphic discs. For $i = 1,2,\ldots, n$,
\begin{align}\label{holodiscpde}
\begin{split}
    \frac{\partial}{\partial \bar{\tau}} z_i (w, \tau) = \frac{\partial}{\partial \bar{\tau}} \xi_i (w, \tau) = 0, \qquad & (w, \tau) \in N \times D; \\
    \xi_{i} (w, \tau) = \Big(\frac{\partial}{\partial{z_i}} \tilde{\psi}_\tau\Big) (z(w, \tau), \tau), \qquad & (w, \tau) \in N \times \partial D;\\
    {z} (w, -i) =  w, \qquad \qquad \qquad & \ w \in N,
\end{split}
\end{align}
and $z(w, \tau)$ gives a family of diffeomorphism of $N$ satisfying, 
\begin{align}\label{holodiscimage}
N'' \subseteq \{z(w, \tau)| \ w\in N\} \subseteq N',\qquad \text{ for }  \tau \in D.
\end{align}

Based on the description of holomorphic discs foliation (\ref{holodiscpde}) and (\ref{holodiscimage}). we can construct a smooth nondegenerate solution of HCMA in $N'' \times D$ providing the data of the smooth family of holomorphic discs, $g_w$, for $w \in N$. The precise construction is given as follows. For each $w \in N$, we define $\Phi (z(w,\tau), \tau)$ as
\begin{align} \label{explicitconstructionsolu}
\begin{split}
\frac{\partial}{\partial {\tau}} \frac{\partial}{\partial \bar{\tau}} \big(\Phi(z(w, \tau), \tau)\big) = 0,\quad\qquad & \tau \in D,\\
    \Phi(z(w, \tau), \tau) = \rho_\tau (z(w, \tau), \tau), \qquad &\tau \in \partial D.
    \end{split}
\end{align}
In other words, we construct a function $\Phi$ in $ N'' \times D$, by defining a harmonic function along each leaf at $x \in N$ agreeing with the potential functions at the boundary. In the following Proposition, we will prove the function we constructed in (\ref{explicitconstructionsolu}) is the solution to HCMA restricting to $N'' \times D$.

\begin{pro}\label{propexplicitconstructionsolu}
    Let $ g_x$, $x\in N$ be a smooth family of holomorphic discs satisfying (\ref{holodiscpde}) and (\ref{holodiscimage}). Then, the function $\Phi$ constructed in (\ref{explicitconstructionsolu}) satisfies,
    \begin{align} \label{holoconditionfirstderivatives}
        \frac{\partial}{\partial \bar{\tau}} \bigg[ \Big(\frac{\partial}{\partial z_{i}} \Phi\Big) \big(z(x, \tau), \tau\big)\bigg] =0, \qquad \tau \in D. 
    \end{align} 
    More precisely, Let $\Lambda_\tau = G(\cdot, \tau)$, $\tau \in D$, which can be viewed a submanifold of $E$. The graph of $ \displaystyle\Big( \frac{\partial}{\partial z_{i}} \Phi\Big) \big(z(x, \tau), \tau\big) $ coincides with $\Lambda_\tau$, $\tau \in D$, and defines an exact Lagrangian submanifold of $E$ for each $\tau \in D$. Furthermore, $\Phi$ is a smooth solution to the homogeneous complex Monge-Amp\`ere equation in $N_{R_0 +1} \times D$,
    \begin{align}\label{HCMAnondegoutside}
    \begin{split}
        (i\partial\overline{\partial} \Phi\big)^{n+1} = 0, \quad \qquad &\text{ in } N' \times D,\\
        \Phi(z, \tau) = \rho(z, \tau),\  \qquad &\text{ in } N' \times \partial D,\\
        i\partial_X \overline{\partial}_X \Phi (z, \tau) >0,   \qquad &\text{ for } (z,\tau) \in N' \times D.
    \end{split}
    \end{align}
Conversely, if there is a nondegenerate solution $\Phi$ to the above HCMA equation in $N' \times D$. Then, there is a smooth family of holomorphic discs in $N \times D$ satisfying (\ref{holodiscpde}) and (\ref{holodiscimage}).
\end{pro}
 
\begin{proof}
The harmonic function in a disk can be written by an integral formula with the Poisson kernel,  
\begin{align} \label{poissonintformulasolu}
    \Phi (z(x,\tau), \tau) = \frac{1}{2\pi}\int_{0}^{2\pi} \Phi (z (x, e^{i\theta}), e^{i\theta})  P(\theta, \tau) d\theta,
\end{align}
where $P(\theta, \tau)$ is the Poisson kernel given by $\displaystyle P(\theta, \tau) = \re \bigg(\frac{1+ \tau e^{-i\theta}}{ 1- \tau e^{-i\theta}}\bigg) $. For simplicity, we introduce a new coordinates $w=(w_1,\ldots, w_n)$ of $N_{R_0}$ such that $\sigma_\tau (w) = z(w, \tau)$. Hence, 
\begin{align}\label{derofsolucompute}
\begin{split}
    \Big(\frac{\partial}{\partial z_i}\Phi\Big) (z(w,\tau), \tau) &= \frac{\partial}{\partial z_i} \big(\Phi(z, \tau)\big)\\  &= \frac{\partial w_j}{\partial z_i} \bigg|_{(z(w, \tau), \tau)} \frac{\partial}{\partial w_j}\big(\Phi(z(w,\tau), \tau)\big)  \\
    &\quad + \frac{\partial \bar{w}_j}{\partial z_i}\bigg|_{(z(w, \tau), \tau)} \frac{\partial}{\partial \bar{w}_j}\big(\Phi(z(w,\tau), \tau)\big). 
    \end{split}
\end{align} 
According to (\ref{poissonintformulasolu}), we have,
\begin{align}
     \frac{\partial}{\partial w_j}\big(\Phi(z(w,\tau), \tau)\big) &= \frac{1}{2\pi} \int_{0}^{2\pi} \frac{\partial}{\partial w_j}\big( \Phi (z (x, e^{i\theta}), e^{i\theta}) \big) P(\theta, \tau)  d\theta \nonumber\\
     \label{poissonintformuladerofsolution}
     & = \frac{1}{2\pi} \int_0^{2\pi}
     \frac{\partial z_q}{\partial w_j} \bigg|_{(z(w, e^{i\theta}), e^{i\theta})} \frac{\partial}{\partial z_q}\big(\Phi(z(w,e^{i\theta}), e^{i\theta})\big)
     P(\theta, \tau)  d\theta\\
     \label{poissonintformuladerofsolutionbar}
     & \quad + \frac{1}{2\pi} \int_0^{2\pi}
     \frac{\partial \bar{z}_q}{\partial w_j} \bigg|_{(z(w, e^{i\theta}), e^{i\theta})} \overline{\frac{\partial}{\partial z_q}\big(\Phi(z(w,e^{i\theta}), e^{i\theta})\big)}
     P(\theta, \tau)  d\theta
\end{align}
The key observation is that we can find simple explicit formulas for (\ref{poissonintformuladerofsolution}) and (\ref{poissonintformuladerofsolutionbar}), providing the data of the smooth family of holomorphic discs. For simplicity, we denote the integral term of (\ref{poissonintformuladerofsolution}) by $H_j(w, \tau)$, then the complex function $H_j(w, \tau)$ is the unique function satisfies the following equations, fixing $w \in N_{R_0}$,
\begin{align*}
    \Delta_{\tau} H_j(w, \tau) &=0, \qquad \tau  \in D \\
    H_j (w, \tau) =  \frac{\partial z_q}{\partial w_j} (w, &\tau)\cdot
    \Big(\frac{\partial}{\partial{z_q}} \rho_\tau\Big) (z(w, \tau), \tau), \qquad \tau \in \partial D.
\end{align*}
Notice that the function,
\begin{align*}
    \frac{\partial z_q}{\partial w_j} (w, \tau)\cdot
   \xi_{q} (w, \tau), \qquad \tau \in D
\end{align*}
is holomorphic in $D$ with the boundary condition,
\begin{align*}
    \frac{\partial z_q}{\partial w_j} (w, \tau)\cdot
   \xi_{q} (w, \tau) = \frac{\partial z_q}{\partial w_j} (w, &\tau)\cdot
    \Big(\frac{\partial}{\partial{z_q}} \rho_\tau\Big) (z(w, \tau), \tau), \qquad \tau \in \partial D.
\end{align*}
Hence, we get,
\begin{align*}
    H_j(w, \tau) = \frac{\partial z_q}{\partial w_j} (w, \tau)\cdot
   \xi_{q} (w, \tau), \qquad (w, \tau) \in N_{R_0} \times D.
\end{align*}.

Similarly, based on the anti-holomorphicity of the function $\displaystyle \frac{\partial \bar{z}_q}{\partial w_j} (w, \tau)\cdot
   \bar{\xi}_{q} (w, \tau)$, we can prove that
   \begin{align*}
       (\ref{poissonintformuladerofsolutionbar}) = \frac{\partial \bar{z}_q}{\partial w_j} (w, \tau)\cdot
   \bar{\xi}_{q} (w, \tau), \qquad (w, \tau) \in N_{R_0}\times D.
   \end{align*}
Therefore, we obtain 
\begin{align}\label{formuladerofsolu}
\begin{split}
    \frac{\partial}{\partial w_j}\big(\Phi(z(w,&\tau), \tau)\big)\\
    &= \frac{\partial z_q}{\partial w_j} (w, \tau)\cdot
   \xi_{q} (w, \tau) + \frac{\partial \bar{z}_q}{\partial w_j} (w, \tau)\cdot
   \bar{\xi}_{q} (w, \tau),
   \end{split}
   \qquad (w, \tau) \in N_{R_0}\times D.
\end{align}
Inserting (\ref{formuladerofsolu}) and its conjugate into (\ref{derofsolucompute}), we obtian,
\begin{align}
    \Big( \frac{\partial}{\partial z_i} \Phi \Big) ( z(w,\tau), \tau ) 
    &= \frac{ \partial w_j}{\partial z_i}(z(w, \tau), \tau) \Big(\frac{\partial z_q}{\partial w_j} (w, \tau) \cdot
   \xi_{q} (w, \tau) + \frac{\partial \bar{z}_q}{\partial w_j} (w, \tau) \cdot
   \bar{\xi}_{q} (w, \tau)\Big) \nonumber \\
   &\quad +\frac{\partial \bar{w}_j}{\partial z_i}(z(w, \tau), \tau) \Big(\frac{\partial \bar{z}_q}{\partial \bar{w}_j} (w, \tau)\cdot
   \bar{\xi}_{q} (w, \tau) + \frac{\partial {z}_q}{\partial \bar{w}_j} (w, \tau)\cdot
   {\xi}_{q} (w, \tau)\Big)\nonumber \\
   &= \xi_q(w, \tau) \Big(\frac{\partial w_j}{\partial z_i}\frac{\partial z_q}{\partial w_j} +\frac{\partial \bar{w}_j}{\partial z_i}\frac{\partial z_q}{\partial \bar{w}_j}\Big) \nonumber \\
   & \quad + \bar{\xi}_q(w, \tau) \Big(\frac{\partial w_j}{\partial z_i}\frac{\partial \bar{z}_q}{\partial w_j} +\frac{\partial \bar{w}_j}{\partial z_i}\frac{\partial \bar{z}_q}{\partial \bar{w}_j}\Big) \nonumber\\
   & = \xi_i (w, \tau). \label{derofsolugraph}
\end{align}
Recall that $\xi_{i}(w, \tau)$ is a $\tau$-holomorphic function based on (\ref{holodiscpde}). Hence, we complete the proof of (\ref{holoconditionfirstderivatives}). According to (\ref{derofsolugraph}), it's direct to see that the graph of  $ \displaystyle\Big( \frac{\partial}{\partial z_{i}} \Phi\Big) \big(z(x, \tau), \tau\big) $ coincides with $\Lambda_\tau$, $\tau \in D$.

To see that $\partial\overline{\partial}\Phi$ is nondegenerate in the space direction, we consider the pull-back of the canonical complex symplectic form, $\Xi$, to $N_{R_0} \times D$.  We have that $G^*(\Xi)$, by restricting to space direction, is constant along each leaf, which implies the nondegeneracy of $\partial\overline{\partial}\Phi$ in the space direction.
The proof can be found in Donaldson \cite[Proposition 1]{donaldson2002holo}.  

It suffices to check that $\Phi$ is the solution to the HCMA equation, (\ref{HCMAnondegoutside}). The following calculation is the simplified version of Semme's \cite{semmes1992complex}. To prove $\displaystyle\big(i\partial \overline{\partial} \Phi\big)^{n+1} =0$, we only need to check,
\begin{align} \label{solutoHCMAequoutsidecalc}
    \frac{\partial^2 \Phi}{\partial\bar{\tau} \partial \tau}-  \frac{\partial^2 \Phi}{\partial\bar{\tau} \partial z_i}\Phi^{i \bar{j}} \frac{\partial^2 \Phi}{\partial\bar{z}_j \partial \tau} =0,
\end{align}
where $(\Phi^{i\bar{j}})$ is the inverse matrix of $\displaystyle \Big(\frac{\partial^2 \Phi}{\partial \bar{z}_j\partial {z}_i}\Big) $.
Based on the equation (\ref{holoconditionfirstderivatives}), we have,
\begin{align}\label{solutoHCMAequoutsidecalc1}
    \frac{\partial^2 \Phi}{\partial \bar{z}_j\partial {z}_i}  \cdot \frac{\partial \bar{z}_j}{ \partial \bar{\tau} } + \frac{\partial^2 \Phi}{\partial \bar{\tau}\partial z_i} =0.
\end{align}
Hence, we have
\begin{align} \label{solutoHCMAequoutsidecalc2}
    \frac{\partial \bar{z}_j}{ \partial \bar{\tau} } = - \Phi^{i \bar{j} } \frac{\partial^2 \Phi}{\partial \bar{\tau}\partial z_i}.
\end{align}
According the construction of $\Phi$, (\ref{explicitconstructionsolu}), we have,
\begin{align*}
    0 &= \frac{\partial}{\partial {\tau}} \frac{\partial}{\partial \bar{\tau}} \big(\Phi(z(x, \tau), \tau)\big)\\
    & = \frac{\partial^2 \Phi}{\partial\bar{\tau} \partial \tau} 
    + \frac{\partial^2 \Phi}{\partial \bar{z}_j\partial \tau}  \cdot \frac{\partial \bar{z}_j}{ \partial \bar{\tau} }
    + \frac{\partial^2 \Phi}{\partial \bar{\tau}\partial z_i}  \cdot \frac{\partial {z}_i}{ \partial {\tau} }
    +  \frac{\partial^2 \Phi}{\partial \bar{z}_j\partial {z}_i}  \cdot \frac{\partial \bar{z}_j}{ \partial \bar{\tau} }\cdot\frac{\partial {z}_i}{ \partial \tau }
\end{align*}
Inserting (\ref{solutoHCMAequoutsidecalc1}), (\ref{solutoHCMAequoutsidecalc2}), we obtain the equation
\begin{align*}
    0 
    &= \frac{\partial^2 \Phi}{\partial\bar{\tau} \partial \tau} - \frac{\partial^2 \Phi}{\partial \bar{z}_j\partial \tau} \Phi^{i \bar{j} } \frac{\partial^2 \Phi}{\partial \bar{\tau}\partial z_i} 
    + \frac{\partial {z}_i}{ \partial \tau } \Big( \frac{\partial^2 \Phi}{\partial \bar{z}_j\partial {z}_i}  \cdot \frac{\partial \bar{z}_j}{ \partial \bar{\tau} } +  \frac{\partial^2 \Phi}{\partial \bar{\tau}\partial z_i}\Big)\\
   &=  \frac{\partial^2 \Phi}{\partial\bar{\tau} \partial \tau}-  \frac{\partial^2 \Phi}{\partial\bar{\tau} \partial z_i}\Phi^{i \bar{j}} \frac{\partial^2 \Phi}{\partial\bar{z}_j \partial \tau}
\end{align*}
Hence, we complete the proof.
\end{proof}

\subsection{The Linearized Problems} \label{seclocalexholodiscs}  In this subsection, we establish the local existence of a foliation by holomorphic discs through a small perturbation of the boundary data on $N' \times \partial D$ for an arbitrary complex dimension $\dim_\CC N' = n$. Recall that the foliation by a family of holomorphic discs is given by the projection down a family of holomorphic discs $g_w(\tau) = G (w, \tau)$ to $N'$. According to the discussion in Subsection \ref{secholodiscHCMA}, the existence of foliation by holomorphic discs is equivalent to solving the equation (\ref{holodiscpde}), (\ref{holodiscimage}) for $G(w, \tau) = ({z} (w, \tau), {\xi}(w,\tau))$. This subsection is dedicated to solving the PDEs (\ref{holodiscpde}), (\ref{holodiscimage}) by introducing a small perturbation of the boundary data of trivial foliation.

\subsubsection{Free boundary problem and linearized equations at trivial foliation} The equation (\ref{holodiscpde}), (\ref{holodiscimage}) is a nonlinear free boundary $\overline{\partial}$-problem. Now, we recall the standard free boundary $\overline{\partial}$-problem on a unit disk $D \subseteq \CC$ and introduce the linearized problem of (\ref{holodiscpde}).

The family of free boundary $\overline{\partial}$-problems with parameter space $N \subseteq \CC^n$ can be described by the following PDE:
 \begin{align}\label{dbardisk}
 \begin{split}
     \overline{\partial}_\tau u =0, \qquad &\text{ in } N \times  D,\\
     \re u = f, \qquad &\text{ on } N \times \partial D.
 \end{split}
 \end{align}
 Here we assume $f$ is a smooth real function defined on $N \times \partial D$. Based on complex Fourier expansion on the unit circle, we can explicitly write down the solutions to (\ref{dbardisk})
 \begin{align*}
 u(x,\tau) = \frac{1}{2\pi} \int^{2\pi}_0 f(x, e^{i\theta}) \frac{1+e^{-i\theta}\tau}{1- e^{-i\theta} \tau} d\theta + i A(x),
 \end{align*}
 where $A(x)$ is an arbitrary real function on $N$. The function $A(x)$ is the imaginary part of $u$ at $\theta= 0$. In what follows, we refer to $A(x)$ as a fixed-point data (or a fixed point condition). Moreover, once $A(x)$ is fixed, the solution $u$ is uniquely determined.
 It is easily observed that the real part of the solution, $\re u$, is harmonic by integrating with the Poisson kernel. The imaginary part $\im u$ is a harmonic conjugate of $\re u$. By restricting $\im u$ on the boundary $\partial D$, the function $\im u |_{\partial D}$ is the Hilbert transform of $f$ plus a fixed-point data. The Hilbert transform is defined by the principal value of a singular integral:
 \begin{align*}
     \Hilb f (\xi) = \frac{1}{\pi} \lim_{\varepsilon \rightarrow 0} \int_{\varepsilon <|\xi-\theta|< \pi} \cot \big(\frac{\xi-\theta}{2}\big) f(\theta) d\theta
 \end{align*}

 It is well-known that the Hilbert transform keeps the H\"older classes on the unit circle;  precisely, $\Hilb: \C^{k,\alpha} (\TT^1) \rightarrow \C^{k,\alpha} (\TT^1) $ is a bounded linear operator (refers to \cite[Charpter III]{garnett2006bounded} for details). However, given a family of functions defined on $ \TT^1$ in a parameter space $N \subseteq \CC^n$, the Hilbert transform on $\partial D$ might cause the loss of regularity in the parameter direction. To describe the loss of regularity, we introduce the notion of a \textit{modulus of continuity} for a function $g$ defined on a metric space $(M, d)$:
 $$\upsilon_{g} (t) = \sup_{d(p,p')\leq t} |g(p)- g(p')|.$$
Then, we have the following lemma:
 \begin{lem}\label{lemhilbtransesti}
     Given a function $f(x, \tau) \in \C^{k,\alpha}(N \times \partial D) $, the Hilbert transform of $f(x, \tau)$ in $\tau$ satisfies:
     \begin{align*}
         \sup_{|x-x'| \leq t}\big|\nabla^k \big(\Hilb_\tau f\big)(x,\tau) -\nabla^k \big(\Hilb_\tau f\big)(x', \tau)\big| \leq C t^\alpha|\log t| ||f||_{k,\alpha; N \times \partial D}.
     \end{align*}
     Furthermore, for each $0<\beta < \alpha$,
     \begin{align*}
         ||\Hilb_\tau f||_{k,\beta; N \times \partial D} \leq C (\alpha-\beta)^{-1} ||f||_{k,\alpha; N \times \partial D}
     \end{align*}
     where $C$ is a uniform constant depending on $n$, $k$ and $\alpha$.
 \end{lem}
 \begin{proof}
     Given a function $f\in \C^{k,\alpha}(N \times \TT^1)$, we discuss the continuity of modulus in both parameter and circle directions as follows:
     \begin{align*}
         \upsilon^N_{f} (t) = \sup_{|x- x'|\leq t,\ \tau \in \TT^1} |f(x, \tau) - f(x', \tau)|, \qquad \upsilon^{\TT^1}_{f} (t) = \sup_{x\in N,\ |\tau-\tau'| \leq t} |f(x, \tau) - f(x, \tau')|,
     \end{align*}
     where the distances are induced by the standard metrics of $N\subseteq \CC^n$ and $\TT^1$.
     
     If $f \in \C^{k,\alpha} (N \times \TT^1)$, note that the Hilbert transform commutes with differential operators up to order $k$ in both parameter and circle directions. It suffices to prove the result for the case $k=0$. For the sake of simplicity, we denote $\tilde{f} = \Hilb_\tau f$. 
     In the circle direction, according to \cite[Charpter III, Theorem 1.3]{garnett2006bounded}, we have the following inequality: for a sufficiently small constant, $0<\delta \ll 1 $,
     \begin{align} \label{esticircle}
         \upsilon^{\TT^1}_{\tilde{f}(x, \cdot)} (\delta) \leq C (\delta^\alpha +  \delta) ||f(x, \cdot)||_{0,\alpha; \TT^1}, \qquad \text{for  } x \in N.
     \end{align}
     where $ C$ is a uniform constant depending only on $\alpha$. In the parameter direction, we have
     \begin{align} \label{estipara}
     \begin{split}
         |\tilde{f} (x,\tau) - \tilde{f}(x',\tau)| \leq & \ C \int_{\tau-\delta}^{\tau + \delta} \frac{|f(x, \theta) - f(x, \tau) |}{|\theta -\tau| }  d\theta  +C \int_{\tau-\delta}^{\tau + \delta} \frac{|f(x', \theta) - f(x', \tau) |}{|\theta -\tau| }  d\theta 
         \\
         &+ \frac{1}{\pi} \int_{\delta<|\theta-\tau|< \pi} {|f(x, \theta) - f(x', \theta)|} \cot \big(\frac{\theta-\tau}{2} \big)d\theta.
    \end{split}
     \end{align}
     Hence, we have 
     \begin{align*} 
         \upsilon^N_{\tilde{f}} (\delta) \leq C (1- \log \delta) \delta^{\alpha} ||f||_{0,\alpha; N \times \TT^1}.
     \end{align*}
     The fact, $\upsilon^N_{\tilde{f}} (\delta)/\delta^\alpha = O(-\log \delta)$, implies  $\tilde{f}$ might lose a small regularity in the parameter directions. If we take any constant $\beta$ with $0<\beta < \alpha$, \eqref{estipara} implies
     \begin{align} \label{estipara1}
         \upsilon^N_{\tilde{f}} (\delta) \leq C (\alpha-\beta)^{-1} \delta^{\beta} ||f||_{0,\alpha; N \times \TT^1}.
     \end{align}
     From \eqref{esticircle} and \eqref{estipara1}, we have 
     \begin{align*}
         ||\tilde{f}(x, \cdot)||_{0,\beta; \TT^1}  \leq C ||\tilde{f} (x, \cdot)||_{0,\alpha; \TT^1} \leq C || f (x, \cdot)||_{0,\alpha; \TT^1} \leq C ||f||_{0,\alpha; N \times \TT^1}, \qquad \text{ for all } x\in N,
     \end{align*}
     and 
     \begin{align*}
         ||\tilde{f} (\cdot, \tau)||_{0,\beta; N} \leq C (\alpha-\beta)^{-1} ||f||_{0,\alpha; N \times \TT^1}, \qquad \text{for all } \tau \in \partial D,
     \end{align*}
     By Lemma \ref{lemfamilylaplace} (Bernstein's Theorem), we complete the proof of the lemma.
\end{proof}

\begin{rem} \label{remcounterex}
    In this remark, we construct a counterexample to see that the Hilbert transform cannot be a bounded operator in $\C^{k,\alpha} (N \times \TT^1)$. Let $N = (-\pi,\pi) \subseteq \RR$. Consider the following function, $f(x, \theta)$ defined on $N \times \TT^1$:
    \begin{align*}
    f(x, \theta) = \begin{cases}
        -s(|\theta|)|\theta|^\alpha, \qquad &|x| \geq |\theta|, \quad 0\leq \theta\leq \pi;\\
        -s(|\theta|)|x|^\alpha, \qquad &|x| \leq |\theta|, \quad 0\leq \theta\leq \pi;\\
        s(|\theta|)|\theta|^\alpha, \qquad &|x| \geq |\theta|, \ -\pi\leq \theta\leq 0;\\
        s(|\theta|)|x|^\alpha, \qquad &|x| \leq |\theta|, \ -\pi\leq \theta\leq 0,
        \end{cases}
    \end{align*}
    where $0 \leq s(t) \leq 1$ is a smooth scaling function defined on $[0,\pi]$ such that $s(t) \equiv 1$ for $t \in [0,1]$ and $s(t) \equiv 0$ near $t = \pi$. 
    One can check that $||f||_{0,\alpha; N \times \TT^1}\leq C $ for a uniform constant $C$. If we denote $\tilde{f} = \Hilb_\theta f$, for $0<|x|\ll1$, we have
    \begin{align*}
        |\tilde{f} (x, 0) - \tilde{f}(0, 0)| &= \frac{2}{\pi}\lim_{\varepsilon \rightarrow 0} \int_{\varepsilon<\theta< \pi}  \big(f(x, \theta)-f(0,\theta)\big) \cot\Big(\frac{\theta}{2}\Big) d\theta
        \\
        &=\frac{2}{\pi}\int_{0}^x |\theta|^\alpha \cot\Big(\frac{\theta}{2}\Big) d\theta + \frac{2}{\pi} \int_{x}^\pi s(|\theta|) |x|^\alpha \cot\Big(\frac{\theta}{2}\Big) d\theta
        \\
        & \geq c |x|^\alpha \big(  -\log |x| \big)
     \end{align*}
    It is clear that \( [\tilde{f}]_{0,\alpha; N \times \mathbb{T}^1} \) is unbounded. More precisely, the \(\alpha\)-Hölder seminorm of \(\tilde{f}\) in the \(N\)-direction is unbounded at \(\theta=0\).

If we instead consider the \(\alpha\)-Hölder norm in the \(N\)-direction at $(0, \theta_0)$ for some fixed \(\theta_{0}\neq 0\), then for sufficiently small \(x\) with \(0<x<|\theta_{0}|/2\) we have  
\begin{align*}
\big|\tilde{f}(x,\theta_{0})-\tilde{f}(0,\theta_{0})\big|
&=\frac{1}{\pi}\bigg|\lim_{\varepsilon\to 0}\int_{\varepsilon<|\theta|<\pi}
\big(f(x,\theta)-f(0,\theta)\big)
\cot\!\Big(\frac{\theta_{0}-\theta}{2}\Big)\,d\theta\bigg|\\[0.7em]
&\le 
 \frac{1}{\pi}\Big| \text{p.v.}\!\int_{|x|}^{2\theta_{0}-|x|} x^{\alpha} 
\cot\!\Big(\frac{\theta_{0}-\theta}{2}\Big)\,d\theta\Big|
+ C|x|^{\alpha}\!\int_{\frac{|\theta_{0}|}{2}}^{\pi}\frac{1}{\theta}\,d\theta\\[0.7em]
&\le C\Bigl(1-\log\frac{|\theta_{0}|}{2}\Bigr)|x|^{\alpha}.
\end{align*}
From this calculation we see that the \(\alpha\)-Hölder norm of \(\tilde{f}\) in the \(N\)-direction at each fixed \(\theta_{0}\neq 0\) is bounded, but the bound diverges as \(\theta_{0}\to 0\). 

Consider the Riemann-Hilbert problem \eqref{dbardisk} with the boundary data $f$ given as above. As discussed, the Hilbert transform of \(f\) in \(\theta\) exhibits a single singularity in its \(\alpha\)-H\"older seminorm at \(\theta=0\) on the circle when \(x=0\). The imaginary part of the solution then has boundary values \(\tilde{f}(x,\theta)+A(x)\), where \(A(x)\) is an arbitrary fixed-point function. Even if one chooses \(A(x)\) to cancel the singularity at \(\theta=0\) (for instance \(A(x)=-\tilde{f}(x,0)\), which is itself singular at \(x=0\)), the function \(\tilde{f}(x,\theta)+A(x)\) remains singular in the \(N\)-direction for every \(\theta\neq 0\). Thus no choice of fixed-point function restores uniform \(\alpha\)-H\"older regularity in \(x\) along \(\partial D\); the loss of regularity persists.
\end{rem}

\begin{rem}
    Let $f \in\C^{k,\alpha} (N \times \TT^1)$, and let $\tilde{f}$ be the Hilbert transform of $f(x, \tau)$ in $\tau$.
    We then define the following singular set at $x \in N$: $$\mathcal{S}(\tilde{f}, x) = \Big\{\theta \in \TT^1;~ \sup_{x' \in N,~ x' \ne x } \frac{|\tilde{f}(x, \theta) - \tilde{f}(x', \theta)|}{|x-x'|^{\alpha}} = +\infty \Big\}.$$
    In Remark \ref{remcounterex}, we observed that $\mathcal{S} (\tilde{f}, 0)$ contains exactly one point. In general, the set, $\mathcal{S}(\tilde{f}, x)$, is small; precisely, under the standard measure of $\TT^1$, the measure of $\mathcal{S}(\tilde{f}, x)$ is expected to be zero. Further details will be discussed in the next subsection.
\end{rem}

 Note that the real and imaginary parts of the solution $u$ to (\ref{dbardisk}) are obtained by solving Dirichlet problems such that $\re u |_{ N \times \partial D} = f$ and $\im u |_{ N \times \partial D} = \Hilb f + A(x)$. 
 According to Lemma \ref{lemhilbtransesti} and Lemma \ref{lemplaestiwrtbdry}, the interior regularity of $u$ in $N \times D$ follows from the regularity of $f$ and $\Hilb f + A$ in the $N \times \partial D$. Then, if we assume the prescribed fixed points data $A(x)$ has at least $\C^{k,\beta}$ regularity,
 \begin{align} \label{estisolutionbabyversionfreebdry}
     ||u||_{k,\beta; N \times D} \leq C \big\{ (\alpha - \beta)^{-1} ||f ||_{k,\alpha; N \times \partial D} + ||A||_{k,\beta; N} \big\},
 \end{align}
 where $C$ is a uniform constant depending on $n$, $k$ and $\alpha$.

Here, we describe the linearized problem of (\ref{holodiscpde}) as follows. Consider the space of smooth families of holomorphic discs parametrized by $N$:
\begin{align} \label{defspaceholodiscs}
    \A =\{ G \in \C^\infty(N \times D, E); \ \overline{\partial}_\tau G (w, \tau) =0,\  H(w, -i)= \pi\circ G (w, -i) =w \}
\end{align}
The free boundary conditions are the smooth families of exact Lagrangian subspaces of $E$ parametrized by $\partial D$. The notation $\Lambda_{\tilde{\psi}}$ denotes the graph of $\partial \tilde{\psi}$ in $E$. Consider a family of smooth functions parametrized by $\partial D$, $\tilde{\psi}_\tau = \widetilde{\Psi} (\cdot, \tau)$. For the sake of simplicity, we denote $\Lambda_{\widetilde{\Psi}} = \{(p,\tau) \in E \times \partial D; p \in \Lambda_{\tilde{\psi}_{\tau}}\}$. The space of free boundary conditions can be described as follows:
    \begin{align*}
    \B = \{ \Lambda_{\widetilde{\Psi}} \subseteq  E \times \partial D; \ \tilde{\psi}_\tau (z)= \widetilde{\Psi}(z, \tau)  \in \C^{\infty} (N' \times \partial D), \ i\partial \overline{\partial} \tilde{\psi}_\tau = i\partial \overline{\partial} \widetilde{\Psi}(\cdot, \tau)>0 \}.
\end{align*}
Let $\D = \C^{\infty} (N \times \partial D, \CC^n)$ and let $T:\A \times \B \rightarrow \D$ be a nonlinear operator measuring the difference between the image of $G$ in $E$ and the free boundary condition $\Lambda_{\widetilde{\Psi}}$. Precisely, for $G \in \mathcal{A}$, $\Lambda_{\widetilde{\Psi}} \in \mathcal{B}$, $T(G, \Lambda_{\widetilde{\Psi}})$ is defined as
\begin{align} \label{defopdiffholodiscbdry}
    T(G, \Lambda_{\widetilde{\Psi}})(w, \tau) = \xi(w,\tau) - \partial \widetilde{\Psi} (z(w, \tau), \tau))
\end{align}
Indeed, we give another description of the family of holomorphic discs satisfies (\ref{holodiscpde}).

\begin{lem} Let $G \in \mathcal{A}$ and $\Lambda_{\widetilde{\Psi}} \in \mathcal{B}$. Then, $(G, \Lambda_{\widetilde{\Psi}})$ is in the kernel of $ T$
    if and only if 
$G(w,\tau) = (z(w, \tau), \xi(w, \tau))$ is a solution to the equation (\ref{holodiscpde}) with the boundary data given by $ \tilde{\psi}_{\tau} = \widetilde{\Psi}(\cdot, \tau)$.
\end{lem}

Fixing a family of holomorphic discs $G  = (z(w, \tau), \xi(w, \tau)) \in \mathcal{A}$, and $\widetilde{\Psi}$ in $\C^\infty (N'\times D)$ with $\Lambda_{\widetilde{\Psi}} \in \mathcal{B}$ such that $(G, \Lambda_{ \widetilde{\Psi}})$ belongs to $\ker T$, the tangent space of $\A$ is independent of precise base points and given by,
\begin{align*}
    \A' = \{\widehat{G} \in \C^\infty(N \times D, E); \ \overline{\partial}_\tau \widehat{G} (w, \tau) =0,\  \pi \circ \widehat{G} (w, -i)= 0 \}.
\end{align*}
Consider the derivative of $T$ in the direction of $\A$. If we write 
\begin{align*}
\widehat{G}(w, \tau) = (\hat{z}(w, \tau); \hat{\xi}(w, \tau)) = (\hat{z}_1 (w,\tau), \ldots \hat{z}_n(w, \tau); \hat{\xi}_1(w, \tau), \ldots, \hat{\xi}_{n}(w, \tau)), 
\end{align*}
then the derivative of $T$ in the direction of $\A$ can be written in the complex coordinates by, 
\begin{align*}
    \big(D_{\A} T |_{(G; \partial\widetilde{\Psi})} (\widehat{G})\big)_{i} = \hat{\xi}_i (w, \tau) - \sum_{j=1}^n \Big\{ (\partial_i\overline{\partial}_{j} \widetilde{\Psi} ) (z(w, \tau), \tau) \overline{\hat{z}}_j + (\partial_i \partial_j \widetilde{\Psi})(z(w, \tau), \tau) \hat{z}_j\Big\}, 
\end{align*}
for each $\tau \in \partial D$, where $ (D_{\A} T |_{(G; \partial\widetilde{\Psi})} (\widehat{G}))_{i}$ represents the $i$-th coordinate of the standard bundle coordinates of $E$.
Then, the linearized problem of (\ref{holodiscpde}) is given as follows,
    \begin{align}\label{holodisclinearpde}
    \begin{split}
        &\overline{\partial}_\tau \hat{z}_i (w, \tau) = \overline{\partial}_\tau \hat{\xi} (w, \tau) = 0, \hspace{5.2cm} \text{ in }  N \times D,\\
        \begin{split}
         &\hat{\xi}_i(w,\tau) - \big(\partial_i \overline{\partial}_{{j}} \widetilde{\Psi} \big)(z(w, \tau), \tau) \overline{\hat{z}}_j(w, \tau)\\ 
         &\hspace{2cm}- \big(\partial_{i}\partial_{j} \widetilde{\Psi}\big) (z(w, \tau), \tau) \hat{z}_j(w, \tau)=\bff_i(w,\tau),
         \end{split}
         \hspace{0.5cm}  \text{ in } N \times \partial D,\\
         &\hat{z}_i(w, -i) = 0,  \hspace{7.5cm}  \ w  \in N,
    \end{split}
\end{align}
where $\bff= (\bff_1, \ldots, \bff_n) \in \C^{\infty}(N \times \partial D, \CC^n)$.  

The linearized problem is a generalized version of the free boundary $\overline{\partial}$-problem. In the following lemma, we will see that the linearized problem at the trivial foliation can be reduced to a double-free boundary problem or, a special case of Riemann-Hilbert problem. This remaining part of the subsection focuses on solving the linearized problem at families of holomorphic discs with trivial foliation. Precisely, consider the linearized problem at  \begin{align} \label{bdryconditrivialfoli}
    \widetilde{\Psi}_0 (z, \tau) = \rho(z), \qquad \text{ in } N'\times \partial D,
\end{align}
where $\rho$ is the potential function of the reference K\"ahler form $\omega$ in $N'$,
and at the family of holomorphic discs with trivial foliation $G_0 \in \mathcal{A}$, 
\begin{align}\label{cxdimntrivialfoli}
\begin{split}
    (G_0)(w,\tau) &= (z_{0,1}(w,\tau),\ldots, z_{0,n}(w,\tau), \xi_{0,1}(w,\tau), \ldots, \xi_{0, n}(w, \tau))\\
    & = (w_1, \ldots, w_n, (\partial_{z_1}\rho)(w), \ldots, (\partial_{z_n} \rho)(w)),
\end{split}
\qquad \text{ in } N \times D.
\end{align}
So, the coefficients in the second equation of (\ref{holodisclinearpde}) are independent of $\tau$. Precisely, 
\begin{align}\label{constantcoeffcon}
\begin{split}
    \big(\partial_i \overline{\partial}_{{j}} \widetilde{\Psi} \big)(z(w, \tau), \tau) &= \big(\partial_i \overline{\partial}_{{j}} \widetilde{\Psi}_0 \big)(w),\\
    \big(\partial_i {\partial}_{{j}} \widetilde{\Psi} \big)(z(w, \tau), \tau) &= \big(\partial_i {\partial}_{{j}} \widetilde{\Psi}_0 \big)(w).
\end{split}
\end{align}
The details of solving (\ref{holodisclinearpde}) with coefficients (\ref{constantcoeffcon}) can also be found in ([Chen-Feldman-Hu], Lemma A.3), and we summarize in the following lemma,

\begin{lem} \label{IFTlinearpart} Let $A= (A_{i\bar{j}}) \in \C^{\infty}(N, \CC^{n \times n})$ be a nowhere degenerate hermitian matrix satisfying
\begin{align*}
    \det A \geq \sigma>0
\end{align*}
and $ B= (B_{ij}) \in \C^{\infty} (N, \CC^{n\times n} )$, a symmetric matrix. For each $\bff = (\bff_1,\ldots,\bff_n) \in \C^{\infty}(B_1 \times \partial D, \CC^n)$, there exists a unique solution, $(\hat{z}_1,\ldots, \hat{z}_n; \hat{\xi}_1, \ldots, \hat{\xi}_n)$, with each $\hat{z}_i$, $\xi_i\in \C^{\infty} (N \times D, \CC)$, to the follow equation
 \begin{align}\label{holodisclinearpde}
    \begin{split}
        &\overline{\partial}_\tau \hat{z}_i (w, \tau) = \overline{\partial}_\tau \hat{\xi} (w, \tau) = 0, \hspace{5.2cm} \text{ in }N \times D,\\
         &\hat{\xi}_i(w,\tau) - A_{i\bar{j}}(w) \overline{\hat{z}}_j(w, \tau)- B_{ij}(w) \hat{z}_j(w, \tau)=\bff_i(w,\tau),
         \hspace{0.3cm}  \text{ in } N \times \partial D,\\
         &\hat{z}_i(w, -i) = 0,  \hspace{7.5cm}  \ w  \in N,
    \end{split}
\end{align}
Moreover, for each $k \geq 2, \ k\in \ZZ$ and $0<\beta< \alpha <1$, there is a uniform constant $C$ only depending on $n$, $k$, $\alpha$, $\sigma^{-1},||A||_{k,\beta; N}$, $||B||_{k, \beta; N}$ such that
\begin{align} \label{estiperturbholodiscs}
   ||(\hat{z},\hat{\xi})||_{k,\beta; N \times D} \leq C (\alpha-\beta)^{-1} ||\bff||_{k, \alpha; N \times \partial D}.
\end{align}
\end{lem}

\begin{proof}
The key idea is to decouple the equation (\ref{holodisclinearpde}) to double free boundary $\overline{\partial}$ problems by introducing the following functions, (see also (Chen-Feldman-Hu~\cite{CHEN2020108603}, Lemma A.3)),
\begin{align}\label{defg1g2}
\begin{split}
    (\hat{g}_{1})_{i} (w, \tau) = \hat{\xi}_i (w, \tau) - B_{i j}(w) \hat{z}_j (w, \tau) - \overline{A_{i \bar{j}}}(w)  \hat{z}_j (w, \tau),  \\
    (\hat{g}_2)_i (w,\tau) = \hat{\xi}_i (w, \tau)- B_{ij} (w) \hat{z}_j(w, \tau) + \overline{A_{i \bar{j}}} (w) \hat{z}_j (w, \tau). 
\end{split}
\end{align}
Then, the equation (\ref{holodisclinearpde}) can be rewritten as follows,
\begin{align} \label{equstandardfbpros}
\begin{split}
    &\overline{\partial}_\tau \hat{g}_1 (w, \tau) = \overline{\partial}_\tau \hat{g}_2 (w, \tau) = 0, \hspace{1.8cm} \text{ in } N \times D,\\
    &\re( \hat{g}_1) = \re \bff, \hspace{4.1cm} \text{ in } N \times \partial D,\\
    &\im (\hat{g}_2) = \im \bff,  \hspace{4.1cm}\text{ in } N \times \partial D,\\
    &\hat{g}_1 (w, -i) = \hat{g}_2 (w, -i) = \bff(w,-i), \hspace{1cm} \text{ for } w \in N.
\end{split}
\end{align}
There is a unique solution $\hat{g}_1$, $\hat{g}_2$ to the above equation. According to the estimate (\ref{estisolutionbabyversionfreebdry}), $\hat{g}_1$ and $\hat{g}_2$ satisfies that,  
\begin{align} \label{estig1g2D}
    ||(\hat{g}_1,\hat{g}_2)||_{k,\beta; N \times D} \leq C(\alpha-\beta)^{-1}||\bff||_{k,\alpha; N \times \partial D}
\end{align}
where $C$ is a uniform constant independent of $w$ and depending only on $n, \ k, \ \text{and } \alpha$.
Then, the solution, $(\hat{\xi}(w, \tau), \hat{z}(w, \tau))$, of (\ref{holodisclinearpde}) is given by
\begin{align} \label{inverserepzxibyg}
\begin{split}
    \hat{z} &= \frac{1}{2} (\bar{A})^{-1} (\hat{g}_2- \hat{g}_1),\\
    \hat{\xi} &= \frac{1}{2} (\hat{g}_1 + \hat{g}_2) + \frac{1}{2} B (\bar{A})^{-1} (\hat{g}_2- \hat{g}_1).
\end{split}
\end{align}
By inverse transformation \eqref{inverserepzxibyg}, along with \eqref{estig1g2D}, we complete the proof of (\ref{estiperturbholodiscs}). The uniqueness follows from the uniqueness of the free boundary $\overline{\partial}$-problem 
under a fixed point condition. \end{proof}

\begin{rem}\label{remregg1g2}
    The estimates in \eqref{estig1g2D} can be stated more precisely by distinguishing the regularity in the $N$-direction from that in the $D$-direction. By applying the estimate for the Hilbert transform in the $\partial D$ direction (see \ref{esticircle}), and together with solving the Dirichlet problems separately for the real and imaginary parts, we obtain the following estimates for $g_{1,2}$ in the $ D$ direction:
\begin{align}\label{estiRHbdaryporblemD}
||g_1(x, \cdot)||_{k,\alpha; D}, \quad ||g_2(x, \cdot)||_{k,\alpha; D} \leq C ||\bff(x, \cdot)||_{k,\alpha; \partial D}, \qquad \text{for each } x \in N,
\end{align}
where the constant $C$ is uniform and independent of $x$. By Lemma \ref{lemhilbtransesti}, for any $x, x' \in N$ with $|x - x'| \leq t$, the regularity of $g_{1,2}$ in the $N$-direction can be described as follows:
\begin{align}\label{estiRHbdaryporblemN}
\sup_{|x - x'| \leq t,\ \tau \in D} \left| \mathrm{D}^{\mathbf{k}} g_i(x', \tau) - \mathrm{D}^{\mathbf{k}} g_i(x, \tau) \right| \leq C t^\alpha \big| \log t \big| \cdot||\bff||_{k,\alpha; N \times \partial D}, \qquad i=1,2. 
\end{align}
where $\mathrm{D}^{\mathbf{k}}$ denotes a differential operator with multi-indices $\mathbf{k}$ of order $|\mathbf{k}| = k$ with respect to the local coordinates on $N \times D$, and $C$ is a uniform constant. One can easily observe that the estimates \eqref{estiRHbdaryporblemD} and \eqref{estiRHbdaryporblemN} together imply \eqref{estig1g2D}. Moreover, the blow-up rate of the coefficient in \eqref{estig1g2D} as $\beta \to \alpha$ is sharp as
\[
\sup_{0\leq t\leq 1} t^{\alpha-\beta} |\log t| = C (\alpha -\beta)^{-1}.
\]
In the following subsections, we study the regularity of the linearized problem near the trivial data. We will show that if a family of elliptic systems is solved in the $D$ direction—exhibiting a loss of regularity in the $N$ direction with the same behavior as \eqref{estiRHbdaryporblemN} at $(G_0, \partial \Psi_0)$—then, in a sufficiently small neighborhood of $(G_0, \partial \Psi_0)$, the estimate of the linearized problem in H\"older norms satisfies the same blow-up rate $(\alpha - \beta)^{-1}$ as $\beta \to \alpha$.
\end{rem}

\subsubsection{BMO norms and the loss of regularity in H\"older spaces} In the last subsection, we observe the loss of regularity in H\"older norms in the parameter direction when we solve the linear elliptic boundary problems. By Lemma \ref{lemhilbtransesti}, the key point is the following: given a function $f \in \C^{k,\alpha} ( N \times \TT^1)$, 
\begin{align}  \label{defckalphainN}
     |h_{\tilde{f}}^{\mathbf{k},\alpha} (\tau; x,x')| :=\frac{|\mathrm{D}^{\mathbf{k}}\tilde{f}(x', \tau) - \mathrm{D}^{\mathbf{k}}\tilde{f} (x, \tau)|}{|x'- x|^\alpha} \quad \sim \quad \big|\log |x'-x|\big|
\end{align}
is unbounded in $L^\infty$ norm as $x' \rightarrow x$. In this part, we reinterpret this phenomenon from a different perspective, namely through the framework of \emph{bounded mean oscillation} (BMO) spaces. There is an extensive literature on this subject; here we only mention a few classical references that may help the reader gain further insight \cite{David1984, grafakos2009modern, JNthm1961, SteinbookHA}. In this subsection, we avoid developing subtle technical details on the theorems concerning BMO spaces. Instead, we present only the essential definitions and lemmas that will be used in establishing the perturbation theorem for the linear elliptic boundary problem discussed in the next subsection.

\begin{define}
    Let $u$ be a locally integrable function defined in $\RR^n$. The mean oscillation (BMO) norm of $u$ is defined to be
    \begin{align*}
        ||u||_{\BMO} = \sup_{Q\subseteq \RR^n} \fint_Q |u-u_Q|,
    \end{align*}
    where the supreme goes through all cubes $Q \subseteq B$, and $u_Q$ is the average of $u$ in $Q$. We say $u$ is a bounded mean oscillation (BMO) function if $||u||_{\BMO} < +\infty$.
\end{define}

\begin{define} \label{defCZO}
    Let $T$ be a linear operator $L^2(\RR^n) \rightarrow L^2(\RR^n)$ given by an integral of singular kernel as follows:
    \begin{align*}
        Tf(x) = \int K(x, y) f(y) dy.  
    \end{align*}
    We say $T$ is Calder\'on-Zygmund operator (CZO) if $T$ is bounded in $L^2$, and the kernel $K(x, y)$ is defined away from the diagonal $\{x=y\}$ satisfying:
    \begin{itemize}
        \item Decay control at singularities: 
        \[|K(x,y)| \leq  \frac{C}{|x-y|^d}.  \]
        \item H\"older regularity: for $0<\alpha\leq 1$,
        \[ |K(x,y) - K(x', y)| \leq C \frac{|x-x'|^\alpha}{|x-y|^{n+\alpha}}, \qquad \text{if }  |x-x'| < \frac{1}{2} |x-y|, \]
        and 
        \[|K(x,y) - K(x, y')| \leq C \frac{|y-y'|^\alpha}{|x-y|^{n+\alpha}}, \qquad \text{if }  |x-x'| < \frac{1}{2} |x-y|. \]
    \end{itemize}
    Furthermore, we say a CZO $T$ is cancellative if $T(1) =0$.
\end{define}
A more general definition of Calder\'on–Zygmund operators can be formulated in the framework of the $T(1)$ theorem (see, e.g., \cite{David1984}). Since our discussion centers on the Hilbert transform, it is enough for us to adopt the simpler $L^2$–bounded setting given above.
Recall that the Hilbert transform defined on the unit circle can be viewed roughly as a convolution with a singular kernel as 
\( \displaystyle
\Hilb f (\tau) = p.v. \int  \cot \big(\frac{\tau-\theta}{2}\big)  f(\theta) d\theta
\). 
It is easy to check that the Hilbert transform is cancellative CZO with $||Tf||_{L^2} = ||f||_{L^2}$. 

It is known that CZOs are bounded maps from $L^\infty$ to BMO spaces; precisely, if $T$ is a CZO, we have
\begin{align} \label{bddCZOBMOLinfty}
||Tf||_{\BMO} \leq C||f||_{L^\infty},
\end{align}
where $C$ is a constant depending on $L^2$ norm of $T$, the decay control and H\"older regularity in Definition \ref{defCZO}. Note that $h_{\tilde{f}}^{\mathbf{k}, \alpha}(\tau; x, x')$ in \ref{defckalphainN} is the Hilbert transform of \[
h_{{f}}^{\mathbf{k},\alpha} (\tau; x,x') = \frac{\mathrm{D}^{\mathbf{k}}f(x', \tau) -  \mathrm{D}^{\mathbf{k}}f (x, \tau)} {|x'-x|^\alpha},
\]
in $\tau$. Consequently, $h_{\tilde{f}}^{\mathbf{k},\alpha}(\tau; x, x')$ is uniformly bounded in the BMO norm, with the bound independent of $x,~x'$. 

The key result that we will apply later is the John–Nirenberg theorem, which provides a precise description of how BMO bounds control the size of a function.

\begin{lem}[John--Nirenberg Theorem \cite{JNthm1961}] \label{lemJNT}
Let $f \in \BMO(\RR^n)$, and let $Q \subset \RR^n$ be any cube. Then there exist constants $c_1, c_2 > 0$, depending only on the dimension $n$, such that
\[
 \big| \{ x \in Q : |f(x) - f_Q| > \lambda \} \big|
\;\leq\; c_1 \exp\!\left(- \frac{c_2 \lambda}{\|f\|_{\BMO}}\right) |Q|,
\qquad \forall \lambda > 0,
\]
where $f_Q $ is the average of $f$ over $Q$. 
\end{lem}

The following is a quick corollary from John-Nirenberg, which gives the alternate characterisation of BMO norm:

\begin{cor}
    For any $1\leq p< \infty$, and any locally integrable function $f$ we have 
    \begin{align*}
        ||f||_{\BMO} \sim  \sup_{Q \subseteq \RR^n}\bigg( \fint_{Q} \big|f-  f_Q\big|^p\bigg)^{1/p} \sim \sup_{Q \subseteq \RR^n} \inf_{c} \bigg( \fint_{Q} \big|f-  c\big|^p\bigg)^{1/p}
    \end{align*}
\end{cor}

The next lemma is a generalization of \eqref{bddCZOBMOLinfty}, which is useful in the iterative process.

\begin{lem} \label{lembddBMOitera}
    Let $T$ be a cancellative CZO in $\RR^n$. Then, $T$ is a bounded operator in $\BMO(\RR^n)$ with
    \begin{align*}
        ||Tf||_{\BMO} \leq C ||f||_{\BMO},
    \end{align*}
    where $C$ is a uniform constant depednding on $n$, the $L^2$ operator norm of $T$, the H\"older index $\alpha$ and the constant in H\"older regularity estimate of Definition \ref{defCZO}.
\end{lem}

\begin{proof}
    Given a cube $Q \subseteq \RR^n$, we estimate the mean oscillation of $Tf$ on $Q$: 
    \begin{align*}
        \frac{1}{|Q|}\int_Q |Tf-c_Q|,
    \end{align*}
    for a fixed constant $c_Q \in \RR$. Let $mQ$ denote the dilation of $Q$ about its center, $x_Q$, by a factor of $m$. To prove the estimate, we split the function $Tf$ to be the near and the far part:
    \[
    Tf = T \big( (f-f_{2Q})\mathbf{1}_{2Q} \big) + T \big( (f-f_{2Q}) \mathbf{1}_{(2Q)^c} \big),
    \]
    where $f_{2Q}$ is the average of $f$ in $2Q$. 
    For simplicity, we write $\tilde{f}_{\text{near}} := T \big( (f-f_{2Q})\mathbf{1}_{2Q} \big)$ and $\tilde{f}_{\text{far}} := T \big( (f-f_{2Q})\mathbf{1}_{(2Q)^c} \big)$. If we take the constant $c_Q$ to be, \(
    c_Q = \tilde{f}_{\text{far}} (x_Q),
    \)
    we have
    \begin{align*}
        \frac{1}{|Q|}\int_Q |Tf-c_Q| \leq  \fint_Q |\tilde{f}_{\text{near}}| + \fint_{Q} |\tilde{f}_{\text{far}} - \tilde{f}_{\text{far}} (x_Q)|.
    \end{align*}
    For the near part, we have
    \begin{align*}
        \fint_Q |\tilde{f}_{\text{near}}| 
        &\leq \frac{1}{|Q|^{1/2}} ||\tilde{f}_{\text{near}}||_{L^2; ~ Q} \\
        &\leq \frac{1}{|Q|^{1/2}} ||T||_{L^2 \rightarrow L^2} ||f-f_{2Q}||_{L^2; ~ 2Q} \\
        &\leq ||T||_{L^2 \rightarrow L^2} ||f||_{\BMO}.
    \end{align*}
    For the far part, we have 
    \begin{align*}
        \fint_Q |\tilde{f}_{\text{far}} - \tilde{f}_{\text{far}}(x_Q)| &\leq \frac{1}{|Q|} \int_{Q} dx \int_{(2Q)^c} |K(x,y) - K(x_Q, y)|\cdot |f(y) - f_{2Q}| dy \\
        & \leq \frac{C_h}{|Q|} \int_Q |x-x_Q|^\alpha dx \int_{(2Q)^c} \frac{|f(y) - f_{2Q}|}{|x_Q-y|^{d+\alpha}} dy
        \\
        & \leq C_h |\diam Q|^\alpha \int_{(2Q)^c} \frac{|f(y) - f_{2Q}|}{|x_Q-y|^{n+\alpha}} dy
    \end{align*}
    where $C_h$ represents the constant in H\"older regularity of Definition \ref{defCZO}. The integral region $(2Q)^c$ can decompose to $(2Q)^c =\cup_{k} A_k$, where $A_k = 2^{k+1} Q \backslash 2^k Q$. Then, 
    \begin{align*}
        \int_{(2Q)^c} \frac{|\diam Q|^\alpha}{|x_Q-y|^{n+\alpha}} |f(y) - f_{2Q}| dy
        &\leq \sum_{k} \frac{1}{2^{k\alpha}} \frac{C_n}{|2^{k+1} Q|} \int_{A_k} |f(y) - f_{2Q}| dy
        \\
        &\leq  \sum_{k\geq 1} \frac{C_n}{2^{k\alpha}}  \bigg( \fint_{2^{k+1} Q} |f(y)- f_{2^{k+1} Q}| dy + |f_{2^{k+1} Q} - f_{2Q}| \bigg)
        \\
        &\leq \sum_{k \geq 1} \frac{C_n}{2^{k\alpha}} \Big(||f||_{\BMO} +  |f_{2^{k+1} Q} - f_{2Q}|\Big).
    \end{align*} 
    It suffices to estimate $|f_{2^{k+1} Q} - f_{2Q}|$. Observe that
    \begin{align*}
        |f_{2^{k+1} Q} - f_{2^k Q}| &\leq \fint_{2^k Q} |f(x)-f_{2^{k+1} Q}| dx
        \\
        &\leq 2^n \fint_{2^{k+1}Q} |f(x)- f_{2^{k+1}Q} | dx \leq C_n||f||_{\BMO},
    \end{align*}
    and 
    \begin{align} \label{estitelescope}
        |f_{2^{k+1} Q} - f_{2 Q}| \leq\sum_{j=1}^k |f_{2^{j+1 } Q} - f_{2^j Q}| \leq C_nk ||f||_{\BMO}.
    \end{align}
    Inserting the above inequality back into the far part estimate, we have
    \begin{align*}
        \fint_Q |\tilde{f}_{\text{far}} - \tilde{f}_{\text{far}}(x_Q)| \leq C_h C_n \sum_{k\geq 1} \frac{k+1}{2^{k\alpha}} ||f||_{\BMO} \leq C ||f||_{\BMO},
    \end{align*}
    where $C$ is a constant depending only on $n,~\alpha, ~ C_h$. Combining the estimates for both the near part and the far part, we complete the proof.
\end{proof}

\begin{lem}\label{lemBMOproduc}
    Let $f\in \C^{0,\gamma} (\RR^n)$ and $g \in \BMO(\RR^n)$. Suppose that $f$ and $g$ has a compact support in $K \subseteq \RR^n$ and
    \[
    \fint_K g \leq C_0.
    \]
    Then, we have
    \[
    ||f\cdot g||_{\BMO} \leq C ||f||_{0,\gamma}\cdot (||g||_{\BMO}+ C_0),
    \]
    where $C$ is a constant depending on $n$, $\gamma$ and the size of $K$. 
\end{lem}
\begin{proof}
    For each cube $Q \subseteq K$, we test the following mean oscillation 
    \[
    I := \fint_Q |fg- f_Qg_Q|.
    \]
    We split the integrand as follows: 
    \[
    |fg-f_Qg_Q| \leq |f-f_Q| |g-g_Q| + |f_Q| |g-g_Q| + |g_Q||f-f_Q|.
    \]
    Then, we have
    \[
    I \leq 3||f||_{L^\infty} ||g||_{\BMO} + |g_Q| \fint_{Q} |f-f_Q|.
    \]
    It suffices to deal with the second term on the right-hand side of the above inequality. Let $Q$ satisfies $Q \subseteq K \subseteq 2^k Q := Q'$. Now, if we fix the large cube $Q'$, by the "telescope" technique in \eqref{estitelescope}, we can control the growth of $g_Q$:
    \begin{align*}
        |g_Q| \leq C \log \frac{|Q'|}{|Q|} (||g||_{BMO}+ C_0).
    \end{align*}
    By the H\"older continuity of $f$, we have
    \[
    \fint_Q |f- f_Q| \leq C  |\diam Q|^\gamma [f]_{0,\gamma}
    \]
    Observe that $|Q|^\gamma\log |Q|^{-1} \leq C \gamma^{-1}$; hence, we complete the proof.
\end{proof}

\begin{rem}
    By John-Nirenberg Theorem, functions in BMO exhibit substantially better behavior than those in $L^p$. In particular, for any domain $B \subseteq \RR^n$ and $f \in L^p (B)$, $\{x\in B,~ |f(x)|> \lambda\} \leq c (\lambda/||f||_{L^p})^{-p}$. The rapid decay of the measure of super-level sets as the threshold increases guarantees that any loss of Hölder regularity remains controlled for the elliptic boundary problem, provided that $h_{b}^{\mathbf{k}, \alpha}(\tau; x, x')$ is bounded in BMO norms, where $b$ is the boundary data. 
\end{rem}

The following lemma serves as the cornerstone in establishing the regularity theorem for a family of linear elliptic boundary problems, which explains the relation between the boundedness of BMO norm and the loss of regularity in H\"older spaces.

\begin{lem} \label{lemBMOHolder} Let $b$ be the boundary function in the family of elliptic problems in \eqref{dbardisk} defined on $N \times \partial D$. 
 Assume the boundary data $b$ satisfies: in $\partial D$ direction,
 \begin{align*}
     \sup_{x \in N}||b(x, \cdot)||_{k,\alpha; \partial D} \leq C_0
 \end{align*}
 and in $N$ direction, $\mathrm{D}^{\mathbf{k}} b$ is bounded for each $k$-th order differential operator in $N$, and 
 \begin{align*}
     ||h^{\mathbf{k}, \alpha}_b (\tau; x,x')||_{\BMO;~ \partial D} \leq C_0, \qquad \fint_{\TT1} h^{\mathbf{k}, \alpha}_b (\theta; x,x') d\theta \leq C_0
 \end{align*}
 where the function $h^{\mathbf{k}, \alpha}_b (\tau; x,x')$ is defined as in \eqref{defckalphainN}. Then the Hilbert transform of $b$ in $\tau$, denoted by $\tilde{b}$, satisfies
 \begin{align} \label{estiinlemhBMO}
     ||h_{\tilde{b}}^{\mathbf{k}, \alpha} (\tau; x, x')||_{\BMO; ~ \TT^1} \leq C_1 C_0 
 \end{align}
 where $C_1$ is a uniform constant depending on the same data as in Lemma \eqref{lembddBMOitera}, and
 \begin{align} \label{estiinlemLinftybound}
     |h_{\tilde{b}}^{\mathbf{k}, \alpha} (\tau; x, x')| \leq C\big\{ \big(\log |x-x'|\big)^2 +1\big\},
 \end{align}
 where $C$ is a uniform constant depending on $C_{0,1}$, $c_{1,2}$, $\alpha$. Furthermore, we have 
 \begin{align}\label{estinlemlossregularityholder}
     ||\tilde{b}||_{k,\beta; N \times \partial D} \leq C(\alpha-\beta)^2.
 \end{align}
\end{lem}

\begin{proof}
    If we denote $\delta = |x-x'|$, by \eqref{estipara}, we have
    \begin{align*}
        |h_{\tilde{b}}^{\mathbf{k}, \alpha} (\tau; x, x')|  \leq & \ C \delta^{-\alpha} \bigg\{ \int_{\tau-\delta}^{\tau + \delta} \frac{|\mathrm{D}^{\mathbf{k}}b(x, \theta) - \mathrm{D}^{\mathbf{k}} b(x, \tau) |}{|\theta -\tau| }  d\theta  + \int_{\tau-\delta}^{\tau + \delta} \frac{|\mathrm{D}^{\mathbf{k}}b(x', \theta) - \mathrm{D}^{\mathbf{k}} b(x', \tau) |}{|\theta -\tau| }  d\theta \bigg\}
         \\
         & \qquad +  \int_{\delta<|\theta-\tau|< \pi} |h_b^{\mathbf{k}, \alpha}(\theta; x, x')| \cot \big(\frac{\theta-\tau}{2} \big)d\theta.
    \end{align*}
    It is obvious that the first two integrals on the right-hand side of inequality are bounded by the constant: $\alpha^{-1} C  \sup_{x \in N}||b(x, \cdot)||_{k,\alpha; \partial D}$. It suffices to deal with the third term. By Lemma \ref{lemJNT}, for a large positive constant $\lambda$, 
    \begin{align*}
        |\{\theta\in \TT^1, ~ |h_b^{\mathbf{k}, \alpha}(\theta; x, x') - h_b^{\mathbf{k}, \alpha}(\cdot; x, x')\big|_{\TT^1}| > \lambda \}| \leq c_1 \exp \big(-c_2 \lambda / C_0\big),
    \end{align*}
    where $h_b^{\mathbf{k}, \alpha}(\cdot; x, x')\big|_{\TT^1}$ is the average of $h_b^{\mathbf{k}, \alpha}(\theta; x, x')$ over $\TT^1$. The third integral can be controlled via the layer-cake representation, yielding the following estimate:
    \begin{align*}
        (\text{The third integral})
        &\leq C \int_\delta^\pi \frac{| h_b^{\mathbf{k}, \alpha}(\tau +t; x, x') |}{t} dt\\
        & \leq CC_0+ C \sum_{n = 1}^\infty  \log \Big(\frac{\delta + c_1 e^{-c_2 n /C_0}}{\delta}\Big) \\
        & \leq CC_0+
        \sum_{1\leq n \leq \frac{C_0}{c_2}|\log 
    \delta|} 
     |\log (c_1 \delta^{-1})|
     +  \sum_{m=1}^\infty (m+ \frac{C_0}{c_2}|\log 
    \delta| ) \log (1+ c_1 e^{-c_2m / C_0})\\
    & \leq \frac{C_0}{c_2} |\log \delta|^2 + \frac{C_0}{c_2} |\log c_1| |\log \delta | + CC_0+ c_1\sum_{m=0}^\infty  m  e^{-c_2m/C_0}
    \\
    &\leq  C(|\log \delta|^2 +1).
    \end{align*}
    Hence, we complete the proof of \eqref{estiinlemLinftybound}. The proof of \eqref{estiinlemhBMO} follows directly from Lemma \ref{lembddBMOitera}. By the same observation as in Remark \ref{remregg1g2},
    \begin{align*}
        \sup_{|x - x'| \leq \delta,\ \tau \in D} \left| \mathrm{D}^{\mathbf{k}} \tilde{b}(x', \tau) - \mathrm{D}^{\mathbf{k}} \tilde{b}(x, \tau) \right| \leq C \delta^\alpha (|\log \delta|^2 + 1).
    \end{align*}
    If we consider the $\C^{k,\beta}$ norm of $\tilde{b}$, the blow-up rate of the coefficient as $\beta \rightarrow \alpha$ is given by
    \begin{align*}
        \sup_{0< \delta\leq 1} \delta^{\alpha-\beta} |\log \delta|^2 \sim (\alpha-\beta)^{-2},
    \end{align*}
    which completes the proof of \eqref{estinlemlossregularityholder}.
\end{proof}

\subsubsection{Linearized problems near the trivial foliation} In this subsection, we solve the linearized problem by introducing a small perturbation of the trivial foliation $(G_0, \partial \widetilde{\Psi}_0)$. Consider $(G, \partial \widetilde{\Psi})$ satisfying
\begin{align}\label{consmallperturb}
    ||\partial \widetilde{\Psi} - \partial \widetilde{\Psi}_0 ||_{k+1, \alpha; N' \times D } \leq \delta/2  \quad \text{and} \quad ||(z- z_0, \xi-\xi_0)||_{k,\alpha; N \times D } \leq \delta/(2C_0),
\end{align}
where $C_0$ is taken to be $||\nabla_0^3 \widetilde{\Psi}_0||_{k,\alpha; N' \times D}$ and $\delta$ is a sufficiently small uniform constant that will be determined later in this subsection.
Defining $\widetilde{A}_{i \bar{j}}(w, \tau) = (\partial_{i} \overline{\partial}_{j} \widetilde{\Psi})(z(w, \tau), \tau)$ and $\widetilde{B}_{i j} (w, \tau) = (\partial_i \partial_j \widetilde{\Psi}) (z(w, \tau), \tau) $, the linearized problem at $(G, \partial \widetilde{\Psi})$ is given by  
\begin{align}\label{holodisclinearpde'}
    \begin{split}
        &\overline{\partial}_\tau z_i (w, \tau) = \overline{\partial}_\tau \xi_i (w, \tau) = 0, \hspace{5.8cm} \text{ in }N \times D,\\
         &\xi_i(w,\tau) - \widetilde{A}_{i\bar{j}}(w, \tau) \overline{z}_j(w, \tau)- \widetilde{B}_{ij}(w, \tau) z_j(w, \tau)=\bff_i(w,\tau),
         \hspace{0.3cm}  \text{ in } N \times \partial D,\\
         &z_i(w, -i) = 0,  \hspace{8.2cm}  \ w  \in N,
    \end{split}
\end{align}
As in the previous subsection, we define $A_{i\bar{j}}(w) = \partial_i \overline{\partial}_j \widetilde{\Psi}_0$ and $B_{ij} (w) = \partial_i \partial_j \widetilde{\Psi}_0$. The small perturbation condition of $(G, \partial \widetilde{\Psi})$ relative to $(G_0, \partial \widetilde{\Psi}_0)$ implies the following estimates:
\begin{align} \label{constantsmallperturb}
\begin{split}
    ||\widetilde{A}_{i\bar{j}}-A_{i \bar{j}}||_{{k, \alpha}, N \times \partial D} \leq \delta,\\
    ||\widetilde{B}_{i\bar{j}}-B_{i \bar{j}}||_{{k, \alpha}, N \times \partial D} \leq \delta.
\end{split}
\end{align}

Lemma \ref{IFTlinearpart} indicates that there exists an inverse of $D_\A T$ at $(G_0, \partial \widetilde{\Psi}_0)$ with a loss of regularity. For simplicity, we denote $T'_0 = D_{\A} T|_{(G_0, \partial \widetilde{\Psi}_0)} $. Under a small perturbation of $T'_0$, given by $T' = T'_0 - P_\delta $, the formal inverse of the perturbed operator, $ T'^{-1} = \sum_n \{P_\delta (T'_0)^{-1}\}^n (T'_0)^{-1}$, may not be well-defined due to the loss of regularity for $(T'_0)^{-1}$. In the following lemma, we resolve the issue and prove that the inverse of the perturbed linear operator exhibits the same regularity loss as $dT_0$. 

\begin{lem} \label{lemlinearproblemperturb}
    Let $A= (A_{i\bar{j}})  \in \C^{\infty}(N, \CC^{n \times n})$ be a nowhere degenerate hermitian matrix satisfying
\begin{align*}
    \det A \geq \sigma>0
\end{align*}
and $ B= (B_{ij}) \in \C^{\infty} (N, \CC^{n\times n} )$, a symmetric matrix. Let $\widetilde{A}(w, \tau)$, $\widetilde{B}(w, \tau) \in \C^{\infty} (N \times \partial D, \CC^{n \times n})$ be a small perturbation of $A$, $B$ respectively satisfying \eqref{constantsmallperturb} with a uniform small constant $\delta$.
For each $\bff = (\bff_1,\ldots,\bff_n) \in \C^{\infty}(N \times \partial D, \CC^n)$, there exists a unique solution, $(z_1,\ldots, z_n; \xi_1, \ldots, \xi_n)$, with each $z_i$, $\xi_i\in \C^{\infty} (N \times D, \CC)$, to the equation \eqref{holodisclinearpde'}.

Moreover, for each $k \geq 1, \ k\in \ZZ$ and $0<\alpha <1$, there is a uniform constant $C$ only depending on $n$, $k$, $\alpha$, $\sigma^{-1}$, $||A||_{k,\alpha; N}$, $||B||_{k, \alpha; N}$ such that
\begin{align} \label{estiperturbholodiscs'}
   ||(z, \xi)||_{k,\beta; N \times D} \leq C (\alpha-\beta)^{-2} ||\bff||_{k, \alpha; N \times \partial D}.
\end{align}
\end{lem}
\begin{proof} Let $T'$ be the perturbed linearized operator, defined by
\[
T'(\hat{z}, \hat{\xi}) (w, \tau) = \hat{\xi}_i (w, \tau) - \widetilde{A}_{i\bar{j}} (w, \tau) \bar{\hat{z}}_j (w, \tau) - \widetilde{B}_{ij} (w, \tau) \hat{z}_{j}(w, \tau), \qquad (w, \tau) \in N \times \partial D,
\]
for $\widehat{G} = (\hat{z}, \hat{\xi}) \in \A'$. 
We define the perturbation operator $P_\delta = T'- T'_0$ by 
\begin{align*}
\begin{split}
P_{\delta} (\hat{z}, \hat{\xi}) (w, \tau) = \big(&\widetilde{A}_{i\bar{j}} (w, \tau) - A_{i\bar{j}}(w)\big) \bar{\hat{z}}_{j} (w, \tau)\\
&\hspace{0.5cm}+ \big( \widetilde{B}_{ij} (w, \tau)- B_{ij} (w)\big) \hat{z}_j(w, \tau),
\end{split}
\qquad (w, \tau) \in N \times \partial D,
\end{align*}
where both $(A,\widetilde{A})$ and $(B, \widetilde{B})$ satisfy the smallness condition \eqref{constantsmallperturb}. We now introduce the notation for the iterative process. The iteration starts by setting $\bff_0 = \bff$ and $\widehat{G}_0 = (\hat{z}_0, \hat{\xi}_0) = (T'_0)^{-1} (\bff_0)$. Inductively, given \( \widehat{G}_{n-1} = (\hat{z}_{n-1}, \hat{\xi}_{n-1}) \), we define the next input data as
\[
\bff_n := -P_\delta(\hat{z}_{n-1}, \hat{\xi}_{n-1}),
\]
and solve
\[
\widehat{G}_n = (\hat{z}_n, \hat{\xi}_n) := T_0^{-1}(\bff_n).
\]
Let  
\(
G_0 = \widehat{G}_0.
\)
The \( n \)-th approximation is then given by
\(
G_n := G_{n-1} + \widehat{G}_n.
\) 
It suffices to deal with the convergence of $(G_n)$ and the associated estimates.

If we define $(\hat{g}_{1,n}, \hat{g}_{2,n})$ by transforming the data $(\hat{z}_n, \hat{\xi}_n)$ as in \eqref{defg1g2}: \(\hat{g}_{1,n} = \hat{\xi}_{n}- B\cdot\hat{z}_n- \overline{A} \cdot \hat{z}_n \) and \(\hat{g}_{2,n} = \hat{\xi}_{n}- B\cdot\hat{z}_n+ \overline{A} \cdot \hat{z}_n \), then, $(\hat{g}_{1,n}, \hat{g}_{2,n})$ solves the standard linear free boundary problem \eqref{equstandardfbpros} with the boundary data $\bff_n$. 
The solution $(\hat{g}_{1,n}, \hat{g}_{2,n})$ is obtained by applying the Hilbert transform to $\re \bff_n$ and $\im \bff_n$, respectively, and extending these boundary values harmonically into the disk $D$. 
The estimates for $(\hat{g}_{1,n}, \hat{g}_{2,n})$ can be obtained by Remark \ref{remregg1g2} and Lemma \ref{lemBMOHolder}. The inverse transformation \eqref{inverserepzxibyg} then provides the estimates for $(\hat{z}_n, \hat{\xi}_n)$. When $n=0$: in the $D$ direction,
\begin{align}\label{induczxi0D}
    ||(\hat{z}_0 (x, \cdot), \hat{\xi}_0 (x, \cdot))||_{k,\alpha; D} \leq C ||\bff(x, \cdot)||_{k,\alpha; \partial D},
\end{align}
and in the $N$ direction, we estimate the modulus of continuity of $\mathrm{D}^{\mathbf{k}}(\hat{z}_0, \hat{\xi}_0) $, where $\mathrm{D}^{\mathbf{k}}$ denotes a differential operator with multi-indices $\mathbf{k}$ of order $|\mathbf{k}| = k$ with respect to the local coordinates in $N$. For clarity, recall the notation in \eqref{defckalphainN}, 
\[
(h^{\mathbf{k},\alpha}_{z_0}, h^{\mathbf{k},\alpha}_{\xi_0}) (\tau; x,x') = \frac{\mathrm{D}^{\mathbf{k}}(z_0, \xi_0) (x,\tau)- \mathrm{D}^{\mathbf{k}}(z_0, \xi_0) (x',\tau)}{|x-x'|^\alpha}, \qquad \text{ for } \tau \in D, ~ x,x' \in N.
\]
Then, 
\begin{align}\label{induczxi0N}
    \sup_{|x-x'| < t}|(h^{\mathbf{k},\alpha}_{z_0}, h^{\mathbf{k},\alpha}_{\xi_0}) (\tau; x,x')| \leq C \big|\log t\big|  \cdot ||\bff||_{k,\alpha; N \times \partial D}.
\end{align}
If we restrict $(h^{\mathbf{k},\alpha}_{z_0}, h^{\mathbf{k},\alpha}_{\xi_0})$ on the boundary $N \times \partial D$, The BMO norm of $(h^{\mathbf{k},\alpha}_{z_0}, h^{\mathbf{k},\alpha}_{\xi_0})$ satisfies
\begin{align}\label{induczxi0NBMO}
    \sup_{x,x' \in N}|| (h^{\mathbf{k},\alpha}_{z_0}, h^{\mathbf{k},\alpha}_{\xi_0}) (\cdot;x, x')||_{\BMO;~ \partial D} \leq C ||\bff||_{k,\alpha; N \times \partial D}.
\end{align}
The constant $C$ in \ref{induczxi0D}, \eqref{induczxi0N} and \eqref{induczxi0NBMO} is uniform depending on the data $n$, $k$, $\alpha$, $\sigma^{-1}$, $||A||_{k,\alpha; N } $ and $||B||_{k,\alpha; N}$. 

When $n \geq 1$, assuming that $C$ is the uniform constant given as above, we define
\begin{align} \label{parachoosedelta}
\delta = (10C)^{-1}
\end{align}
it suffices to prove the following estimates inductively:
\begin{align}
 \sup_{x\in N}||\bff_n(x, \cdot)||_{k,\alpha; \partial D} &\leq 2^{1-n} C \delta ||\bff(x, \cdot)||_{k,\alpha; \partial D} \tag{n1} \label{inducbnD}
\\
        \sup_{x\in N}\|(\hat{z}_n (x, \cdot), \hat{\xi}_n(x, \cdot)) \|_{k,\alpha; D} 
        &\leq 2^{-n}  C \sup_{x\in N} \| \bff(x, \cdot) \|_{k,\alpha; \partial D}, \tag{n2} \label{induczxinD} 
\end{align}
and
\begin{align}
        \sup_{x,x'} ||h^{\mathbf{k},\alpha}_{\bff_n}(\cdot; x,x')||_{\BMO; \partial D} &\leq 2^{1-n} C \delta ||\bff||_{k,\alpha; N \times \partial D}, \tag{n3} \label{inducbnN}
        \\
        \sup_{x\in N}\fint_{\partial D} |h^{\mathbf{k},\alpha}_{\bff_n}(\cdot; x,x')| d\tau &\leq 2^{1-n}  C \delta ||\bff||_{k,\alpha; N \times \partial D}
        \tag{n4} \label{inducbnintegral}
        \\
        \sup_{x,x' \in N}|| (h^{\mathbf{k},\alpha}_{\hat{z}_n}, h^{\mathbf{k},\alpha}_{\hat{\xi}_n}) (\cdot;x, x')||_{\BMO;~ \partial D} 
        &\leq 2^{-n}  C \| \bff \|_{k,\alpha; N \times \partial D}, \tag{n5} \label{induczxinNBMO}
    \end{align}
By Lemma \ref{lemBMOHolder}, we automatically derived that
\begin{align*}
    \sup_{|x-x'|\leq t}|(h^{\mathbf{k},\alpha}_{\hat{z}_n}, h^{\mathbf{k},\alpha}_{\hat{\xi}_n}) (\cdot;x, x')| \leq 2^{-n} C |\log t|^2\cdot ||\bff||_{k,\alpha; N \times \partial D}, \tag{n5'} \label{induczxiNLinfty}
\end{align*}
and, by lowering the regularity index, the following Hölder estimate holds:
\begin{align*}
    ||(\hat{z}_n, \hat\xi_n)||_{k,\beta; N \times \partial D} \leq 2^{-n} C (\alpha-\beta)^{-2} ||\bff||_{k,\alpha; N \times \partial D}. \tag{n5''} \label{induczxilossHolder}
\end{align*}
where $C$ is always chosen to be a uniform constant depending  $n$, $k$, $\alpha$, $\sigma^{-1}$, $||A||_{k,\alpha; N}$ and $||B||_{k, \alpha; N}$.

By the definition of $\bff_{n+1}$, $\bff_{n+1} = - P_{\delta} (\hat{z}_n, \hat{\xi}_n) = (\widetilde{A}(x, \tau)-A(x)) \bar{\hat{z}}_n(x,\tau) + (\widetilde{B}(x, \tau)- B(x)) \hat{z}_n (x, \tau)$, together with the smallness condition \eqref{constantsmallperturb}, we have
\begin{align*}
    \sup_{x\in N}||\bff_{n+1}(x, \cdot)||_{k,\alpha; \partial D} &\leq  \delta \sup_{x\in N} ||(\hat{z}_n, \hat{\xi}_n) (x, \cdot)||_{k,\alpha; D} 
    \\
    &\leq 2^{-n} C\delta \sup_{x\in N} ||\bff(x, \cdot)||_{k,\alpha; \partial D}.
\end{align*}
The pair $(\hat{z}_n, \hat{\xi}_n)$ solves the linear equation \eqref{holodisclinearpde} by setting the boundary data $\bff_{n+1}$. By Remark \ref{remregg1g2}, we have
\begin{align*}
    \sup_{x\in N}||(\hat{z}_{n+1}, \hat{\xi}_{n+1})(x, \cdot)||_{k,\alpha; D} 
    &\leq C \sup_{x\in N}||\bff_{n+1}(x, \cdot)||_{k,\alpha; \partial D}\\
    &\leq 2^{-n-1}  C \sup_{x\in N} \| \bff(x, \cdot) \|_{k,\alpha; \partial D}.
\end{align*}
Hence, we complete the proof of $((n+1)1)$ and $((n+1)2)$. The estimate, $((n+1)4)$, follows directly from the definition of $\bff_{n+1}$ and \eqref{induczxinD}. To prove $((n+1)3)$, note that
\begin{align*}
    h^{\mathbf{k}, \alpha}_{{\bff}_{n+1}}(\tau; x, x') &= \frac{\mathrm{D}^{\mathbf{k}}\bff_{n+1} (x,\tau)- \mathrm{D}^{\mathbf{k}}\bff_{n+1} (x',\tau)}{|x-x'|^\alpha}
    \\ 
    & = \delta K(x, x', \tau)+ (\widetilde{A}-A) (x', \tau) \cdot h^{\mathbf{k}, \alpha}_{\bar{\hat{z}}_{n}}(\tau; x, x') + (\widetilde{B}-B) (x', \tau) \cdot h^{\mathbf{k}, \alpha}_{{\hat{z}}_{n}}(\tau; x, x'),
\end{align*}
where $K(x,x',\tau)$ is an $L^\infty$ function defined on $N \times N \times \partial D$, with the bound derived from \eqref{induczxilossHolder} and Lemma \ref{lemBMOHolder}:
\begin{align} \label{n+1-3Linftybd}
||K(x,x',\tau)||_{L^\infty} \leq \alpha^{-2} 2^{-n} C ||\bff||_{k,\alpha; N \times \partial D},
\end{align}
and the remaining terms, $(\widetilde{A}-A)(x', \tau)\cdot h_{\bar{\hat{z}}_n}^{\mathbf{k}, \alpha} (\tau; x,x')$ and $(\widetilde{B}-B)(x', \tau)\cdot h_{{\hat{z}}_n}^{\mathbf{k}, \alpha} (\tau; x,x')$, are vector valued BMO functions with each element is given by a combination of $(\widetilde{A}_{i\bar{j}}-A_{i\bar{j}}) (x', \tau) h_{\bar{\hat{z}}_j}^{\mathbf{k}, \alpha}(\tau; x,x') $. By Lemma \ref{lemBMOproduc}, together with \eqref{induczxinNBMO} and $((n+1)4)$, we have
\begin{align*}
    ||(\widetilde{A}-A)(x', \cdot)\cdot h_{\bar{\hat{z}}_n}^{\mathbf{k}, \alpha} (\cdot; x,x')||_{\BMO;~\partial D}, \quad ||(\widetilde{B}&-B)(x', \cdot)\cdot h_{{\hat{z}}_n}^{\mathbf{k}, \alpha} (\cdot; x,x')||_{\BMO;~ \partial D}\\ 
    &\leq C\delta \big(||h_{{\hat{z}}_n}^{\mathbf{k}, \alpha}||_{\BMO;~ \partial D}+ 2^{-n} C\delta ||\bff||_{k,\alpha; N\times \partial D}\big)\\
    & \leq 2^{-n} C \delta||\bff||_{k,\alpha; N\times \partial D}.
\end{align*}
Combining with the $L^\infty$ estimate, we have 
\[
\sup_{x,x'\in N}||h^{\mathbf{k}, \alpha}_{{\bff}_{n+1}}(\cdot; x, x')||_{\BMO;~ \partial D } \leq 2^{-n} C\delta ||\bff||_{k,\alpha; N \times \partial D}. 
\]
Hence, we complete the proof of $((n+1)3)$. It suffices to prove $((n+1)5)$. By taking Hilbert transform of $ \re \bff_{n+1}$ and $\im \bff_{n+1}$ on $\partial D$ and Lemma \ref{lembddBMOitera}, we obtain  
\begin{align*}
\sup_{x,x' \in N}|| (h^{\mathbf{k},\alpha}_{g_{1,n+1}}, h^{\mathbf{k},\alpha}_{g_{2,n+1}}) (\cdot;x, x')||_{\BMO;~ \partial D} 
        &\leq   C \sup_{x,x'\in N} \| h^{\mathbf{k},\alpha}_{\bff_{n+1}}(\cdot; x,x')\|_{BMO;~ \partial D}\\
        &\leq 2^{-n} C \delta ||\bff||_{k,\alpha; N \times \partial D}. 
\end{align*}
Applying the relation between $(\hat{z}_{n+1}, \hat{\xi}_{n+1})$ and $(g_{1,n+1}, g_{2, n+1})$, we complete the proof of $((n+1)5)$. The estimate \eqref{induczxiNLinfty} follows directly from \eqref{induczxinNBMO} and Lemma \ref{lemBMOHolder}, and the estimate \eqref{induczxilossHolder} follows from \eqref{induczxinD}, \eqref{induczxinNBMO} and Lemma \ref{lemBMOHolder}. 

It remains to prove that the approximate solutions $G_n$, denoted by $G_n = (z_n,\xi_n)$, converge to the solution to \eqref{holodisclinearpde'}. Notice that
\begin{align*}
\bff- T'(z_n,\xi_n)  = \bff- \sum_{l=0}^n T'(\hat{z}_l, \hat{\xi}_l) = \bff_{n+1}.
\end{align*}
By above discussion, $\bff_{n+1}$ satisfies the estimate $((n+1)1)$ and $((n+1)3)$, and by Lemma \eqref{lemBMOHolder}, we have 
\begin{align*}
    ||\bff- T'(z_n,\xi_n)||_{k,\beta; N \times \partial D} = ||\bff_{n+1}||_{k,\beta; N \times \partial D} \leq (\alpha- \beta)^{-2} 2^{-n} C ||\bff||_{k,\alpha; N \times \partial D }.
\end{align*}
The solution to \eqref{holodisclinearpde'} is given by $ \displaystyle G = \lim_{n\rightarrow \infty} G_n$, satisfying the estimate:
\begin{align*}
    ||(z, \xi)||_{k,\beta; N \times D} \leq \sum_{n=0}^\infty ||(\hat{z}_n, \hat{\xi}_n)||_{k,\beta; N \times  D} \leq 2(\alpha-\beta)^{-2}C||\bff||_{k,\alpha; N \times \partial D}.
\end{align*}
Therefore, $(G_n) $ converges to the solution to \eqref{holodisclinearpde'}, $G = (z, \xi)$, in $\C^{k,\beta}$ norm.

To complete the proof, we now examine the local uniqueness of the solution. Suppose that $(z, \xi)$ is a nonzero solution to \eqref{holodisclinearpde'} such that $\bff = 0$. The fixed point condition $z(\cdot, -i) = 0 = \xi(\cdot, -i)$ implies that the nonzero solution, $(z,\xi)$, cannot be constant. By scaling $(z,\xi)$, we assume $||\xi||_{\BMO; ~ \partial D} + ||z||_{\BMO; ~\partial D} = 1$. Since $T'(z,\xi) = 0$, we have
\begin{align*}
    \bff=T'_0 (z, \xi) = (T'_0 - T') (z, \xi) = (\widetilde{A}-A)\cdot \bar{z} + (\widetilde{B}- B)\cdot z
\end{align*}
satisfying
\begin{align*}
    ||\bff||_{\BMO; ~\partial D} \leq  C \delta ||z||_{\BMO; ~ \partial D} \leq  C\delta. 
\end{align*}
Notice that $(z, \xi)$ is the unique solution to $T'_0 (z,\xi) = \bff$. Let $(g_1, g_2) = (\xi-\overline{A}\cdot z - B \cdot z,~ \xi+\overline{A}\cdot z - B \cdot z )$.  Then, $(g_1, g_2)$ satisfies the standard elliptic boundary problem; hence, we have
\begin{align*}
    g_1\big|_{\partial D} = \re \bff + i \Hilb ( \re \bff) + C_1, \qquad g_2\big|_{\partial D} = -\Hilb (\im \bff) + i  \im \bff+ C_2.
\end{align*}
By Lemma \ref{lembddBMOitera}, $||\Hilb f||_{\BMO; ~ \partial D} = C ||f||_{\BMO; ~ \partial D}$; hence, we have
\begin{align*}
||g_1||_{\BMO; ~ \partial D} + ||g_2||_{\BMO; ~ \partial D} \leq C \delta.
\end{align*}
By inverse transformation, $2z =  \overline{A}^{-1} (g_2-g_1) $, we obtain
\begin{align*}
    ||z||_{\BMO; ~ \partial D} + ||\xi||_{\BMO; ~ \partial D} &\leq C (||g_1||_{\BMO; ~ \partial D} + ||g_2||_{\BMO; ~ \partial D})
    \\
    & \leq C \delta
\end{align*}
where $C$ is a uniform constant depending on $n$, $\sigma^{-1}$, $||A||_{L^\infty; N \times \partial D}$ and $||B||_{L^\infty; N \times \partial D}$.
Recall that we pick $\delta$ with $\delta = 10^{-1} C^{-1}$; we then have
\[
||z||_{\BMO; ~ \partial D} + ||\xi||_{\BMO; ~ \partial D} \leq 10^{-1}, 
\]
which contradicts against the normlized assumption $||z||_{\BMO; ~ \partial D} + ||\xi||_{\BMO; ~ \partial D} = 1$. \end{proof}

\subsection{Existence of the Families of Holomorphic Discs}

In this subsection, we solve (\ref{holodiscpde}) by introducing a small perturbation to the trivial boundary data. Specifically, we begin with the free boundary condition given in (\ref{bdryconditrivialfoli}) and consider the family of holomorphic discs described in (\ref{cxdimntrivialfoli}). After perturbation, we assume that the free boundary condition $\Lambda_{\widetilde{\Psi}}$ satisfies
\begin{align*}
    ||\partial \widetilde{\Psi}- \partial \widetilde{\Psi}_0||_{k+1,\alpha; N' \times \partial D} \leq \varepsilon, \qquad k\geq 1, \quad \alpha\in (0,1),
\end{align*}
where $\widetilde{\Psi}_0$ corresponds to the unperturbed case with $\widetilde{\Psi}_0 = \rho$, and $\varepsilon$ is a sufficiently small uniform constant that will be determined later in this subsection.

The method we used here is Zehnder's version Nash-Moser implicit function based on the classic works \cite{Nash19956Imbedding, Moser1961Tech, Zehnder1975GeneralizedIFT}. Due to the loss of regularity in the linearized problem near the trivial foliation (see Lemma \ref{IFTlinearpart} and Lemma \ref{lemlinearproblemperturb}), the standard implicit function theorem is not directly applicable. Instead, Zehnder’s framework provides the best fit for our setting.

Now, we start proving the following main theorem of the section:

\begin{thm} \label{thmexistbypertb} Let $\widetilde{\Psi}_0(\cdot, \tau)= \rho$ such that $\rho$ is the potential function of reference K\"ahler metric $\omega$ in $N'$, and let $G_0(w, \tau)$ be the family of holomorphic discs defined in (\ref{cxdimntrivialfoli}). Then, there exists a uniform small constant $\varepsilon >0$ only depending on $n$, $k$, $\alpha$, $\sigma^{-1},||\nabla_X^2\widetilde{\Psi}_0||_{k+1,\alpha; N' \times \partial D}$, such that for each $\widetilde{\Psi} \in \C^{\infty} (N' \times \partial D, \RR )$ satisfying
\begin{align} \label{smallconfreebdry}
    || \partial \widetilde{\Psi} - \partial \widetilde{\Psi}_0||_{k+2,\alpha; N' \times \partial D} \leq \varepsilon,
\end{align}
there is a smooth family of holomorphic discs $G: N \times D \rightarrow E$ satisfying the following conditions:
\begin{enumerate}
     \item[(i)] $G$ is smooth in $N \times D$;
     \item[(ii)] $g_w (\tau) = G(w, \tau)$ is holomorphic with respect to $\tau$.
     \item[(iii)] for each $\tau \in \partial D $ and each $x\in N$, $g_w(\tau) \in \Lambda_{\tilde{\psi}_\tau}$;
     \item[(iv)] for each $w\in N$, we have $\displaystyle \pi\circ g_w (-i) = x$;
     \item[(v)] the projection down of the family of holomorphic discs gives a foliation in $N' \times D$. Let $H (x,\tau) = \pi \circ G(x, \tau) $ and $h_\tau (x) = H(x, \tau)$. For each $ \tau \in  D$, the map $h_\tau: N \rightarrow N'$ is a diffeomorphism with the image. 
     In addition, for each $\tau \in D$, $h_\tau (x)$ satisfies $ N''  \subseteq h_\tau ( N ) \subseteq N'$.
\end{enumerate}
If we write $G(w,\tau)$ in the standard coordinates of $E$, $G(w,\tau) = (z_1,\ldots, z_n; \xi_1, \ldots, \xi_n)(w,\tau) $, we have
\begin{align} \label{estimeasureperturb}
    ||(z-z_0,\xi-\xi_0)||_{k,\beta; N \times D} \leq C (\alpha-\beta)^{-2} ||\partial \widetilde{\Psi}-\partial \widetilde{\Psi}_0||_{k, \alpha; N'\times \partial D}, 
\end{align}
where $z_0, \ \xi_0$ are given in (\ref{cxdimntrivialfoli}) and $C$ is a constant depending on $n$, $k$, $\alpha$, $\sigma^{-1}$, $\varepsilon$, and $||\nabla^2_X \widetilde{\Psi}_0||_{k+1, \alpha, N' \times \partial D}$.
\end{thm}
In the previous sections, recall that we introduced a nonlinear operator: $T:  \A \times \B \rightarrow \D$. 
To apply Zehnder's version of the implicit function theorem, we consider the following families of spaces with H\"older regularity $k \geq 1$ and $0< \alpha< 1$: 
\begin{align*}
    &\A^{k,\alpha} =\{ G \in \C^{k,\alpha}(N \times D, E); \ \overline{\partial}_\tau G (w, \tau) =0,\  H(w, -i)= \pi\circ G (w, -i) =w \} \\
    \B^{k,\alpha} &=   \{ \Lambda_{\widetilde{\Psi}} \subseteq  E \times \partial D; \ \tilde{\psi}_\tau (z)= \widetilde{\Psi}(z, \tau)  \in \C^{k+1,\alpha} (N' \times \partial D), \ i\partial \overline{\partial} \tilde{\psi}_\tau = i\partial \overline{\partial} \widetilde{\Psi}(\cdot, \tau)>0 \},
\end{align*}
and $\D^{k,\alpha} = \C^{k,\alpha} (N \times \partial D, \CC^n)$. The operator $T$ can be automatically extended to $T : \A^{k,\alpha} \times \B^{k,\alpha} \rightarrow \D^{k,\alpha}$. By Lemma \ref{IFTlinearpart} and Lemma \ref{lemlinearproblemperturb}, for each pair of data $(G, \partial \widetilde{\Psi})$ satisfying the small perturbation conditions \eqref{consmallperturb}, the linearized operator $D_\A T$ at $(G, \partial \widetilde{\Psi})$ admits an inverse $\eta|_{ (G, \partial \widetilde{\Psi})}$ such that
\begin{align}\label{inverseesti}
    ||\eta\big|_{(G,\partial \widetilde{\Psi})} (\widehat{G})||_{k,\beta; N \times \partial D} \leq C (\alpha -\beta)^{-2} ||\widehat{G}||_{k,\alpha; N \times D}, \qquad \text{for } \beta< \alpha. 
\end{align}
To proceed with Zehnder's version Nash-Moser iteration, we also need to check the following quadratic estimates for $Q (G_2, G_1; \partial \widetilde{\Psi}) = T(G_2, \partial\widetilde{\Psi}) - T(G_1, \partial \widetilde{\Psi}) - D_\A T|_{(G_1,\partial \widetilde{\Psi})} (\widehat{G}_{1,2})$, where $\widehat{G}_{1,2}$ can be viewed as a function in the tangent space of $\A^{k,\alpha}$.
\begin{lem} Let $G_1 = (z_1, \xi_1)$, $G_2 = (z_2, \xi_2)$ and $\widehat{G}_{1,2}= (\hat{z}, \hat{\xi})$. Suppose that $\widehat{G}_{1,2} = G_2 - G_1 = (\hat{z}, \hat{\xi})$ satisfying $||(\hat{z}, \hat{\xi})||_{k,\alpha; N \times D} \leq 1$. Then, there is a uniform constant $C$ such that
    \begin{align} \label{estiTalorformula}
        ||Q(G_1, G_2; \partial \widetilde{\Psi})||_{k,\alpha; N \times \partial D}
        \leq C ||(\hat{z}, \hat{\xi})||_{k,\alpha; N \times D}^{2}
    \end{align}
\end{lem}
\begin{proof}
    By plugging in the formulas of $T(G,\partial \widetilde{\Psi})$ and $D_\A T|_{(G,\partial \widetilde{\Psi})} (\widehat{G})$, we have
    \begin{align*}
    \begin{split}
        Q(G_1, G_2; \partial \widetilde{\Psi})_i &= - \partial_i \widetilde{\Psi} (z_2, \tau) + \partial_i \widetilde{\Psi}(z_1, \tau)  + \partial_i \overline{\partial}_{j} \widetilde{\Psi}(z_1, \tau) \bar{\hat{z}}_{j} + \partial_i \partial_j \widetilde{\Psi} (z_1, \tau) \hat{z}_j\\
        & = - \int_0^1\partial_i \overline{\partial}_j \widetilde{\Psi} (z_t, \tau) \bar{\hat{z}}_j dt +  \partial_i \overline{\partial}_j \widetilde{\Psi} (z_1, \tau) \bar{\hat{z}}_j \\
        & \quad - \int_0^1 \partial_i \partial_j \widetilde{\Psi} (z_t, \tau) \hat{z}_j dt + \partial_i \partial_j \widetilde{\Psi} (z_1, \tau) \hat{z}_j,
    \end{split}
    \end{align*}
    where $z_t = t z_1 + (1-t) z_2$ for some $t \in (0,1)$. By mean value theorem again, together with Lemma \ref{ckalphaesticomposfun}, we have
    \begin{align*} 
        ||Q(G_1, G_2; \partial \widetilde{\Psi})||_{k,\alpha; N \times \partial D}
        &\leq C P_k(1 + ||\hat{z}||_{k,\alpha; N \times D}) ||(\hat{z}, \hat{\xi})||_{k,\alpha; N \times D}^{2}
        \\
        &\leq C ||(\hat{z}, \hat{\xi})||^2_{k,\alpha; N\times D}
    \end{align*}
    where $C$ depends only on $n$, $k$, $\alpha$, $||\nabla^2 \widetilde{\Psi}_0||_{k+1,\alpha; N' \times \partial D}$ and $\varepsilon$.
\end{proof}

\begin{proof}[Proof of Theorem \ref{thmexistbypertb}]
Zehnder's version of the Nash-Moser iteration can be carried out using estimates \eqref{inverseesti} and \eqref{estiTalorformula}. Starting by the trivial data given in \eqref{cxdimntrivialfoli},  $G_0 = (z_0, \xi_0)$, the sequence $(G_n) = (z_n, \xi_n)$ can be defined inductively by
\begin{align*}
    G_n = G_{n-1} + \widehat{G}_{n}, \qquad \widehat{G}_{n} = -\eta|_{(G_{n-1}, \partial \widetilde{\Psi})} (T(G_{n-1}, \partial \widetilde{\Psi}) ).
\end{align*}
The key formula is the following:
\begin{align*}
    T(G_n, \partial \widetilde{\Psi}) = T(G_n, \partial \widetilde{\Psi})- T(G_{n-1}, \partial \widetilde{\Psi}) - D_\A T \big|_{(G_{n-1}, \partial \widetilde{\Psi})} (\widehat{G}_{n}) = Q(G_{n}, G_{n-1}; \partial \widetilde{\Psi})
\end{align*}
with the estimate from \eqref{inverseesti} to \eqref{estiTalorformula}, 
$$||T(G_n, \partial \widetilde{\Psi})||_{k,\beta; N \times \partial D} \leq C ||\widehat{G}_{n}||_{k,\beta; N \times D}^2 \leq C^3 (\alpha- \beta)^{-4} ||T(G_{n-1}, \partial \widetilde{\Psi})||_{k, \alpha; N \times \partial D}^2$$ 
for arbitray $\beta < \alpha$. Now fixing $0<\beta< \alpha<1$, Let $\beta_n = (\alpha-\beta) \exp(-\lambda \kappa^n) + \beta $ for $n \geq 1$. Taking $\kappa = \frac{3}{2} $,  $ \lambda = - \log \frac{(\alpha -\beta)^4}{4C^3}$ and 
$\varepsilon = \min \{\delta, \exp(-6 \lambda)\},$ where $\delta$ is the small constant such that Lemma \ref{lemlinearproblemperturb} holds, 
the following estimates can be proved inductively: 
\begin{align*}
    &||T(G_n, \partial \widetilde{\Psi}) ||_{k,\beta_n; N \times \partial D} \leq \exp (-6 \lambda (\kappa^{n}-1) )||\partial \widetilde{\Psi} - \partial \widetilde{\Psi}_0||_{k,\alpha; N' \times \partial D}, \quad n\geq 0,\\
    &||\widehat{G}_{n}||_{k,\beta_n; N \times D} \leq \exp (-6 \lambda \big(\kappa^{n-1}-1) + \lambda/2 \big) ||\partial \widetilde{\Psi} - \partial \widetilde{\Psi}_0||_{k,\alpha; N' \times \partial D}, \quad n\geq 1.
\end{align*}
The above estimates imply that $(G_n)$ converges in $\A^{k,\beta}$ and $T(G_n, \partial \widetilde{\Psi}) \rightarrow 0$ in $\D^{k,\beta}$ as $ n \rightarrow \infty$. Let $G = \lim G_n$ in $ \A^{k,\beta}$. Since the operator $T$ is continuous on $\A^{k,\beta}$, we have $T (G, \partial \widetilde{\Psi}) = 0$. Moreover, $G = (z, \xi)$ satisfies the following estimates:
\begin{align*}
    ||(z - z_0, \xi-\xi_0)||_{k,\beta; N \times D} \leq \sum_{n=1}^{\infty} ||\widehat{G}_n||_{k,\beta_n; N \times D} \leq 2C^2 (\alpha- \beta)^{-2} ||\partial \widetilde{\Psi} - \partial \widetilde{\Psi}_0||_{k,\alpha; N' \times D}.
\end{align*}
\end{proof}

\section{Uniqueness, Patching, and Global Foliation on the End} \label{secuniqholodiscfoli}

In this section, we address several key remaining steps to complete the proof of the main existence theorem. We begin by establishing a uniqueness result for holomorphic disc foliations near the trivial foliation. This will serve as the foundation for a patching argument, which allows us to glue together local foliations constructed in the previous section. With the patching theorem in place, we then complete the proof of the existence of holomorphic disc foliations on the end $X_\infty$.  

We begin by discussing the uniqueness aspect of the holomorphic disc foliation. The goal is to show that, under a small perturbation of the boundary data corresponding to the trivial foliation, the solution remains unique in a small neighborhood of $G_0 \in \A$ (see \eqref{defspaceholodiscs}). More precisely, suppose the free boundary data $\partial \widetilde{\Psi}$ in a local chart is a slight perturbation of that of the trivial foliation, $\partial \widetilde{\Psi}_0$, and satisfies
\begin{align} \label{unipartpertbcon}
    ||\partial \widetilde{\Psi}- \partial \widetilde{\Psi}_0||_{k+1, \alpha; N' \times \partial D} \leq \varepsilon_0.
\end{align}
Then we aim to prove that there exists a unique family of holomorphic discs $G: N \times D \rightarrow E$, lying in the space $\A^{k,\beta}_{\varepsilon_0}$, defined by 
\begin{align*}
    \A^{k,\beta}_{\varepsilon_0} = \{G \in \A; ||G||_{k,\beta; N \times D} \leq \varepsilon_0\}.
\end{align*}
Note that the constants $\varepsilon$, $\varepsilon_0$ are uniform, sufficiently small, positive constants and will be precisely determined in Subsection \ref{secuniq2ndver}.

One point that requires careful attention is that the holomorphic disc foliation is not uniquely determined if the boundary data is prescribed only at the level of K\"ahler forms. The following example demonstrates this point. Consider the following family of potential functions defined on $\CC^n \times \TT^1$, 
\begin{align}\label{cexnonuniqform}
    \psi_\tau (z) =\Psi(z, \tau) = \sum_{i=1}^n |z_i -\varepsilon_0 \varsigma_i (\tau)|^2,
\end{align}
where $\varsigma_i$ is a complex function defined on $\TT^1$ for $i = 1, \ldots n$ with $\varsigma_i(\sqrt{-1}) =0$. It is easy to verify that for each $\tau \in \TT^1$, the associated K\"ahler form satisfies: 
$$dd^c \psi_\tau = dd^c |z|^2.$$ 
In other words, the potential functions, $\psi_\tau$, $\tau \in \TT^1$, give rise to the standard K\"ahler form of Euclidean space.
Moreover, $\partial \widetilde{\Psi}$ is a small perturbation of $\partial \widetilde{\Psi}_0$ and satisfies \eqref{smallconfreebdry}. By Theorem \ref{thmexistbypertb}, for any smooth function $\varsigma_i$ with $|\varsigma_i|\leq 1$, there exists a family of holomorphic discs $G = (z, \xi): \CC^n \times D \rightarrow T^{(1,0)} (\CC^n) $ satisfying the estimates in \ref{estimeasureperturb}, such that its projection defines a holomorphic disc foliation on $ \CC^n \times D$, where $D$ is a unit disk with boundary $\TT^1$. If the functions \(\varsigma_i(\tau)\), \(1 \leq i \leq n\), are nontrivial complex-valued functions on \(\partial D\), then the associated family of holomorphic discs is nontrivial. In particular, if each \(\varsigma_i(\tau)\) extends holomorphically to \(D\), then the corresponding family of holomorphic discs \(G\) can be written explicitly as
\[
(z_i(w, \tau),\ \xi_i(w, \tau)) = \left(w_i - \varepsilon \varsigma_i(\tau),\ \bar{w}_i - \varepsilon \bar{\varsigma}_i(\tau)\right).
\]

The nonuniqueness of the holomorphic disc foliation arises because a given K\"ahler form \(\omega_\tau\), for \(\tau \in \partial D\), can admit many different potential functions—differing by real pluriharmonic terms. To resolve this ambiguity, one must fix the potential functions along the boundary, up to constants. This can be achieved by restricting to a suitable global class of potentials. For instance, on a compact K\"ahler manifold, we can work in the general K\"ahler potential space \(\mathcal{H}(\omega)\), which decomposes as \(\mathcal{H}_0(\omega) \oplus \mathbb{R}\), where \(\mathcal{H}_0(\omega)\) is the normalized potential space with its integral vanishing, and is in one-to-one correspondence with the K\"ahler forms in the same cohomology class of $\omega$. On an ALE K\"ahler manifold, we instead consider the class \(\mathcal{H}_{-\gamma}(\omega)\), where the decay at infinity ensures uniqueness. Once such a class is fixed, the potential functions corresponding to the K\"ahler form $\omega_\tau$, $\tau \in \partial D$ are uniquely determined, which do not affect the holomorphic disc foliation. If we discuss the holomorphic disc foliation in a local chart or in an open subset of a K\"ahler manifold, we always consider the potential functions as restrictions of global potentials. 
As we will see in Section~\ref{secmaxrkthm}, the local holomorphic disc foliation, constructed from potential functions that are restrictions of global potentials, retains information from the global solution to the HCMA equation.

The following lemma ensures that $\HH_{-\gamma}(\omega)$ uniquely determines an ALE K\"ahler form in the same cohomology class of $\omega$:

\begin{lem} \label{lemlapequonALE}
    Let $(X, g, J)$ be an ALE K\"ahler manifold with complex dimension $n\geq 2$. Let $f \in \C^{k-2, \beta}_{-\gamma-2}$ with $k \geq 2$ and $-\gamma<0$, then there exists a unique solution $u \in \C^{k,\beta}_{-\min\{\gamma, 2n\}}$ to 
    \begin{align*}
        \Delta u = f,
    \end{align*}
    where $\Delta$ is the Laplacian in terms of the reference metric $g$. 
\end{lem}

\begin{proof}
    The proof is a special case of \cite[Proposition 2.4]{yao2022mass}.
\end{proof}

Now, we assume that $\omega_1 = \omega_2$ in $X$ with $\omega_1 = \omega + dd^c \psi_1$ and $ \omega_2 = \omega + dd^c \psi_2$, $\psi_{1,2} \in \HH_{-\gamma} (\omega)$. Then, we have $\Delta_\omega (\psi_1 -  \psi_2)  =0$ on $X$. By lemma \ref{lemlapequonALE}, we have $\psi_1 = \psi_2$ on $X$. Hence, in our setting, prescribing the boundary data in terms of the exact Lagrangian subspace of the cotangent bundle—such as $\partial \psi_\tau$-is equivalent to prescribing the K\"ahler form-such as $dd^c \psi_\tau$-and, in turn, equivalent to prescribing the potential function, $\psi_\tau$.

The remainder of this section is organized as follows. In Subsection 7.1, we complete the proof of uniqueness. In Subsection 7.2, we establish the patching theorem. Subsections 7.3 through 7.5 are devoted to completing the proof of the general existence theorem for holomorphic disc foliations on $X_\infty \times D$, along with establishing weighted estimates for the displacement of the holomorphic discs.

\subsection{Local uniqueness of foliation by holomorphic discs with a fixed potential} \label{secuniq2ndver} 
Let $N$, $N'$, $N''$ be as before. 
Recall that $ \rho $ is the K\"ahler potential of reference K\"ahler metric $\omega_0$ in $N'$ and $\widetilde{\Psi}_0 (\cdot, \tau) = \rho$. As given in (\ref{cxdimntrivialfoli}), $G_0(w, \tau) = (w, \partial \widetilde{\Psi}_0 (w, \tau))$ with its boundary belong to $\Lambda_{\widetilde{\Psi}_0}$, and the projection down of $G_0$ to $N \times D$ gives the trivial foliation of $N \times D$.
Assume $\widetilde{\Psi}(z, \tau)$ is another smooth family of K\"ahler potentials on $N' \times \partial D$ by perturbing $\widetilde{\Psi}_0$ slightly and satisfying (\ref{unipartpertbcon}). 
In Theorem \ref{thmexistbypertb}, we prove the existence of families of holomorphic discs $G(w, \tau)$ such that the boundary of $G(w, \tau)$ belongs to $\Lambda_{\widetilde{\Psi}}$ for each $\tau \in \partial D$. In this subsection, we prove the local uniqueness of the family of holomorphic discs, $G(w, \tau)$, in a small neighborhood of $G_0$, fixing the free boundary condition given by $\partial \widetilde{\Psi}$. 

Recall that we define the set of smooth families of holomorphic discs, $\A$, as in the section \ref{seclocalexholodiscs}. To describe the local neighborhood of $G_0 \in \A$, we introduce notation, $\A^{k,\beta}_{\delta_0}(G_0)$. $\A^{k,\beta}_{\delta_0} (G_0)$ is the set of smooth families of holomorphic discs, $G(w, \tau) = (z(w, \tau), \xi(z, \tau)) $, satisfying,
    \begin{align}\label{familyholodiscsngbhcondi}
        ||z(w, \tau)- z_0(w, \tau)||_{k,\beta; N \times D} +||\xi(w, \tau)- \xi_0(w, \tau)||_{k,\beta; N \times D} \leq \delta_0.
    \end{align}
where $G_0(w, \tau) = (z_0(w, \tau), \xi_0(z, \tau)) $ and $0< \beta< \alpha<1$.

\begin{thm} \label{thmlocalunifixpotential}
    Let $\omega_0$, $\widetilde{\Psi}_0$, $G_0= (z_0, \xi_0)$ be the same as in Theorem \ref{thmexistbypertb}.
    Let $\widetilde{\Psi}(z,\tau)$ be another smooth family of potential functions obtained by a small perturbation of $\widetilde{\Psi}_0$ satisfying,
    \begin{align} \label{condithmsmallperturb}
        ||\partial \widetilde{\Psi}- \partial \widetilde{\Psi}_0||_{k+1,\alpha; N' \times \partial D} \leq \varepsilon_0,
    \end{align} 
    for $k \geq 1$ and $\alpha \in (0,1)$. 
    Let $\A^{k, \beta}_{\delta_0} (G_0)$ denote the set of smooth families of holomorphic discs satisfying condition (\ref{familyholodiscsngbhcondi}). Assume $\varepsilon_0$ and $\delta_0$ are small constants depending only on $n$, $k$, $\alpha$, $\sigma^{-1}$, $||\nabla^2_X \widetilde{\Psi}_0||_{k+1, \alpha, N' \times \partial D}$.
    Then, there exists a unique smooth family of holomorphic discs $G(w, \tau)$, in a small neighborhood of $G_0$, $\A^{k,\beta}_{\delta_0}(G_0)$, satisfying $(i)-(v)$ of Theorem \ref{thmexistbypertb}, with the free boundary condition $G(w, \tau) \in \Lambda_{\widetilde{\psi}_{\tau}}$.
\end{thm}

The key step of proving Theorem \ref{thmlocalunifixpotential} is solving the linearized problem at $(\widetilde{G}, \partial \widetilde{\Psi})$ close to $(G_0, \partial \widetilde{\Psi}_0)$, where $\widetilde{G} \in \A^{k,\beta}_{\delta_0} (G_0) $ and $\partial \widetilde{\Psi}$ satisfies (\ref{condithmsmallperturb}). Let $\widetilde{G}(w, \tau) = (\tilde{z}(w,\tau), \tilde{\xi}(w,\tau))$, the linearized problem at $(\widetilde{G}, \partial \widetilde{\Psi})$ is given as follows,
  \begin{align}\label{holodisclinearpdenearby}
    \begin{split}
        &\overline{\partial}_\tau \hat{z}_i (w, \tau) = \overline{\partial}_\tau \hat{\xi} (w, \tau) = 0, \hspace{5.2cm} \text{ in }N \times D,\\
        \begin{split}
         &\hat{\xi}_i(w,\tau) - \big(\partial_i \overline{\partial}_{{j}} \widetilde{\Psi} \big)(\tilde{z}(w, \tau), \tau) \overline{\hat{z}}_j(w, \tau)\\ 
         &\hspace{2.3cm}- \big(\partial_{i}\partial_{j} \widetilde{\Psi}\big) (\tilde{z}(w, \tau), \tau) \hat{z}_j(w, \tau)=\bff_i(w,\tau),
         \end{split}
         \hspace{0.5cm}  \text{ in } N \times \partial D,\\
         &\hat{z}_i(w, -i) = 0,  \hspace{7.5cm}  \ w  \in N.
    \end{split}
\end{align}
Since $\widetilde{G} = (\tilde{z}, \tilde{\xi}) \in \mathcal{A}_{\delta_0}^{k,\beta}$, and $\widetilde{\Psi}$ satisfies (\ref{condithmsmallperturb}), we have 
\begin{align*}
    ||\partial_i \overline{\partial}_j \widetilde{\Psi}(\tilde{z}(w, \tau), \tau) - \partial_i \overline{\partial}_j \widetilde{\Psi}_0(z_0(w, \tau), \tau)||_{{k,\beta}; N \times \partial D } \leq C_0 \delta_0 + \varepsilon_0,
\end{align*}
and
\begin{align*}
    ||\partial_i {\partial}_j \widetilde{\Psi}(\tilde{z}(w, \tau), \tau) - \partial_i {\partial}_j \widetilde{\Psi}_0(z_0(w, \tau), \tau)||_{{k,\beta}; N \times \partial D } \leq C_0 \delta_0 + \varepsilon_0.
\end{align*}
where we can take the constant $C_0$ to be $||\nabla^2 \widetilde{\Psi}_0||_{k+1,\alpha; N' \times D}$.

The uniqueness of the linearized problem (\ref{holodisclinearpdenearby}) is proved in Lemma \ref{lemlinearproblemperturb}. To prove Theorem \ref{thmlocalunifixpotential}, we pick $\delta_0$ and $\varepsilon_0$ as follows,
\begin{align*}
    \delta_0 = \frac{\delta}{2 C_0}, \qquad \varepsilon_0 = \frac{\delta}{2},
\end{align*}
where $\delta $ is the small constant given in $(\ref{parachoosedelta})$.

\begin{proof}[proof of Theorem \ref{thmlocalunifixpotential}] Fixing a $\widetilde{\Psi}$ satisfying (\ref{condithmsmallperturb}), we assume that there exist two distinct families of holomorphic discs $G(w, \tau)$ and $\widetilde{G}(w, \tau)$ satisfying $(i)-(v)$ in Theorem \ref{thmexistbypertb}. Write $G(w, \tau)$ and $\widetilde{G} (w, \tau)$ explicitly in the coordinates of $E$, $G(w, \tau) = ({z}(w, \tau), {\xi} (w, \tau))$ and $\widetilde{G} (w, \tau) = (\tilde{z}(w, \tau), \tilde{\xi} (w, \tau)))$. Apply the operator in (\ref{defopdiffholodiscbdry}), we have
\begin{align*}
    T(G, \partial \widetilde{\Psi} ) =0 = T(\widetilde{G}, \partial \widetilde{\Psi})
\end{align*}
If we connect $G$ to $\widetilde{G}$ with a curve in $\A$, $G(t) = (1-t) G + t\widetilde{G}$, there exist $t_0 \in (0,1)$ such that,
\begin{align} \label{keykernel}
    \D_{\A} T\big|_{(G(t_0), \partial \Phi)} (\widetilde{G}- G) =0.
\end{align}
Therefore, $\widetilde{G}(w, \tau) - G (w, \tau) \in \A'$ is a solution to (\ref{holodisclinearpdenearby}) with $\bff =0$. Notice that $||G(t_0)||_{k,\beta; N \times D} \leq \varepsilon_0$. Lemma \ref{lemlinearproblemperturb} implies that the kernel of $\D_\A T|_{(G(t_0), \partial \Phi)}$ is trivial in $(\A^{k,\beta'})'$ with $0<\beta' < \beta$, which implies the local uniqueness of the solution.
\end{proof}

\subsection{A Patching Theorem to Foliation by Holomorphic Discs} \label{subsecpatthm}

By slightly revising the statement of Corollary \ref{corcxasymptoticcovering}, there exists a countable family of triples $\{(N_i'', N_i, N_i' : N_i'' \subseteq N_i \subseteq N_i', \ i \in \mathcal{I} )\}$. For each $i \in \mathcal{I}$, there exists a biholomorphism $I_i: (N_i'', N_i, N_i') \rightarrow (B_{R_i -1}, B_{R_i}, B_{R_i +1}) $, where $B_{R}$ is a ball in $\CC^n$ centered at the origin with radius $R$. The radius $R_i$ is taken as in Corollary \ref{corcxasymptoticcovering}, $R_i = \delta r(x)$, where $x$ is the point mapped to the origin by $I_i$, i.e., $I_i (x) = 0$. The set $\{U'': i \in \mathcal{I}\}$ is a locally finite and countable open covering of the set $X_l \subseteq  X_\infty$. Let $\psi_{0,i}$ be the potential function of the reference K\"ahler metric $g$ in $N_i'$ and denote $\Psi_{0, i}$ to be the pull-back function of $\psi_{0,i}$ in $N_i' \times \partial D$. By slightly perturbing the boundary data, we have a new smooth function, $\widetilde{\Psi}$, defined on $N'_i \times \partial D$ such that $\widetilde{\Psi} \in \C^{\infty} (N_i'\times \partial D)$ and satisfies 
\begin{align*}
    ||\partial \widetilde{\Psi}_i - \partial \widetilde{\Psi}_{0, i} ||_{k+1, \alpha; N_i'\times \partial D} \leq \varepsilon,
\end{align*}
where $\varepsilon$ is a small constant given in Theorem \ref{thmexistbypertb} and Theorem \ref{thmlocalunifixpotential} and independent of $i \in \mathcal{I}$. Let $E_i$ be the holomorphic cotangent bundle of $N_i'$ for each $i \in \mathcal{I}$. By Theorem \ref{thmexistbypertb}, there exists a smooth family of holomorphic discs $G_i: N_i \times D \rightarrow E_i$ satisfying $(i)$-$(v)$ in Theorem \ref{thmexistbypertb}. The following theorem describe the relation between $G_i$ and $G_j$ for different $i,j \in \mathcal{I}$:

\begin{thm}\label{thmpatch}
    Given the data $(N_i'', N_i, N_i', \widetilde{\Psi}_{i}, \widetilde{\Psi}_{0,i})$, $i \in \mathcal{I}$ as above, there is a unique smooth family of holomorphic discs $ G_i : N_i \times D \rightarrow E_i$ for each $i \in \mathcal{I}$ satisfying conditions $(i)$-$(v)$ and \eqref{estimeasureperturb} in Theorem \ref{thmexistbypertb}. Let $p_i: E_i \rightarrow N_i'$ be the standard projection and let $H_i = \pi \circ G_i$. Then, for $i, j \in \mathcal{I}$, we have 
    \begin{align*}
        H_i (w, \tau) = H_j (w, \tau), \qquad (w, \tau) \in \big( N_i \bigcap N_j \big) \times D.
    \end{align*}
    Furthermore, if we write $G_i(w, \tau) = (z_{1,i}, \ldots, z_{1,i}, \xi_{n,i}, \ldots, \xi_{n,i})$ and $G_j(w, \tau) = (z_{1, j},\ldots, z_{n,i}, \xi_{1,j}, \ldots \xi_{n, j})$, then we have,
    \begin{align*}
        \begin{split}
            z_{p, j} (w, \tau) &= z_{p,i} (w, \tau),\\
            \xi_{p,j} (w, \tau) &= \xi_{p,i}(w, \tau) + \partial_{p} \big(\tilde{\psi}_{0,j}- \tilde{\psi}_{0,i} \big)(w)
        \end{split}
          \hspace{1cm}  p= 1,\ldots, n,\quad (w, \tau) \in  \big( N_i \bigcap N_j \big) \times D.\\
    \end{align*}
\end{thm}

\begin{proof}
    Fixing indices $i,j \in \mathcal{I}$, we have two sets of data, $(N_i'', N_i, N_i', \widetilde{\Psi}_i, \Psi_{0,i}, G_i )$ and $(N_j'', N_j, N_j', \widetilde{\Psi}_j, \Psi_{0,j},\\ G_j )$. By restricting the families of holomorphic discs $G_i$, $G_j$ to $N_i \cap N_j$, it is easy to check that $G_i,\ G_j: N_i \cap N_j \rightarrow E_{ij}$ are the smooth families of holomorphic discs satisfying the conditions $(i)$-$(v)$. The boundary conditions of $G_i$, $G_j$ differ; precisely, if we write $G_i = (z_{1,i},\ldots, z_{n,i}; \xi_{1,i}, \ldots, \xi_{n,i} )$ and $G_j = (z_{1,j}, \ldots, z_{n,j}; \xi_{1,j},\ldots \xi_{n,j})$, we have
    \begin{align*}
    \begin{split}
        \xi_{p, i} = \partial_p \widetilde{\Psi}_i (z_{,i} (w, \tau), \tau), \\
        \xi_{p, j} = \partial_{p} \widetilde{\Psi}_j (z_{,j} (w,\tau), \tau), 
    \end{split}
    \qquad (w, \tau) \in \big( N_i \bigcap N_j \big) \times \partial D.
    \end{align*}
    $\widetilde{\Psi}_j$ differs with $\widetilde{\Psi}_i$ by exactly the difference of K\"ahler potentials of $\omega$ in $N_j'$ and $N_i'$,
    \begin{align*}
        \widetilde{\Psi}_j (z, \tau) = \widetilde{\Psi}_i(z, \tau) + \tilde{\psi}_{0,ij} (z), \qquad (z, \tau) \in \big( N_i'\bigcap N_j'\big) \times \partial D,
    \end{align*}
    where $\tilde{\psi}_{0, ij} = \tilde{\psi}_{0,j} - \tilde{\psi}_{0,i}$.

    In this paragraph, it will be checked that 
    \begin{align}\label{checkdatazxij}
    \begin{split}
        z_{p, j} (w, \tau) &= z_{p, i} (w, \tau), \\
        \xi_{p,j} (w, \tau) &= \xi_{p, i} (w, \tau) + \partial_p \tilde{\psi}_{0,ij}(z_{,i}(w, \tau)),
    \end{split}
    \quad (w, \tau) \in \big( N_i\bigcap N_j\big) \times \partial D.
    \end{align}
    Provided that $(z_{a,i}, \xi_{p,i})$ is holomorphic with respect to $\tau$, it suffices to check $\xi_{p,j}(w, \tau)$ is holomorphic. The result can be derived from the following calculation:
    \begin{align*}
        \overline{\partial}_{\tau} \xi_{p,j} (w,\tau) &= \overline{\partial}_{\tau} \xi_{p,i}(w, \tau) + \overline{\partial}_{\tau} \big[ \partial_p \tilde{\psi}_{0,ij}(z_{,i}(w, \tau))\big]\\
        &= \partial_q \partial_{p} \tilde{\psi}_{0,ij} (z) \frac{\partial z_q}{\partial \bar{\tau}} (w, \tau) +  \overline{\partial}_{q} \partial_{p} \tilde{\psi}_{0,ij}  (z) \frac{\partial \bar{z}_q }{\partial \bar{\tau}} (w, \tau)\\
        &= 0.
        \end{align*}
        By estimate \eqref{estimeasureperturb}, we can assume $G_i \in \A^{k,\beta}_{\varepsilon_0}(G_{0,i})$ in $N_i \times D$. Then, $(z_{p,j}, \xi_{p,j})$ defined in \eqref{checkdatazxij} is contained in $\A^{k,\beta}_{\varepsilon_0}(G_{0,j})$ in $(N_i\cap N_j)\times D$. The local uniqueness of the family of holomorphic discs implies that 
        \begin{align*}
             G_j = \big(z_{, i}(w, \tau), \xi_{,i}(w,\tau) + \partial \tilde{\psi}_{0,ij}(z_{,i}(w,\tau))\big).
        \end{align*}
        In conclusion, we have $H_i(w, \tau) = H_j(w, \tau)$, for $(w, \tau) \in ( N_i \cap N_j ) \times D$.
\end{proof}

We conclude this subsection by reconstructing the family of holomorphic discs over over the end $X_l \subseteq X_\infty$. Recall that for each $\tau \in \partial D$, the K\"ahler form is given by $\omega_\tau = \omega + dd^c \psi_\tau$ with $\psi_\tau \in \HH_{-\gamma}$, and $\Psi (\cdot, \tau) = \psi_\tau(\cdot) $ with $\Psi \in \C^{\infty}_{-\gamma} (X \times \partial D)$.
We assume the boundary data satisfies the following uniform bounds on $X_\infty$ with respect to the Euclidean metric $g_0$; or equivalently, with respect to the reference ALE K\"ahler metric $g$
\begin{align*}
    ||\partial {\Psi}||_{k+1,\alpha; X_\infty} \leq \varepsilon
\end{align*}
Recall also that $X_l \subseteq X_\infty$ can be covered by countable holomorphic balls, $\{N_i''; ~ i \in \mathcal{I}\}$, with nested inclusion $N_i'' \subseteq N_i \subseteq N_i'$ for each $i \in \mathcal{I}$.
By Lemma \ref{lemholocornearinfty}, if we restrict the uniform bound above to each holomorphic coordinate chart $N'_i$, we obtain $ ||\partial \widetilde{\Psi}_i - \partial \widetilde{\Psi}_{0,i}||_{k+1,\alpha; N'_i} \leq \varepsilon$, where $\varepsilon>0$ is a small constant independent of $i \in \mathcal{I}$. Here, $\widetilde{\Psi}_{0, i}$ is the local potential function of $\omega$ in $N'_i$, and $\widetilde{\Psi}_i = \Psi + \widetilde{\Psi}_{0,i}$ is the adjusted local potential associated with $\omega_\tau$ for $\tau \in \partial D$. 

We now construct a holomorphic fiber bundle $\mathcal{W}$ over $X_l$ by gluing together the local cotangent bundles $E_i = T^* N_i$ and $E_j = T^* N_j$ over $N_i \bigcap N_j$. The gluing is performed via the additive transition function $\partial (\widetilde{\Psi}_{0,i} - \widetilde{\Psi}_{0,j})$: that is, for any point $x \in N_i \bigcap N_j$ and any covector $\xi\in T^*_x N_i$, we identify $\xi \sim \xi+ \partial (\widetilde{\Psi}_{0,i} - \widetilde{\Psi}_{0,j}) \in T^*_x N_j$. This provides a well-defined fiber bundle $\mathcal{W}\rightarrow X_l$, modeled locally on the cotangent bundles $E_i$. Moreover, $\mathcal{W}$ admits a global complex symplectic $(2,0)$ form $\Xi$, which is locally given by $\Xi|_{E_i}= \sum_p d\xi_p \wedge dz_p$, where $(z_p, \xi_p; 1\leq p \leq n)$ are the standard holomorphic coordinates of $E_i$. 
By the discussion in Section \ref{secholodiscHCMA}, for each $\tau \in \partial D$, the boundary data $\partial \Psi$ defines an exact Lagrangian submanifold $\Lambda_{\psi_\tau}$ on $\mathcal{W}$ and can be locally represented by $\partial \widetilde{\Psi}_{i}$ on each $E_i$ with
\begin{align*}
    \Xi|_{\Lambda_{\psi_\tau} \cap E_i} = \partial \overline{\partial} \widetilde{\Psi}_{i};\quad \text{hence, } \quad \Xi|_{\Lambda_{\psi_\tau}} = -i\omega_{\tau}. 
\end{align*}

The family $G$, defined over $X_l \times D$, can be regarded as a collection of smooth submanifolds of the total space $\mathcal{W}$ assembled by gluing the local families $G_i$ over $N_i \times D$ across the overlap $N_i \bigcap N_j$, in accordance with the patching theorem.
Precisely, by Theorem \ref{thmexistbypertb}, for each $i \in \mathcal{I}$, we obtain a smooth family of local holomorphic discs $G_i : N_i \times D \rightarrow E_i$ defined over the holomorphic coordinate chart $N_i$, where $E_i \rightarrow N_i'$ is the trivial holomorphic cotangent bundle. The patching theorem, Theorem \ref{thmpatch}, implies that the data $(G_i, N_i)$, $i \in \mathcal{I}$ can be glued together to obtain the family of holomorphic discs over $X_l$, $G: X_l \times D \rightarrow \mathcal{W}$. 
The projection of $G$ under $\mathcal{W} \rightarrow X_l$ defines the holomorphic discs foliation of $X_l \times D$.
According to Proposition \ref{propexplicitconstructionsolu}, the correspondence between the holomorphic disc foliation and the collection of exact Lagrangian submanifolds can extend to each $\tau \in D$. Locally, for each $\tau$, the Lagrangian submanifold $\Lambda_{\tau} = G(\cdot, \tau)$ is given by $\partial\widetilde{\Phi}_i$, where $\widetilde{\Phi}_i$ is defined by solving \eqref{explicitconstructionsolu} along the family of holomorphic discs defined by $G_i: N_i \times D \rightarrow E_i$. Furthermore, by the patching theorem, one can easily verify that on the holomorphic disc leaf at the $x \in N_i \bigcap N_j$, the local functions satisfy $\widetilde{\Phi}_i - \widetilde{\Psi}_{0,i} = \widetilde{\Phi}_{j} - \widetilde{\Psi}_{0,j} $. This implies the existence of a function $\Phi$ on $X_l \times D$ such that
\begin{align*}
    \phi_\tau (\cdot) = \Phi(\cdot, \tau) =  \widetilde{\Phi}_i - \widetilde{\Psi}_{0,i}, \qquad\text{ for each } i \in \mathcal{I}.
\end{align*}
As a result, the restriction of the holomorphic symplectic form to the Lagrangian submanifold satisfies $$\Xi|_{\Lambda_\tau} = \omega_\tau = \omega + dd^c \phi_\tau $$ for each $\tau \in D$.

\subsection{The Weighted Estimates} 
In this section, we prove a weighted estimate for the coordinates shifting of the holomorphic discs under the assumption that the perturbation of boundary data satisfies the decay condition ${\Psi} \in \C^{\infty}_{-\gamma} (X_\infty) $. The proof depends heavily on the existence theorem (Theorem \ref{thmexistbypertb}) and the scaling technique (Lemma \ref{lemscalweightedholder}). According to Lemma \ref{lemholocornearinfty} and Corollary \ref{corcxasymptoticcovering}, we perturb the boundary data of the trivial foliation in each open holomorphic ball in the open covering of $X_\infty$. In each holormphic ball $N'_i$, there exists a K\"ahler potential, $\widetilde{\Psi}_{0,i}$, of the reference K\"ahler form, and let $\widetilde{\Psi}_i = \widetilde{\Psi}_{0,i} + \Psi $ in $N_i'$. We now select a family $\{\widetilde{\Psi}_{0,i}\}$ on each $N'_i$, subject to uniform bounds that are suitable for applying the existence theorem.

By Corollary \ref{corcxasymptoticcovering}, there exists a countable family of pairs $\{N_i \subseteq N_i'; i \in \mathcal{I}\}$ such that the pair $(N_i, N_i')$ is biholomrphic to $(B_{R_i}, B_{2R_i} )$ satisfying coordinate changes \eqref{weightesticoordinatesholocball}, where $R_i = \delta r(x_i)$ and $x_i$ is the center of the chart. Moreover, the collection $\{N_i; i \in \mathcal{I}\}$ forms an open covering of $X_l$ for some $X_l = \{x\in X; r(x) > l \} \subseteq X_\infty$.
Let $s_i: B_2 \rightarrow B_{2R_i}$ be the standard scaling map. 
We first prove the following lemma:

\begin{lem} \label{lemuniformestipotentialref} Let $X$ be an ALE K\"ahler manifold with the reference K\"ahler form $\omega$, satisfying the fall-off condition \eqref{decayaleintro}.
    On each $N_i'$, there exists a $\widetilde{\Psi}_{0,i}$ such that $||\nabla_X^2 \widetilde{\Psi}_{0,i}||_{0;k+1, \alpha; N_i' \times \partial D} \leq C_0$ for a uniform constant $C_0$ independent of $i$.
\end{lem}

\begin{proof}
    Let $\omega$ be the reference K\"ahler form on $X$. By restricting to $N_i'$ and pulling back under scaling map $s_i$, we defined a K\"ahler form on $B_2$:
    \begin{align*}
        \omega^*_i = R_i^{-2} \big(s_i^* \omega\big) \big|_{N_i'}. 
    \end{align*}
    By the fall-off condition \eqref{decayaleintro}, along with the estimates on coordinate transform \eqref{weightesticoordinatesholocball}, if we write $\omega$ in holomorphic coordinates in $N_i$, we have $\omega_{i\bar{j}}= \delta_{ij} + O(R_i^{-\mu})$. Pulling-back by $s_i$, under the standard Euclidean metric on $B_2$, we have 
    $||\omega^*_i||_{k+1,\alpha; B_1 \times D} \leq C$, where $C$ is a uniform constant independent of $i \in \mathcal{I}$. Since $d\omega_i^* = 0$ on $B_2$, there is a $(0,1)$ form $v^*$ with $\overline{\partial} v^* =0$ and $\omega_i^* = d (v^* + \bar{v^*})$.
    If we write $\omega^*_i = i h_{i\bar{j}} d z^i\wedge d\bar{z}^j$ on $B_2$, the explicit formula for $v$ is given by 
    \begin{align*}
        v^* = i \Big( \int_0^1 h_{i\bar{j}} (tz) t dt \Big) z^i d\bar{z}^j.
    \end{align*}
    Then, it's easy to verify using the above formula that $\sup_{B_2}|v|\leq \sup_{B_2} |\omega^*_i|$. By \eqref{Linftyesti} in the proof of Proposition \ref{prodbaresti}, there exists a function $\widetilde{\Psi}^*_0$ in $B_{\frac{3}{2}}$ such that $\sup_{B_{3/2}}| \widetilde{\Psi}^*_0| \leq \sup_{B_2} |v|$. By applying classic interior Schauder's estimates to $\Delta \widetilde{\Psi}^*_0 = \tr_{g_0} \omega^*_i$, we have $||\nabla^2 \widetilde{\Psi}_0^*||_{k+1,\alpha; B_1} \leq C_0$. Let $ \widetilde{\Psi}_0 = R_i^2 (s_i^{-1})^* \widetilde{\Psi}^*_0$. Then, $\widetilde{\Psi}_0$ is a K\"ahler potential of $\omega$ and, by Lemma \ref{lemscalweightedholder}, $||\nabla_X^2 \widetilde{\Psi}_0||_{0;k,\alpha; N_i \times \partial D} \leq C_0$.
\end{proof}

 Then, we have the following theorem: 

\begin{thm} \label{thmweightestifoli}
Let $\widetilde{\Psi}_0$, $G_0(w, \tau) = (z_0(w, \tau), \xi_0(w, \tau))$ be the same as in Theorem \ref{thmexistbypertb}. Suppose that $C_0$ is the uniform constant such that $||\nabla_X^2 \widetilde{\Psi}_{0,i}||_{0;k+1, \alpha; N_i' \times \partial D} \leq C_0$ and 
\begin{align}\label{weiconfreebdry}
    || \partial {\Psi} ||_{-\gamma; k+1,\alpha; X_\infty \times \partial D} \leq C_0.
\end{align}
Then,
there is a uniform constant  $l$ only depending on $n$, $k$, $\alpha$, $\sigma^{-1}$ and $C_0$, such that on $X_l \subseteq X_\infty$, 
there exists a smooth family of holomorphic discs $G: X_l \times D \rightarrow  \mathcal{W}_\infty$ satisfying the conditions (i)-(v) in Theorem \ref{thmexistbypertb}. If we write $G(w, \tau) $ in a holomorphic ball, $N_i$, from the open covering of $X_l$, $G(w, \tau) = (z(w, \tau), \xi(w,\tau))$, and $G_0 (w,\tau) = (z_0(w, \tau), \xi_0 (w,\tau)) $ be the trivial foliation in $N_i$, we have
\begin{align} \label{estiweiholodisccoordishift}
    ||(z-z_0,\xi-\xi_0)||_{-\gamma; k,\beta; N_i\times D} 
    \leq C || \partial {\Psi}||_{-\gamma; k, \alpha ; X_\infty \times \partial D}, 
\end{align}
where $C$ is a uniform constant depending on $n, k,\alpha, \sigma^{-1}, (\alpha-\beta)^{-1} \text{ and } C_0$.
\end{thm}
\begin{proof} 
Let $X_\infty$ be the end of $X$ defined by $X_\infty = \{x \in X, r(x) > R_0\}$ and there is a diffeomorphism $I_0: X_\infty \rightarrow (\CC^n- \overline{B}_{R_0})/\Gamma$. In this proof, we fix the constant $l$ by setting
\begin{equation}\label{pickl}
l := \max \bigg\{ 2R_0,; \left( \frac{\kappa C_0}{\varepsilon} \right)^{\tfrac{1}{\gamma+2}} \bigg\},
\end{equation}
where $\varepsilon$ is the small constant introduced in \eqref{smallconfreebdry} of Theorem~\ref{thmexistbypertb}, and $\kappa$ is the uniform ratio appearing in Lemma~\ref{lemholocornearinfty}. More precisely, $(N_i, N'_i)$ can be identified with $(B_{R_i/2}(x_i),, B_{R_i}(x_i))$, where $R_i = \kappa r(x_i)$.
Consider the scaling map $s_i: B_1 \rightarrow B_{R_i}$. 
By pulling back the boundary data $\Psi$ and $\widetilde{\Psi}_0$ to $B_1$, we may apply Theorem~\ref{thmexistbypertb} to obtain a smooth family of holomorphic discs on $B_{\tfrac{1}{2}} \subseteq B_1$. Specifically, set  
\[
    \widetilde{\Psi}_0^* := R_i^{-2} s^* \widetilde{\Psi}_0, 
    \qquad 
    \widetilde{\Psi}^* := R_i^{-2} s^* (\Psi + \widetilde{\Psi}_0).
\]  
After rescaling, Lemma~\ref{lemscalweightedholder} together with \eqref{pickl} shows that condition \eqref{weiconfreebdry} can be reformulated as  
\[
    \|\partial^* \widetilde{\Psi}^* - \partial^* \widetilde{\Psi}_0^*\|_{k+1, \alpha; B_1 \times \partial D} \leq \varepsilon,
\]  
where $\partial^*$ is the complex differential operator with respect to the $\alpha$ complex coordinate in $B_1$.
Hence, by Theorem~\ref{thmexistbypertb}, there exists a smooth family of holomorphic discs  
\[
    G^*(w^*, \tau) = \big(z^*(w^*, \tau),\, \xi^*(w^*, \tau)\big), 
    \qquad (w^*, \tau) \in B_{\tfrac{1}{2}} \times D.
\] 
Then, $(z^*, \xi^*)(w^*, \tau)$ is holomorphic in $\tau$ and satisfies the following free boundary condition:
\begin{align*}
    \xi_\alpha^*(w^*, \tau) = \partial^*_\alpha \widetilde{\Psi}^*(z^*(w^*,\tau), \tau), \hspace{0.5cm} (w^*,\tau) \in B_{\frac{1}{2}} \times D,
\end{align*} 
In addition, a fixed point 
$$z_0(w^*,-i) =w^*$$ 
is satisfied. Also by Theorem \ref{thmexistbypertb}, $(z^*, \xi^*)$ satisfies the estimates $||(z^*-z^*_0, \xi^*- \xi^*_0)||_{k,\beta; B_{\frac{1}{2}}\times D} \leq C ||\partial^* \widetilde{\Psi}^*- \partial^* \widetilde{\Psi}_0^*||_{k, \alpha; B_1 \times \partial D}$. By Lemma \ref{lemuniformestipotentialref}, the constant $C$ is uniform and independent of the choice of the holomorphic balls $N_i$ in the open covering.
Let 
$$(z, \xi) (w, \tau) = R_i \cdot \big\{(s_i^{-1})^* (z^*, \xi^*)\big\} (w,\tau)$$ 
for $(w, \tau) \in B_{\frac{R_i}{2}} \times D$. It can be verified that $(z,\xi) (w,\tau)$ is holomorphic in $\tau$ and satisfies
\begin{align*}
    \xi_p(w, \tau)&= R_i \cdot \partial^*_p \widetilde{\Psi}^* (z^* (s_i^{-1}(w), \tau), \tau)\\
    & = \partial_p \widetilde{\Psi} (R_i \cdot \{(s_i^{-1})^* z^*\} (w, \tau), \tau) \\
    &= \partial_p \widetilde{\Psi} (z(w, \tau), \tau),
\end{align*}
 and the fixed point condition follows from 
$$ z(w, -i) = R_i \cdot \{(s_i^{-1} )^* z^*\} (w, -i) = R_i \cdot s_i^{-1} (w) =w.$$
Hence, $G(w,\tau) = (z, \xi) (w, \tau)$, $(w, \tau) \in N_i \times D$ is the smooth family of holomorphic discs satisfying $(i)$-$(v)$ in Theorem \ref{thmexistbypertb}, where the free boundary condition is given by $\partial \widetilde{\Psi}$. By the patching theorem (Theorem~\ref{thmpatch}), the holomorphic disc foliations on each $N_i$ can be patched together to produce a global family of holomorphic discs on  $X_l$, $G: X_l \times D \rightarrow \mathcal{W}_{\infty}$. Restricting to each $N_i$,
note that $(z_0, \xi_0) (w,\tau) = R_i \cdot \{(s_i^{-1})^* (z_0^*, \xi_0^*)\} (w, \tau)$ and $\partial \Psi = R_i (s_i^{-1})^*  (\partial^* \widetilde{\Psi} - \partial^* \widetilde{\Psi}_0 )$; hence, we have 
$$R_i^{\gamma} ||s_i^*(z-z_0, \xi-\xi_0)||_{k,\alpha; B_{\frac{1}{2}}\times D} \leq  R_i^{\gamma} C ||s_i^* \partial \Psi||_{k,\alpha; B_1 \times D},$$ 
where $C $ is a uniform constant depending on the data $n, k,\alpha, \sigma^{-1}, (\alpha-\beta)^{-1} \text{ and }  C_0$.
By scaling technique (Lemma \ref{lemscalweightedholder}),
\begin{align*}
    ||(z-z_0,\xi-\xi_0)||_{-\gamma; k,\beta; X_l \times D} 
    \leq C || \partial {\Psi}||_{-\gamma; k, \alpha ; X_\infty \times \partial D}.
\end{align*}
\end{proof}

\section{The construction of global subsolution} \label{secmaxrkthm}

In this section, we complete Step 3 (without weighted H\"older estimates) of the proof of Theorem \ref{mthmasymbehHCMA}, which concerns the global construction of a solution to HCMA equation on $X \times D$. Building on the holomorphic disc foliation constructed over the end $X_l \times D$, the goal is to extend the data at the infinity to a globally defined $\Omega$-psh subsolution on the full space $X \times D$. Roughly speaking, we construct a globally defined bounded continuous $\Omega$-psh funcion $F$ on $X \times D$ such that $F$ is smooth on $X_{2l} \times D$ and $F$ agrees with the global $\C^{1,1}$ solution, $\Phi$, given in Theorem \ref{thmc11soluasupenvelope}, on $X_{2l} \times D$. Recall that $\F_{\Omega, \Psi}$ denotes the space of bounded continuous $\Omega$-psh subsolution to HCMA equation with boundary data prescribed by $\Psi$.
The main theorem is the following:

\begin{thm}\label{thmconstpshsolu}
    There exists a global defined function $F$ on $X\times D$ satisfying:
    \begin{itemize}
        \item[(i)] $ F \in \F_{\Omega,\Psi}$ and $F$ is smooth in $X_{2l} \times D$;
        \item[(ii)] Let $\Phi$ be the unique bounded continuous solution to the HCMA equation. Then, for a uniform large constant $l>0$ given as in Theorem \ref{thmweightestifoli}, we have $F = \Phi $ on $X_{2l} \times D$.
    \end{itemize}
\end{thm}

\subsection{The convexity of ALE K\"ahler potentials} \label{seccvxpotential}
Let $(X,g,J)$ be an ALE K\"ahler manifold, where the metric $g$ asymptotically decays to the flat metric $g_0$, as specified in (\ref{decayaleintro}). Let $X_\infty$ be the end of $X$.
As in Section \ref{secholodiscHCMA}, we denote $N \subseteq N'\subseteq X_\infty$ to be either holomorphic balls that can be identified as $B_R$ and $B_{2R}$ respectively in $\CC^n$ ($n\geq 2$), or the complement of a compact set of $X$ whose universal covering is biholomorphic to $\CC^n\backslash B_{2R}$ and $\CC^n\backslash B_{R}$ respectively ($n\geq 3$).  Let $\omega$ be the K\"ahler form corresponding to $(g, J)$. By restricting $\omega$ on $N'$, we have $\omega = \frac{i}{2} \partial \overline{\partial } |z|^2 + \omega'$, where $z$ is the standard holomorphic coordinates by identifying $N'$ as a subset of $\CC^n$. The asymptotic conditions of $(g, J)$ then imply:
\begin{align} \label{decaycondKform}
    \omega' = i h_{i\bar{j}} d z^i \wedge d \bar{z}^j, \qquad |h_{i\bar{j}}
| \leq C r^{-\tau}  \text{ in }  N',
\end{align}
where $h_{i\bar{j}}$ is the component of Hermitian matrix and $C$ is a uniform constant.

In this subsection, we aim to show that there exists a convex K\"ahler potential of $\omega$ on $N$. To achieve this, we will prove that $\sup_{N}|D^2 h| \ll 1$. To state the theorem, we recall the weighted Hölder norm introduced in Section~\ref{secwholdernorm}. In this subsection, we introduce a simplified notation to express the weighted Hölder norm on the domain $N'$. Precisely, for a function $f$ defined on the domain $N'$, we define
\begin{align*}
    ||f||'_{k,\alpha; N'} = \sum_{|\beta|=0}^k R^{|\beta|} \sup_{N}|D^\beta f|_{g_0} + R^{k+\alpha} [f]_{k,\alpha; N},
\end{align*}
where the weight is given by $R= \dist (N, \partial N')$.

\begin{pro} \label{prosmallpotentialesti}
    Let $\omega'$ be a closed real $(1,1)$-form on $N$ defined in \eqref{decaycondKform}. Then there exists a potential function $h$ so that $i\partial\overline{\partial} h = \omega'$ on $N'$. Moreover, the K\"ahler potential $h$ satisfies the following estimates:
    \begin{align}\label{smallpotentialesti}
        ||h||'_{{k+2,\alpha}; N} \leq C R^2 \sum_{i,j} ||h_{i\bar{j}}||'_{k,\alpha; N'},
    \end{align}
    where $C$ is a uniform constant and $R= \dist (N, \partial N')$.
\end{pro}

\begin{proof}
The proof of Proposition \ref{prosmallpotentialesti} is based on the classic Schauder estimates along with the H\"older estimates for a solution to $\overline{\partial}$-equation, Proposition \ref{prodbaresti}. Let $\widetilde{N}$ be the domain $\overline{N} \subseteq \widetilde{N} \subseteq \overline{\widetilde{N}} \subseteq N'$. Specially, we choose $\widetilde{N} = B_{\frac{3}{2} R}$.  Let $h$ be a real potential function of $\omega'$.
By taking the trace of $i\partial\overline{\partial} h = \omega'$, we have  $ \Delta h = \sum_{i} h_{i\bar{i}}$. Then, the classic Schauder estimates implies that
\begin{align} \label{smallpotentialschauderesti}
    ||h||'_{k+2, \alpha; N} \leq C \Big( R^2\sum_i ||h_{i\bar{i}}||'_{k,\alpha; \widetilde{N}} + ||h||_{L^{\infty}; \widetilde{N}} \big).
\end{align}
It suffices to show that there exists a real potential $h$ such that 
\begin{align} \label{estC0potential}
        \sup_{\widetilde{N}}|h| \leq C R^2 \sum_{i,j}\sup_{N'} |h_{i\overline{j}}|,
\end{align}
where $C$ is a uniform constant.
    
Recall that $\omega' = i h_{i\bar{j}} dz^i \wedge d\bar{z}^j$ is closed on $N'$, $d \omega' =0$. By the Poincar\'{e} lemma, $\omega'$ is exact on $N'$. Specially, there exists a real $1$-form $v$ so that $dv = w'$, and $v$ is given by:
\begin{align*}
    v = -i \bigg(\int_0^t h_{i\bar{j}} (tz) t dt\bigg) \bar{z}^j d z^i + i \bigg(\int_0^t h_{i\bar{j}} (tz) t dt\bigg) z^i d\bar{z}^j
 \end{align*}
 Let $v^{0,1} $ be the $(0,1)$ part of $v$, and write 
 \begin{align} \label{pflemLinftypotentialformulav}
  v^{0,1} = i\sum_{i} v_{\bar{i}} d\bar{z}^i, \qquad  v_{\bar{i}}= \sum_j z^j \int_0^t h_{j\bar{i}} (tz) t dt. 
 \end{align}
 Since $\omega'$ is $(1,1)$ form, we have $\overline{\partial } v^{0,1} =0$. Then, by Proposition~\ref{prodbaresti}, there exists a function $h''$ on $\widetilde{N}$ such that $i\overline{\partial}{h''} = v^{0,1}$. 
 Applying the estimate \eqref{Linftyesti} in Proposition~\ref{prodbaresti} and using standard scaling arguments, we obtain
 \begin{align*}
    \sup_{\widetilde{N}}|h''| \leq C  R  \sum_{i} \sup_{N;} |v_{\bar{i}}|.
 \end{align*}
From the explicit construction in \eqref{pflemLinftypotentialformulav}, we deduce the $L^\infty$ estimate:
\[ \sup_{\widetilde{N}}|h''| \leq C R^2 \sum_{i,j}\sup_{N'} |h_{i\overline{j}}|.\]
Let $h' = \bar{h''}$. Then, $h = h' + h''$ is a real potential function of $\omega'$, $i \partial \overline{\partial} h = \omega'$. The $L^\infty$ estimate \eqref{estC0potential} follows directly from the above $L^\infty$ estimate for $h''$.
 
Finally, the estimate \eqref{smallpotentialesti} follows directly from the Schauder estimate \eqref{smallpotentialschauderesti}, combined with the $L^\infty$ estimate \ref{estC0potential}, completing the proof.  
\end{proof}

In the second case where $N$ and $N'$ can be identified with $(\CC^n\backslash B_{2R}) /\Gamma$ and $(\CC^n\backslash B_{R})/ \Gamma$ ($n\geq 3$) respectively, we have the following ddbar lemma:

\begin{pro} \label{prosmallpotentialdecay}
    Let $\omega' \in \C^{k,\alpha}_{-\tau}$ ($-\tau <0$) be a closed real $(1,1)$-form on $N'$ defined in \eqref{decaycondKform}. Then, there is a smooth K\"ahler potential $h$ in $\C^{k+2, \alpha}_{2-\tau}$ so that $i\partial \overline{\partial} h = \omega'$.
\end{pro}

\begin{proof} We begin by considering the case where $\Gamma$ is trivial.

The existence of the smooth K\"ahler potential on $N'$ is based on the fact that $H^{1} (N', \OO_{N'}) =0$; please see Conlon-Hein~\cite[Proposition A2]{conlon2013asymptotically}. 
According to \cite[Proof of Theorem 3.11]{conlon2013asymptotically}, there exists a real $1$-form $v \in \C^{k+1, \alpha}_{1-\tau}$ so that $dv = \omega'$.
Let $v^{0,1}$ be the $(0,1)$ part of $v$. Since $H^{1}(N', \OO_{N'}) =0$, the condition $\overline{\partial } v^{0,1} = 0$ implies that there exists a smooth function $ \hat{h}'' $ in $N'$ so that $\overline{\partial } \hat{h}'' = v^{0,1}$. Hence, we have $ i\partial \overline{\partial} \hat{h} = \omega'$, where $\hat{h} = \hat{h}'' + \overline{\hat{h}''}$. It suffices to find a new potential $h$ with growth control at infinity.

Let $\chi$ be a smooth cutoff function in $\CC^n$ so that $\chi(x) \equiv 1$ on $N$ and has a compact support in $N'$. Then, $\hat{\omega}' = i \partial \overline{\partial} \hat{h}$ is a $d$-exact real $(1,1)$-form defined on the whole space $\CC^n$, and $\hat{\omega}' = \omega'$ on $N'$. Noting that $\hat{\omega}' \in \C^{k,\alpha}_{-\tau}$, according to the weighted ddbar lemma on $\CC^n$ in Yao~\cite[Theorem A]{yao2022mass}, there exists a smooth function $ h $ so that $ i\partial \overline{\partial} h = \hat{\omega}'$, with  $h\in \C^{k+2,\alpha}_{2-\tau}$. Restricting to $N'$, we have $i\partial\overline{\partial } h = \omega$ in $N$.

If $\Gamma$ is any finite subgroup of $U(n)$ acting freely on the unit sphere, there is a universal covering of $N$, $p: \CC^n\backslash B_{2R} \rightarrow N$. Let $\tilde{\omega}'$ be the pull-back of $\omega'$ on the universal covering. We can find the K\"ahler potential $\tilde{h}\in \C^{k+2,\alpha}_{2-\tau}$ of $\tilde{\omega}'$ on $\CC^n\backslash B_{2R}$. Define 
$$\displaystyle h (z) = \frac{1}{|\Gamma|}  \sum_{\gamma\in \Gamma}\gamma^*\tilde{h}(\tilde{z}),$$
where $z\in N$ and $\tilde{z}$ is one of the lift points of $z$ on universal covering. Then, $h \in \C^{k+2,\alpha}_{2-\tau}$ and $i\partial \overline{\partial} h = \omega'$, which complete the proof.

\end{proof}

In conclusion, we have the following result:

\begin{pro} \label{proconvexpotentialN}
Let $\omega$ be the K\"ahler form corresponding to the ALE K\"ahler metric $g$ on $X$. 
\begin{enumerate}
\item[(i)]
For complex dimension $n\geq 2$, if $N$ is a holomorphic ball in the end $X_\infty$, there is convex K\"ahler potential $\rho$, with $ i \partial \overline{\partial}\rho = \omega$ in $N$.
\item[(ii)] 
For complex dimension $n\geq 3$, if $N$ is the complement of a compact set in $X$, there is convex K\"ahler potential $\rho$, with $ i \partial \overline{\partial} \rho = \omega$ in $N$.
\end{enumerate}
For both of the cases, we have $\displaystyle D^2 \rho \geq \frac{1}{2} I_{2n}$.
\end{pro}

\begin{proof}
    The statement (i) follows directly from Proposition \ref{prosmallpotentialesti}, while the statement (ii) follows directly from Proposition \ref{prosmallpotentialdecay}.
\end{proof}

\subsection{Construction of $\Omega$-psh subsolutions} \label{secconstructF}
In this subsection, we construct a $\Omega$-psh subsolution to the HCMA equation with the boundary data, $\Psi$, and complete the proof of Theorem \ref{thmconstpshsolu}. The process of construction relies on the following two points:
\begin{itemize}
    \item the existence of a holomorphic disc foliation on $X_{l} \times D$, where $c>0$ is a sufficiently large constant;
    \item the existence of a convex K\"ahler potential for the reference K\"ahler form $\omega$ in holomorphic coordinate charts covering $X_l$, or in the holomorphic asymptotic chart of $X_l \subseteq X_\infty$.
\end{itemize}

We adopt the notations, $\rho$, $G$, $\Lambda_\tau$, $z_p$, $\xi_p$, $(N''\subseteq N\subseteq N') $ and the holomorphic cotangent bundle $E\rightarrow N'$ from the above sections. In the complex dimension $n=2$, we may need to consider a countable family of triples of holomorphic coordinate balls $(N_i'', N_i, N_i'')$, indexed by $i \in \mathcal{I}$, and apply the construction developed at the end of Subsection \ref{subsecpatthm}. For the sake of simplicity, when focusing on local construction, we will use the shorthand $(N'', N, N')$ to refer to a generic such triple.
Let $\tilde{\psi}_\tau = \rho + \psi_\tau$ be a K\"ahler potential of the K\"ahler form, $\omega_\tau$ on $N' \times \{\tau\}$, for $\tau \in \partial D$. We can choose the convex K\"ahler potentials $\tilde{\psi}_\tau$ such that the Hessian is strictly convex with $D^2 \tilde{\psi}_\tau \geq (1/4) I_{2n}$ on $N'$, for each $\tau \in \partial D$. The corresponding Lagrangian submanifold $\Lambda_\tau$ is locally given by the graph of $\partial \rho_\tau$ in $E$.
We briefly recall that the holomorphic disc foliation on $X_l \times D$, constructed in Subsections~\ref{secholodiscHCMA} and~\ref{subsecpatthm}, is described by a smooth family of holomorphic maps $G(x, \cdot)= g_x (\cdot) : D \rightarrow \mathcal{W}$, for $x \in X_l$, satisfying the free boundary condition, $g_x(\tau) \in \Lambda_\tau$ for $\tau \in \partial D$. In local holomorphic coordinates on $E$, this takes the form $G(w, \tau) = (z(w, \tau), \xi(w, \tau))$, as in~\eqref{holodiscpde}.

We will construct a family of real functions $\{L_{x_0}\}$ in $N' \times D$, for each $x_0 \in N$. Along the leaf of the holomorphic disk at $x_0 \in N$, we define $L_{x_0}$ to be the harmonic function with respect to $\tau$, such that $L_{x_0}$ agrees with $\tilde{\psi}_\tau$ on the boundary of the leaf at $x$. 
In the spatial direction of $N' \times D$, $L_{x_0}$ is defined to be linear such that the derivative of $L_{x_0}$ in $z_p$ direction agrees with $\xi_p (x_0, 
\tau)$, for $\tau \in D$, where $\xi_p $ is the bundle coordinate of the holomphic dick $G(x_0, \tau)$. More precisely, $L_{x_0}$ satisfies the following equations:
\begin{align}
    \frac{\partial}{\partial \tau} \frac{\partial}{\partial\bar{\tau}} \big\{ L_{x_0} (z(x_0, \tau), \tau) \big\}  = 0, \ \ \qquad \qquad & \tau \in D; \label{pshconst1} \\
    L_{x_0} (z(x_0, \tau), \tau) = \rho_\tau (z(x_0, \tau)), \qquad \qquad   &\tau \in \partial D; \label{pshconst2}\\
    \big( \partial_{z_p} L_{x_0}\big) (z(x_0,\tau), \tau) =  \xi_p (x_0, \tau), \quad \ \qquad  &(z, \tau) \in (X-B_R) \times D. \label{pshconst3}
\end{align}
In the following lemma, we will check that $L_{x}$ is pluriharmonic in $N' \times D$ for each $x \in N$. 

\begin{lem} \label{lempshLcheck} Let $L_{x_0}$ be the function defined as in (\ref{pshconst1})-(\ref{pshconst3}). Then, $L_{x_0}$ is a smooth pluriharmonic function in $N'\times D$.
\end{lem} 
\begin{proof} We can write down $L_{x_0}$ explicitly in the asymptotic coordinates as a family of linear functions varying with respect to $\tau$, 
\begin{align} \label{explicitL}
    L_{x_0} (z, \tau) = 2 \sum_{p=1}^n \re \big\{ \xi_p(x_0,\tau) (z_p - z_p(x_0, \tau)) \big\} + L_{x_0} (z(x_0, \tau), \tau), \quad (z , \tau) \in N' \times D.
\end{align}
 By taking derivatives for $L_{x_0} (z,\tau)$, we have
 \begin{align}
         (\partial_{z_p} L_{x_0})(z, \tau) = & \ \xi_p (x_0, \tau),\label{1stderivofLspace}\\
         \begin{split} \label{1stderivofLtime}
         (\partial_\tau L_{x_0}) (z, \tau) = & \sum_{p =1}^n \partial_{\tau} \xi_p(x_0, \tau) (z_p - z_p(x_0, \tau))  -  \sum_{p=1}^n  \xi_p (x_0,\tau) \frac{\partial z_p}{\partial \tau}  (x_0, \tau)   \\& + \frac{\partial}{\partial \tau} \big\{ L_{x_0} (f(x_0, \tau), \tau) \big\}. 
         \end{split}
 \end{align}
 The smoothness of $L_{x_0}$ is obvious.
By (\ref{1stderivofLspace}), and noting that $\xi_p$ is holomorphic with respect to $\tau$, we can compute the second derivatives of $L_{x_0}$,
\begin{align*}
    (\overline{\partial}_{z_q}\partial_{z_p}L_{x_0}) (z, \tau)  &=0,\\
    (\overline{\partial}_{\tau} \partial_{z_p}L_{x_0}) (z, \tau) & = 0.
\end{align*}
It suffice to check $\overline{\partial}_\tau \partial_\tau L_{x_0} =0$. By taking the derivative of (\ref{1stderivofLtime}), we have 
\begin{align}
\begin{split}\label{pshttleaf}
    \big(\overline{\partial}_\tau \partial_\tau L_{x_0}\big)(z, \tau) 
    =& \sum_{p=1}^n \overline{\partial}_\tau \partial_{\tau} \xi_p(x_0,\tau) (z_i - z_i(x_0, \tau))  - \sum_{p=1}^n \partial_\tau \xi_p(x_0,\tau) \frac{\partial z_p}{\partial \overline{\tau}}  (x_0, \tau)
    \\& -  \sum_{p=1}^n \frac{\partial}{\partial \overline{\tau}} \big\{ \xi_p (x_0, \tau)  \frac{\partial z_p}{\partial \tau} (x_0, \tau)  \big\} 
    + \frac{\partial}{\partial \overline{\tau}}\frac{\partial}{\partial \tau} \big\{ L_{x_0} (z(x_0, \tau), \tau) \big\}. 
    \\ = & \frac{\partial}{\partial \overline{\tau}}\frac{\partial}{\partial \tau} \big\{ L_{x_0} (z(x_0, \tau), \tau) \big\} =0.
\end{split}
\end{align}
Hence, we complete the proof that $L_{x_0}$ is pluriharmonic.
\end{proof}

Let $M$ be the uniform $\C^{1,1}$ bound of the solution of HCMA equation in Theorem \ref{thmc11soluasupenvelope} and let $\omega$ be the reference K\"ahler form defined on $X$. By restricting $\omega$ to $N'$, $\omega$ admits a convex K\"ahler potential, $\rho$ satisfying $D^2  \rho \geq (1/2) I_{2n} $ due to Proposition \ref{proconvexpotentialN}. The prescribed K\"ahler form on boundary $X \times \{\tau\}$, $\tau \in \partial D$, is given by $\omega_\tau =  \omega + i\partial \overline{\partial}  \psi_\tau$, where $\psi_\tau$ is a prescribed smooth K\"ahler potential in $\C^{\infty}_{\gamma} (X)$ $(\gamma <0)$. Let $\widetilde{N}$ be the median domain between $N$ and $N'$; for instance, we take $\widetilde{N} = B_{3R/2} $ for the first case, and $\widetilde{N} = (\CC^n- B_{3R/2})$ for the second case. Then, we have the following properties of $L_{x_0}$:
\begin{lem} \label{lemsubsolucon}
    Fixing a triple $N\subseteq \widetilde{N} \subseteq N'$ as above, for each $x_0 \in N$, we have the following properties:  
    \begin{align}
        L_{x_0} (z, \tau) - \rho(z) &\leq \psi_\tau (z), \qquad (z, \tau) \in N' \times \partial D,\label{subsoluconbdry}\\
        L_{x_0} (z, \tau) - \rho(z) &\leq -M, \hspace{1cm} (z, \tau) \in \big(N'-\widetilde{N}\big) \times D. \label{subsoluconcpt}
    \end{align}
\end{lem}
\begin{proof} The lemma can be easily derived from the convexity of $\rho$.

    Let $\mathcal{L}_{x_0}$ be the leaf at $x_0 \in N$ given by $\mathcal{L}_{x_0} = \{H(x_0, \tau); \tau \in D\} $.  
    According to the construction of $L_x(z,\tau)$, (\ref{pshconst1})-(\ref{pshconst3}), for $(z,\tau) \in \partial \mathcal{L}_{x_0}$,
    \begin{align*}
    \begin{split} 
        (L_{x_0} - \rho -\psi_{\tau}) (z, \tau) = 0, \quad\\
        \partial_{z_p}  (L_{x_0} - \rho -\psi_{\tau}) (z, \tau) = 0.
    \end{split}
    \end{align*}
    Let $D_X$ be the covariant derivative in the space direction.  Using the linearity of $L_x$ and the convexity of K\"ahler potential in $N'$, we have
    \begin{align*}
        D_X^2\big\{ L_{x} (z, \tau)-\rho(z) - \psi_\tau(z)\big\} \leq -\frac{1}{3} I_{2n}, \qquad  (z, \tau) \in N' \times \partial D.
    \end{align*}
    Then, we have
     \begin{align*} 
         L_x(z, \tau) - \rho(z) -\psi_\tau(z) \leq  -\frac{1}{6} d_0(z(x, \tau), z)^2, \qquad  (z, \tau) \in (X-B_R) \times \partial D
     \end{align*}
     where $d_0(x,z)$ represents the Euclidean distance between $z$ and $x$ in the standard complex coordinates of $N'$. The inequality (\ref{subsoluconbdry}) is proved.

    Choose a large constant $R$ so that $R \geq 5\sqrt{M}$, where $M$ is the $\C^{1,1}$ bound of the solution of HCMA. For any $x_0 \in N$, we have
    \begin{align*}
        L_{x_0} (z, \tau) - \rho(z) \leq -M, \qquad (z, \tau) \in \big( N'- \widetilde{N}\big) \times \partial D.
    \end{align*}
    To get the interior control, according to (\ref{pshttleaf}), we observe that $L_{x_0} (z, \tau) -\rho(z)$ is harmonic on each slice $\{z\} \times D$. The maximal principle implies (\ref{subsoluconcpt}).
\end{proof}

Now, we have all the ingredients to construct the global $\Omega$-plurisubharmonic function. 

\subsubsection*{Case 1: $\dim_\CC X = 2$} According to Corollary \ref{corcxasymptoticcovering}, there exists a locally finite and countably infinite open covering $\{N_i;\ i\in \mathcal{I} \}$ of $X_{2L}$, where each $N_i$ is biholomorhic to $B_R$. Moreover, for each $i\in \mathcal{I}$, we can find a triple $(N_i, \widetilde{N}_i, N'_i)$ such that there sets are biholomorphic to concentric balls $(B_R, B_{\frac{3}{2} R}, B_{2R})$ respectively. For each point $x_0 \in N_i$, we can construct the function $L_{x_0, i}$ on $N'$ based on (\ref{pshconst1})-(\ref{pshconst3}). Then, we can construct a $\Omega$-plurisubharmonic function $F_i(z, \tau)$ as follows:
\begin{align*}
    F_i (z, \tau) = \sup_{x \in N_i} \big(L_{x, i} (z, \tau)- \rho_i (z)\big), \qquad (z, \tau ) \in N' \times D,
\end{align*}
where $\rho_i$ is a convex K\"ahler potential of $\omega$ in $N_i'$ satisfies the condition. 
\begin{lem}\label{lemconstructpshlocal}
    $F_{i} (z,\tau)$ is a continuous $\Omega$-plurisubharmonic function defined in $N' \times D$. Furthermore, if we restrict $F_{i} (z, \tau)$ on the leaf of the holomorphic disc at $x \in N_i$, denoted by $(z(x, \tau), \tau)$, then
    \begin{align*}
        F_{i}(z(x, \tau), \tau) = L_{x, i} (z(x, \tau), \tau) - \rho_{i} (z(x, \tau)).
    \end{align*}
\end{lem}
\begin{proof}
    For the first statement, it suffices to prove that $F_{i} (z, \tau)$ is continuous in $N'$. By Theorem \ref{thmexistbypertb}, we have $|\xi_p(x, \tau)| \leq M_i$ for $x \in N_i$, where $M_i$ is a constant depending on $N_i$. Using formula \eqref{1stderivofLspace}, this implies $|\partial_{z_p} L_{x. i} (z, \tau)| \leq M_i$, for all $(z, \tau) \in N' \times D$. Similarly, from formula \eqref{1stderivofLtime}, we see that $|\partial_\tau L_{x, i}|$ is also bounded by a constant depending on $N_i$ and the diameter of $N'_i$.  These facts imply that the family of smooth pluriharmonic functions $\{L_{x,i} \in \text{Ph}(N'_i\times D);\ x\in N_i \}$ is equicontinuous. Therefore, $F(z, \tau) = \sup_{x \in N_i} L_x (z, \tau)$ is continuous.

    For the second statement, we need to show the following: by fixing $x_0 \in N_i$, 
    \begin{align*}
        \sup_{x \in N_i} {L_{x,i} (z(x_0, \tau), \tau)} = L_{x_0, i} (z(x_0, \tau), \tau).
    \end{align*}
    When restricted to the leaf of the holomorphic disc at $x_0$, $(z(x_0, \tau), \tau)$, $L_{x,i}(z(x_0, \tau), \tau)$ is harmonic in terms of $\tau$ for $\tau \in D$. By Lemma \ref{lemsubsolucon} and (\ref{pshconst2}), we have
    \begin{align*}
        L_{x, i} (z(x_0, \tau), \tau) \leq (\rho_i+ \psi_\tau)(z(x_0, \tau)) = L_{x_0, i} (z(x_0, \tau), \tau),
    \end{align*}
    for $\tau \in \partial D$ and for all $x \in N_i $.
    According to the classic maximal principle for harmonic functions, we have $L_{x, i} (z(x_0, \tau), \tau) \leq L_{x_0, i} (z(x_0, \tau), \tau)$, for $\tau \in D$, completing the proof.
\end{proof}

The patching theorem of the foliation by holomorphic discs, Theorem \ref{thmpatch}, ensures that the leaves of holomorphic discs coincide in the overlap $N_i \bigcap N_j$. Now, we discuss the relationship between $F_i$ and $F_j$ in the overlapping region of $N'_i$ and $N'_j$.

\begin{lem} \label{lempshpatching}
    Let $x \in N_i$. For any other index $j \in \mathcal{I}$ such that $x \in N'_j$ we have 
    \begin{align*}
        F_i (z(x,\tau), \tau) \geq F_j (z(x, \tau), \tau).
    \end{align*}
    Furthermore, if $ x \in N_j $, then $F_i (z(x,\tau), \tau) = F_j (z(x, \tau), \tau).$
\end{lem}
\begin{proof}
    We first prove the case when $x\in N_j$. By Lemma \ref{lemconstructpshlocal}, we have $F_{i} (z(x, \tau), \tau) = L_{x, i} (z(x,\tau), \tau) - \rho_i (z(x, \tau)) $. Consider the function $F_i(z, \tau) + \rho_j(z)$ defined on $(N'_i\bigcap N'_j ) \times D$. Restricting to the leaf of the holomorphic disc at $x$, we have 
    \begin{align*}
        \partial_\tau \overline{\partial}_\tau \big\{ F_i (z(x, \tau), \tau) + \rho_j(z(x, \tau)) \big\}
        = &\partial_\tau \overline{\partial}_\tau  \big\{ L_{x, i} (z(x, \tau), \tau)  \big\} + \partial_{z_p} \overline{\partial}_{z_{q}} (\rho_{j} - \rho_{i}) \frac{\partial z_{p}}{\partial \tau} \frac{\partial \bar{z}_{q}}{\partial \bar{\tau}}.
    \end{align*}
    The first term on the right-hand side of the above equality is vanishing due to the construction of $L_{x, i}$, and so is the second term due to $i\partial \overline{\partial} \rho_i = i\partial \overline{\partial} \rho_j$ on $N'_i \bigcap N'_j$. Hence, $F_{i} + \rho_j $ is harmonic on the leaf of the holomorphic disc at $x$. On the boundary of the leaf at $x$, we observe that
    \begin{align*}
        F_i (z(x, \tau), \tau) + \rho_j(z(x, \tau)) = L_{x, j} (z(x, \tau), \tau), \qquad \tau \in \partial D.
    \end{align*}
    Hence, we have $F_i = L_{x,j} - \rho_{j} = F_j$ on the leaf at $x \in N_i \bigcap N_j$.

    Consider the case when $x \in N'_j$. Without loss of generality, we assume that the leaf of the holomorphic disk at $x$ is contained in $N'$. The function defined as above, $F_i (z(x, \tau), \tau) + \rho_j(z(x, \tau))$ is also harmonic in terms of $\tau$  with $F_{i} (z(x, \tau), \tau) + \rho_{j}(z(x, \tau)) = (\rho_j + \psi_\tau) (z(x, \tau)) $, for $\tau \in \partial D$. By Lemma \ref{lempshLcheck} and \ref{lemsubsolucon}, we have for each $y \in N_j$, 
    \begin{align*}
        F_{i} (z(x, \tau), \tau) + \rho_{j}(z(x, \tau)) \geq L_{y,j} (z(x, \tau), \tau).
    \end{align*}
    Therefore, $F_i(z(x, \tau), \tau) \geq F_{j}(z(x, \tau), \tau)$, which completes the proof.
\end{proof}

According to Lemma \ref{lemconstructpshlocal} and \ref{lempshpatching}, we construct a global $\Omega$-plurisubharmonic function as follows:
\begin{align}\label{constructdim2}
     F(z, \tau) = \max \big\{\sup_{i \in  \mathcal{I}} F_i (z, \tau), -M\big\}, \qquad (z, \tau ) \in X \times D.
\end{align}
Suppose that the family of holomorphic balls $\{U_i;~ i \in \mathcal{I}\}$ covers the region $X_l$ for some large constant $l$. Theorem \ref{thmexistbypertb} implies that there exists a foliation by a smooth family of holomorphic discs $G_i: N_i \times D \rightarrow E $ for each $N_i$. If we denote $H_i (w, \tau) = \pi \circ G_i(w, \tau)$, by Theorem \ref{thmpatch}, $H_i (w, \tau) = H_j(w, \tau)$ for $ w\in U_i \bigcap U_j$. Hence, there exists a foliation by holomorphic discs on $X_l$, denoted by $H(w, \tau)$ such that $H|_{U_i} = H_i$. In local coordinates, we write $H (w, \tau) = z(w, \tau) = (z_1,\ldots, z_n )(w, \tau)$. Fixing a point $x \in X_l$, the set, $H_x = \{(H(x, \tau), \tau); \tau \in D\}$, is called the leaf of the holomorphic disc at $x$.

\begin{proof}[Proof of Theorem \ref{thmconstpshsolu}: general cases]
The continuity of $F(z, \tau)$ follows directly from Lemmas \ref{lemsubsolucon}, \ref{lemconstructpshlocal} and \ref{lempshpatching}. Hence, $F(z,\tau)$ is a continuous plurisubharmonic function on $X \times D$. It suffices to check the boundness of $F(z, \tau)$ on $X \times D$. 

It is obvious that $F(z, \tau)$ is bounded below. To show $F(z, \tau)$ is bounded above, observe that $F(z, \tau) + \rho(z)$ is a global plurisubharmonic function. By restricting to each slice $\{z\} \times D$, we see that $F(z, \tau)$ is subharmonic in $\tau$. From Lemma \ref{lemsubsolucon}, we have $F(z, \tau) \leq \psi_\tau (z)$ for $\tau \in \partial D$. Hence, 
\begin{align*}
    F(z, \tau) \leq \sup_{X \times \partial D} |\psi_\tau (z)|.
\end{align*}

In this paragraph, we complete the proof of part $(b)$. By Theorem \ref{thmc11soluasupenvelope}, the solution $\Phi$ is the upper envelope of $\F_{\Omega, \Psi}$, the class of bounded, continuous $\Omega$-plurisubharmonic functions, bounded above by $\Psi$ on the boundary $X \times \partial D$. Therefore, we have 
\begin{align*}
F(z, \tau) \leq \Phi (z, \tau), \qquad (z, \tau) \in X \times D
\end{align*}
If we restrict $F(z, \tau)$ on the leaf $H_x$ for $x \in U_i$, then by Lemma \ref{lemconstructpshlocal}, $F(z(x, \tau), \tau) = L_{x, i} (z(x,\tau), \tau) - \rho_i (z(x, \tau))$. By the construction of $L_{x,i}$ in \eqref{pshconst1}-\eqref{pshconst3}, we have $F(x, \tau) = \psi_\tau (z(x, \tau))$ for $\tau \in \partial D$. Also, notice that $\Phi + \rho_i$ is plurisubharmonic in $U_i \times D$. Then, we have
\begin{align*}
\Phi(z(x, \tau), \tau) \leq F(z(x, \tau), \tau), \qquad \tau \in D.
\end{align*}
Since $\{U_i, i \in \mathcal{I}\}$ covers $X_l$, we can assume for each point $(z, \tau) \in X_l \times D$, there exists a leaf $H_x$ passing through $(z, \tau)$. Therefore $ F(z, \tau) = \Phi(z, \tau)$, if $(z, \tau )\in X_l \times D.$
\end{proof}

\subsubsection{Case 2: $\dim_{\CC} X \geq 3$} According to Proposition \ref{propcxasymcoordinates}, there exists a triple $(N, \widetilde{N}, N')$ such that the sets are biholomorphic to $(\CC^n- B_{l}) /\Gamma $, $(\CC^n- B_{\frac{3}{4}l}) /\Gamma $ and $(\CC^n- B_{\frac{1}{2}l}) /\Gamma $ respectively. Let $\rho$ be the convex K\"ahler potential of $\omega $ on $N'$ and let $L_{x}(z, \tau)$ be the pluriharmonic function satisfying (\ref{pshconst1})-(\ref{pshconst3}) defined on $N' \times D$ for each $x \in N$. Then, we construct
\begin{align*}
    F_{\infty} = \sup_{x \in N} \big( L_{x} (z, \tau)- \rho (z) \big), \qquad (z, \tau) \in N' \times D.
\end{align*}
We can now construct a global $\Omega$-plurisubharmonic function:
\begin{align*}
    F(z, \tau) = \max \{ F_{\infty} (z, \tau), -M\}, \qquad X\times D.
\end{align*}
Similar to Case 1, the function $F(z, \tau)$ satisfies the following property:

\begin{proof}[Proof of Theorem \ref{thmconstpshsolu}: holomorphic asymptotic coordinates]
The proof of boundedness is the same as in the proof in the general case.

It suffices to show that $F(z, \tau)$ is continuous for complex dimension $n\geq 3$.
Using the same method as in the proof of Lemma \ref{lemconstructpshlocal}, by restricting to the leaf $H_x$, we can show
\begin{align*}
    F(z(x, \tau), \tau) = L_{x} (z(x, \tau), \tau) - \rho(z(x, \tau)).
\end{align*}
By Proposition \ref{propexplicitconstructionsolu}, $F(z, \tau)$ is smooth on $X_{L+1} \times D$. It suffices to show the continuity on the region $B_{2L}$. In the proof of Lemma \ref{lemsubsolucon}, we have 
\begin{align*}
    L_x (z, \tau) - \rho(z) \leq -\frac{1}{2} (d(x, z))^2.
\end{align*}
For all $z \in B_{2l}$, if we pick $x \in X_{l'}$ ($l' = 2l + 2 \sqrt{M}$), we have $d(x, z) \geq 2 \sqrt{M}$ and 
\begin{align*}
    L_{x}(z, \tau) - \rho(z) \leq -2M.
\end{align*}
That is to say, $ L_x (z, \tau) -\rho(z)$ does not contribute the the value of $F(z, \tau)$ if $r(x) \geq l'= 2l + 2 \sqrt{M}$. Therefore,
\begin{align*}
    F(z,\tau) =\Big\{ \sup_{x \in B_{L'} \bigcap N } \big( L_x (z, \tau) - \rho(z) \big), -M\Big\}.
\end{align*}
It is easy to check that the family of smooth pluriharmonic functions $\{L_{x} \in \C^{\infty} (N'\times D); x\in N\bigcap B_{l'} \}$ is equicontinuous, and $L_{x} (z, \tau) \leq -M$ for $(z, \tau) \in (N' - \widetilde{N}) \times D $. Hence, $F(z, \tau)$ is continuous and $\Omega$-psh on $B_{l'}$. 

The proof of parts (ii) and (iii) follows directly from the argument given in the , as the conditions and underlying structure are equivalent.
\end{proof}

\section{The asymptotic behavior of the solution to HCMA equations} \label{secweiesti}

This section is dedicated to the weighted estimates of the solution to the following HCMA equation:
\begin{align}
    &(\Omega + dd^c \Phi)^{n+1} = 0, \qquad \text{ in } X \times D, \nonumber \\
    &\Phi = \Psi, \hspace{2.85cm} \text{ in } X \times \partial D, \label{HCMAforweightedesti} \\
    &\Omega + dd^c \Phi \geq 0, \hspace{1.68cm} \text{ in } X \times D. \nonumber
\end{align}
Let $\Phi$ be the bounded continuous solution to \eqref{HCMAforweightedesti} given in Theorem \ref{thmgsolHCMAupenv}. Additionally, assume $\Psi \in \C^{k, \alpha}_{-\gamma} (X \times \partial D)$. 
According to the discussion in Subsection \ref{subsecC11esti}, the optimal global regularity for $\Phi$ is $\C^{1,1} $, though an improved regularity is expected near infinity on the set $X_{l} = \{x \in X; r(x) \geq l\}$ for a sufficiently large constant $c$. 
In this section, we derive weighted estimates that control the higher-order regularity of $\Phi$ on the asymptotic region $X_l$.
By Theorem \ref{thmweightestifoli}, there exists a foliation of $X_l \times D$ by holomorphic discs, with weighted estimates on coordinates shifts as given in \eqref{estiweiholodisccoordishift}. Let $ X_{l'}$ be a slightly smaller open subset of $X_l$, $X_{l'} = \{x \in X; r(x) > l'\}$ with $l'>l$, such that $H(X_{l} \times D) \subseteq X_{l'}$.

In Subsection \ref{secconstructF}, we construct a bounded, continuous $\Omega$-purisubharmonic subsolution $F$ to HCMA (\ref{HCMAforweightedesti}). 
By Theorem \ref{thmconstpshsolu}, we have
\begin{align} \label{idsoluXL}
    \Phi \equiv F, \qquad \text{ on } X_{l'} \times D.
\end{align}
The main theorem of this section is the following:

\begin{thm} \label{thmweiestisolutiontoHCMA}
Let $\Phi$ be the $\C^{1,1}$ solution to HCMA equation \eqref{HCMAforweightedesti}. If the boundary function $\Psi $ in \eqref{HCMAforweightedesti} belongs to $ \C^{k+3,\alpha}_{-\gamma} (X \times \partial D)$, then there exists a uniform constant $C$, and a large uniform constant $l'$ in \eqref{idsoluXL} such that,
\begin{align} \label{estiweisoluHCMA}
    ||\Phi||_{-\gamma; k,\beta; X_{l'} \times D} \leq C ||\Psi||_{-\gamma; k+1,\alpha; X_{\infty} \times D},
\end{align}
where 
Furthermore, the form $\Omega + \tilde{d}\tilde{d}^c \Phi$ is nondegenerate in the direction of $X$ on $X_{l'} \times D$ such that the following holds:
\begin{align*}
   \frac{1}{C} \omega  \leq \omega + i \partial \overline{\partial} \varphi_\tau \leq C \omega, \qquad \text{ on } N \times D.
\end{align*}

\end{thm}

\subsection{The weighted $\C^0$ estimates of the solution} \label{secweiestisoltoHCMAend}
Recall that there is a countable family of tripes $\{(N_i'', N_i, N_i'); N_i'' \subseteq N_i \subseteq N_i' \subseteq X_\infty, i \in \mathcal{I}  \}$ where $\{N_i\}$ and $\{N_i''\}$ form locally finite and countable open coverings of $X_l$ and $X_{l'}$, respectively. For each $i \in \mathcal{I}$, $(N_i'', N_i. N_i'')$ is biholomorphic to a triple of concentric balls $(B_{R_i-1}, B_{R_i}, B_{R_i +1})$ by the map $I_i$. The radius is given by $R_i = \delta r(x)$ with $x$, $I_i(x) =0$. The idea of proving $\C^0$ estimatse is simply by restricting the solution $\Phi$ to each $N_i''$. Then, we have the following proposition:

\begin{pro} Let $\Phi$ be the bounded $\C^{1,1}$ solution to HCMA equation \eqref{HCMAforweightedesti} and let $\Psi \in \C^{k+3,\alpha}_{-\gamma} (X \times \partial D)$ be the boundary function of \eqref{HCMAforweightedesti}. Then, $\Phi$ have the following weighted $\C^0$ estimates:
\begin{align*}
    ||\Phi||_{-\gamma;~ 0; X_{l'} \times D} \leq ||\Psi||_{-\gamma;~ 0; X_\infty \times \partial D }.
\end{align*}
\end{pro}

\begin{proof}
For simplicity, we denote the leaf of the holomorphic disk at $x$ by $\mathcal{L}_w = \{(H(w, \tau), \tau); \tau \in D\}$ and the leaf of the vertical disc at $x$ by $\mathcal{V}_z = \{(z, \tau); \tau \in D\}$. 

Given an arbitrary point $(z, \tau ) \in N_i'' \times D$, there exists another $w \in N_i$ such that $(z, \tau) \in \mathcal{L}_{w}$. Let $\rho$ be a local K\"ahler potential of $\omega$.
If we restrict function $ F(z, \tau) + \rho(z)$ on $\mathcal{L}_{w}$, by Lemma \ref{lemconstructpshlocal}, $F(z(w, \tau),\tau) + \rho(z(w, \tau))$ is harmonic in $\tau$. Therefore, $F(z(w, \tau), \tau)$ is superharmonic in $\tau$, which implies that
\begin{align*}
     F(z, \tau) =  F(z(w, \tau), \tau) \geq \inf_{\tau \in \partial D} \Psi(z(w, \tau), \tau).
\end{align*}
Since $F(z, \tau) + \rho(z)$ is plurisubharmonic in $X \times D$, by restricting $F(z, \tau) + \rho(z)$ to the vertical slice $\mathcal{V}_z$, $F(z, \tau)$ is subharmonic in $\tau$. Then, we have
\begin{align*}
     F(z, \tau) \leq \sup_{\tau \in \partial D} \Psi(z, \tau).
\end{align*}
Therefore, for each $(w,\tau) \in N_i \times D$,   
\begin{align*}
    \sup_{z\in N_i''}r(x_0)^\gamma |F(z, \tau)| \leq \sup_{z \in N_i'} {r(x_0)^\gamma}  |\Psi (z, \tau)|.
\end{align*}
where $x_0$ is the center of $N_i$ with $I_i (x_0) = 0$.
In conclusion, we have $||\Phi||_{-\gamma;~0; X_{l'} \times D} \leq  ||\Psi||_{-\gamma;~0; X_\infty \times D }$.
\end{proof}

\subsection{The weighted estimates of higher order regularity} In this subsection, we complete the proof of the main theorem, Theorem \ref{thmweiestisolutiontoHCMA}. In the proof that follows, we shall repeatedly employ the constants $C_0$ and $C$ in our estimates. The constant $C_0$ denotes the fixed bound arising from the weighted control of the underlying data, such as the fall-off condition of the ALE metrics and the boundary data $\Psi$. The symbol $C$ may vary from line to line, but it always refers to a uniform constant depending only on the parameters $n$, $k$, $\alpha$, $\sigma^{-1}$, $(\alpha-\beta)^{-1}$, and $C_0$. 

\begin{proof}[Proof of Theorem \ref{thmweiestisolutiontoHCMA}]
To prove the weighted estimates of higher-order regularity for $\Phi$ on $X_{l'}$, we restrict $\Phi$ to each $N_i'$. For the sake of simplicity in calculation, we work in the coordinate system given by the family of holomorphic discs, ${N_i \times D, (w, \tau)}$. We then denote $\widehat{F}(w, \tau) = F(z(w, \tau), \tau) = \Phi (z(w, \tau), \tau)$ on $N_i \times D $. The foliation by holomorphic discs gives the coordinate transform:
\begin{align*}
    N_i \times D \xrightarrow{H \times id} N_i' \times D \xrightarrow{I_i \times id} X_\infty \times D, \hspace{1cm} (w, \tau) \xrightarrow{H\times id} (z, \tau) \xrightarrow{I_i \times id} (z^0, \tau),
\end{align*}
where $z^0 = (z^0_1, \ldots, z_n^0)$ is the complex coordinates in the asymptotic chart. 
The higher-order weighted estimates for the coordinate transforms are given in \eqref{weightesticoordinatesholocball} of Lemma \ref{lemholocornearinfty} and in \eqref{estiweiholodisccoordishift} of Theorem \ref{thmweightestifoli}. For the sake of simplicity, we establish the higher-order estimates with respect to the Euclidean metric in each coordinate system. 
Providing the assumption that $||\Psi||_{-\gamma; k+3, \alpha; X_\infty \times D } \leq C_0 $, Theorem \ref{thmweightestifoli} implies that,
\[
||z-z_0||_{-\gamma-1; k,\beta; X_{l} \times D} \leq C ||\partial \Psi ||_{-\gamma-1; k,\alpha; X_\infty\times \partial D}. 
\]
The relationship between $||\widehat{F}||_{-\gamma; k,\beta; X_l \times D}$ and $||F||_{-\gamma; k,\beta; X_{l'} \times D}$ then follows from \eqref{weightesticoordinatesholocball}, and is given by:
\begin{align} \label{esticoortran}
    C^{-1}||F||_{-\gamma; k,\beta; N_i''\times D} \leq ||\widehat{F}||_{-\gamma; k, \beta; N_i \times D} \leq C ||F||_{-\gamma; k,\beta; N_i' \times D},
\end{align}
where $C>1$ is a uniform constant. 

It suffices to establish the uniform weighted estimates for $\widehat{F}$ in each $\{N_i \times D, (w, \tau)\}$. Recall that, by Proposition \ref{proconvexpotentialN}, there exists a local potential function $\rho$ on $N_i$ for the reference K\"ahler metric, satisfying
\begin{align} \label{estiKpotential}
    ||\rho(z) - |z|^2||_{2-\mu; k,\alpha; N'_i}  \leq C_0.
\end{align}
Let $\mathfrak{e} (w, \tau) = \rho(z(w, \tau)) -\rho(w) $ be an error term defined on $N_i \times D$.
On each leaf $\mathcal{L}_w$, $w\in N_i$, we have that $\widehat{F}(w, \tau) + \mathfrak{e} (w, \tau)$ is harmonic in terms of $\tau$:
\begin{align*} 
\begin{split}
    &\Delta_{\tau}  \big\{ \widehat{F}(w, \tau) + \mathfrak{e} (w, \tau)\big\} =0, \hspace{1cm} (w,\tau) \in N_i \times D;\\
    &\quad \widehat{F}(w, \tau) =  \Psi ( z(w, \tau), \tau), \hspace{1.2cm} (w,\tau) \in N_i \times \partial D.
\end{split}
\end{align*}
By the weighted estimates for the family of Dirichlet problems in Lemma \ref{lemplaestiwrtbdry}, we have
\begin{align*}
    ||F(w, \tau) + \mathfrak{e}(w, \tau)||_{-\gamma; k,\beta; N_i \times D} \leq C ||\Psi (z(w, \tau), \tau) + \mathfrak{e} (w, \tau)||_{-\gamma; k,\beta; N_i \times \partial D}.
\end{align*}

To estimate $F(w, \tau)$ on $N_i \times D$, it remains to control the H\"older norm of $\mathfrak{e}(w, \tau)$ on $N_i \times D$, $||\mathfrak{e}||_{-\gamma; k, \beta; N_i \times  D}$. Note that 
\begin{align*}
    \mathfrak{e}(w, \tau) = &\quad \big(z_p (w, \tau) - w_{p}\big) \int_0^1 \frac{\partial \rho}{\partial z_p} (z_t) dt  \\
    &+ \big(\bar{z}_p (w, \tau) - \bar{w}_{p}\big) \int_0^1 \frac{\partial {\rho}}{\partial \bar{z}_p} (z_t) dt,
\end{align*}
where $z_t (w, \tau) = t z(w,\tau) + (1-t) w $. Then, by \eqref{estiproductweiholder} and Lemma \ref{ckalphaweiesticomposfun}, it follows that
$$
||\mathfrak{e}||_{-\gamma;k,\beta; N_i \times D} \leq C||z-z_0||_{-\gamma-1; k,\beta; N_i \times D} \cdot ||D \rho||_{1; k,\beta; N_i' \times D}
$$
By the weighted estimates of $\rho$, \eqref{estiKpotential}, we have 
\[
||D\rho||_{1; k,\alpha; N'_i} \leq ||D|z|^2||_{1; k,\alpha; N'_i} + C_0 r^{-\mu}  \leq C.
\]
Hence, we obtain $||\mathfrak{e}||_{-\gamma; k,\beta; N_i \times D} \leq C ||z-z_0||_{-\gamma-1; k,\beta; N_i \times D} \leq C ||\Psi ||_{-\gamma; k+1,\alpha; X_\infty\times \partial D}$. 

Again, by Lemma \ref{ckalphaweiesticomposfun}, it follows that 
$$\displaystyle ||\Psi((z(\cdot,\cdot), \cdot))||_{-\gamma; k,\beta; N_i \times \partial D} \leq C  ||\Psi||_{-\gamma; k,\beta; N_i' \times \partial D}.$$
Combining with the above discussion on the estimates of $\mathfrak{e}$, we have
\begin{align*}
||\widehat{F}||_{-\gamma; k, \beta; N_i \times D} &\leq C ||\Psi(z(w,\tau), \tau) + \mathfrak{e}||_{-\gamma; k, \beta; N_i \times \partial D} + ||\mathfrak{e}||_{-\gamma; k, \beta; N_i \times D}\\
&\leq C ||\Psi||_{-\gamma; k+1,\alpha; X_\infty\times \partial D}.
\end{align*}
Together with the estimates on the coordinates transforms \eqref{esticoortran}, we complete the proof of \eqref{estiweisoluHCMA}. In conclusion, we complete the proof of the weighted estimates for the main theorem, Theorem \ref{thmweiestisolutiontoHCMA}.
\end{proof}

\subsection{On Local Regularity of Solutions to HCMA Equations}
In this subsection, we develop a local regularity result for solutions to the HCMA equation. Throughout, we consider a complete K\"ahler manifold \(X\), which may be either compact or non-compact.

The regularity theorem relies on pluripotential theory on the product space \(X \times D\), and it will be useful to recall the admissible class of \(\Omega\)-plurisubharmonic subsolutions with prescribed boundary data. Given a boundary function \(\Psi\) defined on \(X \times \partial D\), we define the admissible class of functions, \(\mathcal{B}_{\Psi, \Omega}\), as a subset of \(\Omega\)-plurisubharmonic subsolutions \(u\) on \(X \times D\) satisfying \( \limsup_{(x', \tau') \to (x, \tau)} u(x', \tau') \leq \Psi(x, \tau) \text{ for all } (x, \tau) \in X \times \partial D.
\)

Different choices of the admissible subset \(\mathcal{B}_{\Psi, \Omega}\) may be appropriate depending on the setting. Here, we take $\B_{\Omega, \Psi}$ to be the bounded functions as follows: 
\begin{align*}
     \B_{\Omega, \Psi} = \{u \in \psh_{\Omega}(X \times D);\  u \text{ is bounded, and } \limsup_{(x',\tau')\rightarrow (x, \tau)} u(x', \tau') \leq \Psi (x, \tau), \text{ on } X \times \partial D   \}
\end{align*}

\begin{thm} \label{thmlocalreg}
    Let $(X, \omega)$ be a K\"ahler manifold with a reference K\"ahler form $\omega$, and let $D \subseteq \CC$ denote the unit disk. Consider the following HCMA equation on $X \times D$ 
    \begin{align}\label{localregHCMA}
        \begin{cases}
            (\Omega + dd^c \Phi)^{n+1} = 0, \qquad &\text{ in } X \times D,
            \\
            \Phi = \Psi, & \text{ in } X \times \partial D,
            \\
            \Omega + dd^c \Phi \geq 0, \qquad &\text{ in } X \times D,
        \end{cases}
    \end{align}
    where $\Omega$ is a reference K\"ahler form on $X \times D$, defined as the pullback of $\omega$ via the standard projection $\pi: X \times D \rightarrow X$.
    Let $\Phi$ be a global solution to the above equation given as the upper envelope of the admissible class $\B_{\Psi, \Omega}$, and the global solution has a $\C^{0}$ bound, given by $||\Phi||_{L^\infty(X \times D)}\leq M$. 
    Assume there exists a holomorphic embedding $i: N' \rightarrow X$. Suppose that there is a local K\"ahler potential $\rho $ on $N'$ is uniformly convex in the sense that $D^2 \rho \geq \lambda I_{2n}$. Consider a pair $(N, N')$ with $N \subseteq N'$ biholomorphic to the concentric balls $(B_{r}, B_{r+R_0})$.
    If the boundary data satisfies the smallness condition locally on $N' \times D$: 
    \begin{align}\label{localregsmallcon}
    ||\Psi||_{k+2,\alpha; N' \times \partial D} \leq \varepsilon,
    \end{align}
    where $\varepsilon$ is a sufficiently small constant depending on $k$, $\alpha$ and $\lambda^{-1}$. Then, there exists a uniform constant $R_0>0$ depending only on $\varepsilon$, $\lambda$ and $ M$ such that for any $r >0$, the solution $\Phi$ is smooth on $N \times D$, and satisfies the estimate:
    \begin{align}\label{estilocal}
        ||\Phi||_{k,\beta; N \times D} \leq C ||\Psi||_{k+1, \alpha; N\times \partial D}.
    \end{align}
\end{thm}

\begin{proof} The proof follows by combining several results established in the preceding sections. By the smallness condition and Theorem~\ref{thmexistbypertb}, there exists a family of holomorphic discs $G: N\times D \rightarrow T^* N'$. If we write $G (w, \tau)= (z(w, \tau), \xi(w,\tau))$, where $w$ denote the standard holomorphic coordinates on $N'$ and $\xi$ denote the complex bundle coordinates of $T^* N'$, then we have the following displacement estimate:
\begin{align*}
    ||(z-z_0, \xi-\xi_0)||_{k, \beta; N \times D} \leq C ||\partial \Psi||_{k, \alpha; N' \times \partial D},
\end{align*}
where $G_0(w, \tau)= (z_0, \xi_0)(w, \tau) $ is the trivial foliation given by $(z_0, \xi_0) (w, \tau) = (w, \partial \rho (w))$. After projection via $\pi$, the family of holomorphic discs $G$ induces a holomorphic discs foliation on $N \times D$. Let $H(w, \tau) =p \circ G (w, \tau) $. Then, for each $x_0 \in N$, the corresponding leaf of foliation is given by $\mathcal{L}_{x_0} = \{H(x_0, \tau); ~\tau \in D\}$. 

A global subsolution $F \in \B_{\Omega, \Psi}$ can be constructed based on the data $(G, \Psi)$, following the construction developed in Section \ref{secconstructF}. Precisely, for each $x_0$, there exist a real pluriharmonic function $L_{x_0}$ (see Lemma \ref{lempshLcheck}) on $N' \times D$ defined as follows:
\begin{itemize}
    \item Along the leaf $\LL_{x_0}$, $L_{x_0}$ is harmonic in $\tau$ and matches the boundary data $\tilde{\psi}_\tau = \rho + \psi_\tau$.
    \item In the spatial direction of $N' \times D$, $L_{x_0}$ is linear in $z_p$, with the derivative in each $z_p$ given by $\xi_{p}$.
\end{itemize}
If we choose $\varepsilon$ in \eqref{localregsmallcon} small enough (here might depend on $n$, $\lambda$), we may guarantee the uniform convexity condition, $D_X^2 (\rho + \psi_\tau) \geq \frac{\lambda}{2} I_{2n} $, holds. Let $\widetilde{N}$ denote the intermediate concentric ball between $N$ and $ N'$, with the radius slightly less than $r + R_0$. According to the proof of Lemma \ref{lemsubsolucon}, for $x\in N$ and $z \in N'-\widetilde{N}$, we have
\begin{align*}
    L_x (z, \tau) - \rho (z, \tau) \leq -\frac{1}{3} \lambda d(z, x)^2.
\end{align*}
Now, choosing $R_0  \geq 2\sqrt{M/\lambda}$, it follows that $L_x (z, \tau) - \rho (z, \tau) \leq -M$ for $x \in N$ and $z \in N' - \widetilde{N}$.
Then, the global subsolution $F$ can be defined as follows: 
\begin{align*}
    F (z, \tau) = \max \big\{-M, \sup_{x \in N}( L_x- \rho) (z, \tau) \big\}.
\end{align*}
According to the proof of Theorem \ref{thmconstpshsolu} (general case), we have that $F \in \B_{\Omega, \Psi}$, and moreover, $F \equiv \Phi$ on $N \times D$.

It suffices to prove the estimate \eqref{estilocal}. The method follows the same strategy as in Section~\ref{secweiestisoltoHCMAend}. In particular, we solve Dirichlet problems on the family of holomorphic discs defined by the foliation, as follows:
\begin{align*}
    \begin{cases}
        \Delta_\tau u (w, \tau) = 0, \qquad & \text{ in } N \times D; \\
        u (w, \tau) = f(w, \tau), & \text{ in } N \times \partial D.
    \end{cases}
\end{align*}
where $f$ is given by $f(w, \tau) = \Psi (z(w, \tau), \tau) + \rho(z(w, \tau)) - \rho(w)$. Using the notation $\mathfrak{e} = \rho(z(w, \tau)) - \rho(w)$, we have $$||\mathfrak{e}||_{k,\beta; N \times D} \leq ||z-z_0||_{k,\beta; N \times D} \cdot ||\nabla \rho||_{k,\beta; N \times D} \leq  C ||\partial\Psi||_{k,\alpha; N' \times \partial D}.$$
By Lemma \ref{lemplaestiwrtbdry} and Lemma \ref{ckalphaesticomposfun}, we have
\begin{align*}
 ||u||_{k,\beta; N \times D} 
 &\leq C||f||_{k,\beta; N \times \partial D} \\
 & \leq C ||\Psi||_{k+1, \alpha; N' \times \partial D}.
 \end{align*}
Note that along the leaf $\LL_x$ for $x \in N$, $F(z(w, \tau), \tau) = u(w, \tau) - \mathfrak{e}(w, \tau)$ for $(w, \tau) \in N \times D$. Since $\Phi \equiv F$ in $N \times D$, we obtain
\begin{align*}
    ||\Phi||_{k,\beta; N \times D} \leq C ||\Psi||_{k+1, \alpha; N\times \partial D}.
\end{align*}
\end{proof}

As a direct application of Theorem~\ref{thmlocalreg}, we obtain the following regularity result for solutions to the HCMA equation on $X \times D$, for a compact K\"ahler manifold $X$, under a global smallness condition on the boundary data: $||\Psi||_{k+2,\alpha; N' \times \partial D} \leq \varepsilon$. This result is originally due to Donaldson~\cite{donaldson2002holo}, and a PDE-based proof was later given by Hu in~\cite{hu2024metric}; both arguments are global in nature. In contrast, the result presented here follows directly from the local regularity theorem, Theorem~\ref{thmlocalreg}, illustrating how global regularity can be deduced from local construction:

\begin{cor}
Let \((X, \omega)\) be a compact K\"ahler manifold with reference K\"ahler form \(\omega\), and let \(D \subseteq \mathbb{C}\) denote the unit disk. Consider the HCMA equation as in \eqref{localregHCMA}. Assume that the boundary data \(\Psi\) satisfies the smallness condition
\begin{align} \label{smallconcpt}
\|\Psi\|_{k+2,\alpha; X \times \partial D} \leq \varepsilon,
\end{align}
for some sufficiently small \(\varepsilon > 0\) (depending on \(k\), \(\alpha\), and the geometry of \(X\)).

Then there exists a smooth, nondegenerate solution \(\Phi\) to \eqref{localregHCMA}, satisfying the estimate
\[
\|\Phi\|_{k,\beta; X \times D} \leq C \|\Psi\|_{k+1, \alpha; X \times \partial D},
\]
for some constant \(C > 0\) depending on \(k\), \(\alpha\), \(\beta\), and the background data.
\end{cor}
\begin{proof}
    Note that there exists a finite collection of pairs of holomorphic balls in $X$, $\{(N_i, N'_i);~ N_i \subseteq N'_i \subseteq X,~  1\leq i \leq m\}$,  such that $\{N_i\}_{i=1}^m$ is a covering of $X$, and each pair, $(N_i. N'_i)$, is biholomorphic to a pair of concentric balls $(B_{r_i},B_{r_i+ R})$ in $\CC^n$. By the compactness of $X$, the constant $R$ can be chosen uniformly depending on the geometry of $X$. 
  Furthermore, we may assume that on each \(N_i'\), there exists a smooth K\"ahler potential \(\rho_i\) such that \(\omega|_{N_i'} = dd^c \rho_i\). Suppose \(\rho_i\) is uniformly convex in the sense that
\(
D^2 \rho_i \geq \lambda I_{2n}
\)
for some constant \(\lambda > 0\). By compactness, \(\lambda\) can be chosen uniformly for all \(1 \leq i \leq m\), depending only on the geometry of \((X, \omega)\).
    
    Let $\B_{\Omega, \Psi}$ denote the class of bounded continuous $\Omega$-psh subsolution on $X \times D$ to the HCMA equation with boundary data $\Psi$. 
    By Remark \ref{remgsolHCMAupenvcptcase}, the solution to HCMA equation with boundary data $\Psi$ is unique and is given by the upper envelope of the class $\B_{\Omega, \Psi}$. 
    Let \(\Phi\) be the unique solution to the HCMA equation. Then the global \(L^\infty\) bound of \(\Phi\) satisfies
\[
\|\Phi\|_{L^{\infty}(X \times D)} \leq \|\Psi\|_{L^{\infty}(X \times \partial D)}.
\]
In particular, under the smallness condition on the boundary data \eqref{smallconcpt}, we have \(\|\Phi\|_{L^{\infty}(X \times D)} \leq \varepsilon\).
We now apply Theorem~\ref{thmlocalreg} to each pair of local holomorphic balls \((N_i, N_i')\). Suppose that \(\varepsilon\) is sufficiently small such that
\[
\varepsilon \leq \lambda \left(\frac{R}{2}\right)^2.
\]
Then, following the construction in the proof of Theorem~\ref{thmlocalreg}, for each \(1 \leq i \leq m\), there exists a global subsolution \(F_i\) defined on \(X \times D\) such that \(F_i = \Phi\) on \(N_i \times D\).
The smoothness of \(\Phi\) and the estimate in the statement follow directly from Theorem~\ref{thmlocalreg}, which completes the proof.
\end{proof}

\bibliographystyle{amsplain}
\bibliography{reference}

\end{document}